\RequirePackage{fix-cm}
\documentclass{svjour3}   
\smartqed
\usepackage{graphicx}
\usepackage{verbatim}   
\usepackage{color}      
\usepackage{hyperref}   
\usepackage{amsmath}
\usepackage{setspace}
\usepackage[noend]{algpseudocode}
\usepackage{mathtools}
\usepackage{multirow}
\usepackage{amsfonts}
\usepackage{amssymb}
\usepackage[pdftex,dvipsnames]{xcolor}  
\usepackage{subfigure}  
\usepackage{enumerate}   
\usepackage{xargs}
\usepackage{epstopdf}

\newcommand{\argmin}{\arg\!\min}
\newcommand{\argmax}{\arg\!\max}
\newcommand{\eqnum}{\refstepcounter{equation}\textup{\tagform@{\theequation}}}

\newcommand\abs[1]{\left|#1\right|}

\makeatletter
\def\namedlabel#1#2{\begingroup
    #2%
    \def\@currentlabel{#2}%
    \phantomsection\label{#1}\endgroup
}

\usepackage{xparse}

\usepackage{enumitem} 
\usepackage[ruled,vlined,linesnumbered]{algorithm2e} 
\usepackage{tabularx}
\usepackage{booktabs}

\DeclareMathOperator{\Argmin}{Argmin}
\DeclareMathOperator{\dist}{dist}
\DeclareMathOperator{\into}{int}

\DeclareMathOperator{\dom}{dom}

\DeclareMathOperator{\prox}{Prox}

\newcommand{\floor}[1]{\left\lfloor #1 \right\rfloor}
\newcommand{\ceil}[1]{\left\lceil #1 \right\rceil}

\newcommand{\Hs}{{H}}

\newcommand{\Ws}{{W}}

\newcommand{\R}{{\mathbb{R}}}
\newcommand{\N}{{\mathbb{N}}}

\newcommand{\cost}{{\Psi}} 			
\newcommand{\costr}{{\mathcal{F}}} 	
\newcommand{\costp}{{\mathcal{R}}} 	
\newcommand{\costbar}{{\bar{\Psi}}} 

\newcommand{\costrD}{{ \nabla \costr}} 	
\newcommand{\costrDD}{{\costr''}} 	

\newcommand{\LipcostrD}{{L_{\costr'}}} 

\newcommand{\crit}{{S_{*}}} 

\newcommand{\domF}{{\into(\dom \costr)}}

\newcommand{\pgrad}{{\mathcal{T}}} 	
\newcommand{\gmap}{{\mathcal{G}}}	

\newcommand{\ssize}{{\alpha}} 						
\newcommand{\ssizek}{\alpha_k} 					
\newcommand{\ssizeinit}{{\alpha_{\mathrm{int},k}}}	
\newcommand{\ssizelb}{\ssize_{\inf}}				
\newcommand{\ssizeub}{\ssize_{\sup}}				
\newcommand{\ssizebar}{\overline{\ssize}}			
\newcommand{\ssizeBB}{\alpha^{\mathrm{BB}}} 		
\newcommand{\ssizeBBoneOld}{\alpha^{\mathrm{BB}1\mathrm{a}}} 		
\newcommand{\ssizeBBoneNew}{\alpha^{\mathrm{BB}1\mathrm{b}}} 		
\newcommand{\ssizeBBtwoOld}{\alpha^{\mathrm{BB}2\mathrm{a}}}		
\newcommand{\ssizeBBtwoNew}{\alpha^{\mathrm{BB}2\mathrm{b}}} 		
\newcommand{\ssizeABBOld}{\alpha^{\mathrm{ABBa}}}	
\newcommand{\ssizeABBNew}{\alpha^{\mathrm{ABBb}}}	
\newcommand{\tol}{\varepsilon_{\mathrm{tol}}} 	
\newcommand{\mmax}{{m_{\mathrm{max}}}}

\newcommand{\ubar}{{\bar{u}}}

\newcommand{\maxfeval}{k_{\max}^f}
\newcommand{\maxgeval}{k_{\max}^g}
\newcommand{\maxfparam}{\gamma_{\mathrm{comp}}^f}
\newcommand{\maxgparam}{\gamma_{\mathrm{comp}}^g}
\newcommand{\decparam}{\gamma_{\mathrm{decr}}}

\newcommand{\quadset}{\Gamma}
\newcommand{\quadconst}{\gamma_{\cost, \quadset}}

\newcommand{\yd}{{y_d}}
\newcommand{\ulb}{{u_a}}
\newcommand{\uub}{{u_b}}
\newcommand{\PDEcon}{{\mathcal{E}}}

\newcommand{\costyu}{{\mathcal{J}}}

\makeatletter
\newcommand{\rom}[1]{(\romannumeral #1)}
\makeatother

\begin{document}

\title{On the nonmonotone linesearch for a class of infinite-dimensional
nonsmooth problems}

\titlerunning{On the nonmonotone linesearch for a class of infinite-dimensional
nonsmooth problems}

\author{Behzad Azmi  \and  Marco Bernreuther}

\institute{Behzad Azmi,  Corresponding author \at
             University of Konstanz, Department of Mathematics and Statistics, Chair for Numerical Optimization, Universitätsstraße 10, 78457 Konstanz, Germany  \\
              behzad.azmi@uni-konstanz.de
           \and
            Marco Bernreuther   \at
              University of Konstanz, Department of Mathematics and Statistics, Chair for Numerical Optimization, Universitätsstraße 10, 78457 Konstanz, Germany  \\
              marco.bernreuther@uni-konstanz.de
}

\date{Received: date / Accepted: date}

\maketitle

\begin{abstract}
This paper provides a comprehensive study of  the nonmonotone forward-backward splitting (FBS)  method for solving a class of nonsmooth composite problems in Hilbert spaces.   The objective function is the sum of a Fréchet differentiable (not necessarily convex) function and a proper lower semicontinuous convex (not necessarily smooth) function.  These problems appear, for example,  frequently in the context of optimal control of nonlinear partial differential equations (PDEs) with nonsmooth sparsity promoting cost functionals.  We discuss the convergence and complexity of FBS equipped with the nonmonotone linesearch under different conditions.  In particular,  R-linear convergence will be derived under quadratic growth-type conditions.  We also investigate the applicability of the algorithm to problems governed by PDEs.  Numerical experiments are also given that justify our theoretical findings. 
 \end{abstract}

\section{Introduction}
\label{sec:introduction}
In this work,  we are concerned with the following composite problem
\begin{equation}
\label{eq:opt_problem}
\tag{\textbf O}
\begin{aligned}
& \min_{u \in \Hs} \hspace{1.4mm}\cost(u) \coloneqq \costr(u) + \costp(u),
\end{aligned}
\end{equation}
where  $\Hs$ is a real Hilbert space,  $\costr$ is a continuously Fr\'echet differentiable  function (possibly nonconvex),   and $\costp$ is a convex function whose proximal operator is assumed to be explicitly computable.  The precise details will be introduced in Section \ref{sec:problem_algorithm}.  Problems of the form \eqref{eq:opt_problem} appear in several fields of application such as optimal control  problems \cite{AzmKunRod,6422363},   system identification , signal and image processing \cite{MR2858838},   machine learning,  and statistics \cite{parikh2014proximal}.  

Arguably one of the most well-known algorithms for solving problem \eqref{eq:opt_problem}  is forward-backward splitting (FBS),  also known as the proximal gradient method \cite{B17,MR2858838},  which is the generalization of the classical gradient method for problems with an additional nonsmooth term (see  \eqref{eq:update_v3}).  The iterations of FBS are defined by
\begin{align}
\label{eq:update_v1}
& u_{k + 1} = \prox_{\frac{1}{\ssizek} \costp} \left(u_k - \frac{1}{\ssizek} \costrD(u_k)\right), && \text{for } k \in \N_0,
\end{align}
where the step-sizes $\ssizek >0 $ are supposed to be choosen in a way that guarantees convergence of the algorithm and accelerates it.   It is known that global convergence of FBS is to be sublinear of order $(1 / k)$ for the convex case \cite{B17},  where $k$ stands for the number of iterations.  This order can be improved to $(1 / k^2)$ using an inertial variant of
the algorithm based on Nesterov's accelerated techniques \cite{zbMATH05618078,zbMATH06197840}.  Convergence of the iterates of FBS to a critical point of problem \eqref{eq:opt_problem},  even for the nonconvex case,  has been shown for  functions $\cost$ satisfying the Kurdyka--Łojasiewicz property e.g., \cite{zbMATH05382665,zbMATH06145973,MR3707370,MR3341671} or quadratic growth error bounds  e.g., \cite{MR4215308,DruLew_Error,garrigos_convergence_2022,necoara_linear_2019}.  Due to the simplicity and efficiency of FBS, the convergence, complexity, and applicability of this approach have been extensively studied in the literature,  and it is still considered an active research area.  Many different step-size strategies have been proposed and studied in the context of FBS under different assumptions. Among them, we mention  \cite{MR4218411,MR3500980,MR3482398,MR3429743,MR3845278}  for the finite-dimensional setting  and  \cite{MR3616647,MR4215308,MR4430995,MR3707899}  for the infinite-dimensional setting.

As is well-known for smooth problems,  the nonmonotone linesearch appears numerically efficient in situations where the monotone scheme is forced to propagate along the bottom of a narrow curved valley.  Further,  due to the fact that it permits growth in function values,  it can be combined in an efficient manner with spectral gradient methods such as Barzilai-Borwein (BB) step-sizes \cite{azmi_analysis_2020,MR967848},  which inherently have a nonmonotone behavior.   Based on the nonmonotone approach developed by Grippo et al \cite{MR849278},  Wright et al \cite{MR2650165} proposed a nonmontone FBS for convex composite optimization problem posed in $\mathbb{R}^n$.  The convergence of this approach known as SparSa is further studied in \cite{MR2792408}  and the order $(1 / k)$  and R-linear convergence have been proven for convex and strongly convex functions,  respectively.   Very recently,  in  \cite{MR4504980}, the authors managed to prove  convergence to a critical point of this scheme for a finite-dimensional nonconvex composite function under milder conditions,  i.e.,  local Lipschitz continuity of the gradient for the smooth part and uniform continuity for the objective function on its level sets.

From a computational point of view,   both classic gradient methods and FBS can be accelerated provided that they are combined with the BB step-sizes for approximating the curvature of the objective function and its smooth part,  respectively.   In particular,  the BB step-sizes appear to be quite efficient in the context of PDE-constrained optimization.  In case of a strongly quadratic smooth part, R-linear convergence is established in case of a sufficiently small spectral condition number $(\leq 2)$ of the quadratic operator,  see \cite[Remark 3.4]{azmi_analysis_2020}.   However,  even in case of a quadratic smooth part with condition numbers larger than $2$, convergence of the BB step-sizes is not clear and, to guarantee it, one needs to combine the BB step-sizes with a nonmonotone linesearch, which is relying on function evaluations.

Weak and strong convergence can be distinguished only in the infinite-dimensional setting.  In practice, a numerical solution of an infinite-dimensional problem is of course obtained by the numerical implementation of algorithms for finite-dimensional approximations of these problems.  However, to ensure the numerical robustness and stability of the finite-dimensional approximations, it is of great interest to analyze the convergence properties of the algorithm in the infinite-dimensional setting. These are expressed by the so-called mesh-independent principle (MIP),  see e.g.,  \cite{MR821912,Behzad2021,MR1202003,MR2085262,MR912453,MR1049770,MR1756894}.  Roughly speaking,  based on the infinite-dimensional convergence results, this principle gives predictions about the convergence properties of finite-dimensional approximations (discretized problems).  MIP also provides a theoretical justification for the development of refinement strategies,  see e.g., \cite{MR917455}.   

In light of the above discussion, and despite the numerical efficiency of nonmonotone FBS for problems with PDEs, to the best of our knowledge this approach has not yet been studied for infinite-dimensional problems.   In this paper,  we take a step towards this direction and address convergence and  complexity of the nonmonotone FBS for problem \eqref{eq:opt_problem}.

\subsection*{Contribution}
\label{sec:introduction:contribution}
More precisely,  the contributions of this work can be summarized as follows: \rom{1} Starting with the nonconvex case,  we prove well-posedness of the algorithm even without Lipschitz continuity of the gradient  of $\costr$.  Under the global Lipschitz continuity of the gradient  of $\costr$,   we prove global convergence of the algorithm with complexity $(1 / \sqrt{k})$.  To be more precise,  we show that the norm of the prox-gradient mapping of iterates vanishes with complexity   $(1 / \sqrt{k})$.  \rom{2} We also establish that every weak sequential cluster point of iterates is a stationary point.  \rom{3} We derive the worst-case evaluation complexity of finding an $\tol$-stationary point.  More precisely,   we give estimates on the maximal number of function and  prox-grad  operator evaluations for computing an approximate stationary point with a user-defined accuracy threshold $\tol>0$.  \rom{4}  In the convex setting,  relying on the concept of quasi-Fejer sequences,  we are able to extend previously established results to global convergence,   both in terms of  function values  and iterates with respect to the sequential weak topology.   Further,  we show that the convergence is sublinear of order $(1 / k)$ in function values.  \rom{5}  Under quadratic growth-type conditions,  we show global R-linear convergence, both in terms of function values and iterates.   The proof of the latter is more delicate for infinite-dimensional problems since the transition from weak sequential convergence to strong convergence is not straightforward.   \rom{6}  Finally,  aiming at  optimization problems governed by PDEs,  we discuss the validity of the convergence results of  the algorithm without strong Lipschitz continuity of $\costrD$.  Our theoretical framework is supported by two nonsmooth problems governed by PDEs,  including semilinear elliptic and parabolic equations.  Finally, we show that our results are applicable to these problems and report on numerical findings.

\subsection*{Outline of the paper}
The rest of this paper is organized as follows: Section \ref{sec:problem_algorithm} presents the preliminaries,  assumptions on the optimization problem \eqref{eq:opt_problem},  the algorithm,  and the nonmonotone linesearch strategy.    Section \ref{sec:convergence_complexity} investigates the convergence and complexity of  the algorithm comprehensively under different conditions.  Section \ref{sec:PDE}  discusses the applicability of the results from the previous section to PDE-constrained optimization problems.  Finally,  numerical experiments are reported in Section \ref{sec:numerics} that justify our theoretical findings.  To improve the readability of the paper,  we provide proofs to some results from Section \ref{sec:convergence_complexity} in Appendix \ref{Appendix}.

\subsection*{Notation}
 Throughout this paper,   the Hilbert space $\Hs$ is endowed with the scalar product $(\cdot, \cdot)_\Hs$ and the induced norm $\|\cdot\|_\Hs$.  For a radius $r>0$ and  $\bar{u} \in \Hs$,  we define $\mathbf{B}_{r} (\bar{u}):=\{ u \in \Hs :  \| u-\bar{u} \|_{\Hs} <r \}$.  We also denote with $P_{S}: \Hs \to \Hs$,  the orthogonal projection onto the set $S \subset \Hs$.  Further for every  $\tilde{\cost} \in \mathbb{R}$,  we define  $\left[\cost \leq  \tilde{\cost}  \right]:=\{ u\in \Hs: \cost(u)\leq \tilde{\cost} \}$. By $\Argmin \cost$ we denote the set of minimizers of $\cost$. Sometimes we will use set-values (in-)equalities, i.e. $A \leq B$ for $A, B \subset \Hs$, meaning that $a \leq b$ for all $a \in A$ and $b \in B$.
   
We call a sequence $\{ u_k \}_k  \subset \Hs$  quasi-F\'ejer monotone with respect to a non-empty set $S \subset \Hs$,  if for every $v \in S$ it holds that 
\begin{equation*}
\|u_{k+1}-v\|^2_{\Hs} \leq  \|u_{k}-v\|^2_{\Hs} + \epsilon_k,      
\end{equation*} 
where $ \{\epsilon_k\}_k \subset \mathbb{R}_{\geq 0}$ is a summable sequence,  i.e.,  $\sum^{\infty}_{k =0} \epsilon_k <\infty $.   
To avoid confusion in the notation,  the derivative and gradient of  $\costr$ are  defined  by  $\costr' \colon \Hs \rightarrow \Hs'$ and  $\nabla \costr \colon \Hs \rightarrow \Hs$, respectively.  In this case,  for every $u \in \Hs$  we identify  $\nabla \costr(u) \in \Hs$  with the unique Riesz representative of $\costr' \in \Hs'$.

\section{Problem formulation and algorithm}
\label{sec:problem_algorithm}
In this section the precise theoretical framework and algorithm will be introduced.   First,  we review some notions of stationarity for nonconvex and nonsmooth functions. The \text{Fr\'echet subdifferential} of $ \phi: \Hs \to \R \cup \{\pm \infty\}$ at $u \in \dom \phi$ is defined as 
\begin{equation*}
 \hat{\partial} \phi(u):=\left\{  w \in \Hs  :  \liminf_{v\to u}  \frac{\phi(v)-\phi(u)-(w,v-u)_\Hs}{\|v-u\|_\Hs} \geq 0  \right\}. 
\end{equation*}    
For this notion of subdifferential,  we obtain the Fermat principle \cite[Proposition 9.1.5]{zbMATH05162727}  for the local minimizers stating the following.
\begin{proposition}
\label{Pro: Fermat}
Let $ \phi: \Hs \to \R \cup \{\pm \infty\}$ be proper and lower semicontinuous  and $u^* \in \dom \phi$ be a local minimizer.  Then $0 \in \hat{\partial} \phi (u^*)$. 
\end{proposition} 
The Fr\'echet subdifferential  fails to be outer semicontinuous,  which is also not desirable.  To mitigate this issue,  we mention the \textit{Mordukhovich} or \textit{ limiting  subdifferential } of  $\phi$ at  $ u  \in \dom \phi$, which is defined  as the sequential Painlev\'e-Kuratowski outer limit  of $\hat{\partial} \phi (u)$ \cite[Theorem 2.34]{zbMATH02176481}, that is
\begin{equation*}
\partial_M \phi(u):= \left\{ w \in \Hs : \exists \{ u_n\}_n ,  \{ w_n\}_n \subset \Hs \text{ s.t.  } u_n \xrightarrow{\phi} u , w_n \rightharpoonup w  \text{ with } w_n \in \hat{\partial} \phi(u_n) \right\}.
\end{equation*}
Here $u_n \xrightarrow{\phi} u$ stands for  $u_n \to u$ with $\phi(u_n) \to \phi(u)$.  We can infer directly from the definition  that $\hat{\partial} \phi(u) \subset   \partial_M \phi(u)$.  Here,  we summarize from \cite{zbMATH02176481},   some of the results for these subdifferentials
\begin{itemize}
\item  If  $\phi$  is convex, the above notions of subdifferential coincide with the convex subdifferential. That is $\hat{\partial} \phi(u)  =  \partial_M \phi(u)  = \partial \phi(u)$ with $\partial \phi(u):=\{ w \in \Hs :  \phi(v)-\phi(u) \geq (w,v-u)_\Hs   \text{ for all  } v \in \Hs \}$.

\item  For every continuously Fr\'echet differentiable function $\phi$,   the above notions of subdifferential reduce to the first derivative, i.e.  $\hat{\partial} \phi(u)  =  \partial_M \phi(u)  = \{ \nabla \phi(u)\}$.

\item  Given an arbitrary function $\psi: \Hs \to \R \cup \{\pm \infty\}$, which is finite at $u$ and a function $\phi: \Hs \to \R \cup \{\pm \infty\}$, that is continuously Fr\'echet differentiable at $u$, we have
\begin{equation}
\label{eq:subdiff_rule}
\hat{\partial} (\phi +\psi)(u) =  \nabla\phi (u)+ \hat{\partial}\psi(u)   \text{ and } \partial_M (\phi +\psi)(u) =   \nabla \phi (u)+ \partial_M \psi(u).
\end{equation}
\end{itemize} 
 Therefore,  according to Assumption \ref{ass:general} and Proposition \eqref{Pro: Fermat},  the Fermat principle for \eqref{eq:opt_problem} reads as follows:  If $u^*\in \Hs$  is a local minimizer of  $\cost$, then 
 \begin{equation}
\label{stationary}
 -\costrD(u^*) \in \partial \costp(u^*), 
 \end{equation}
and,  with similar arguments as in the proof  \cite[Theorem 26.2]{MR3616647},  \eqref{stationary} can be equivalently expressed as 
\begin{equation}
\label{stationary2}
u^* = \prox_{\frac{1}{\ssize} \costp}(u^*-\frac{1}{\ssize}\costrD(u^*)) \quad  \text{ for some }  \ssize >0.
\end{equation}
where the proximal operator  $\prox_{\frac{1}{\ssize} \costp}: \Hs \to \Hs$ is defined by
\begin{equation*}
\prox_{\frac{1}{\ssize} \costp} (u) :=  \argmin_{v\in \Hs}\left(  \costp(v)+ \frac{\ssize}{ 2}\|v-u \|_\Hs^2 \right).
\end{equation*}  
This operator is well-defined provided that  $\costp$ is proper,  convex,  and lower semicontinuous.  We also define the set of critical points by
\begin{equation*}
\crit := \{  u \in \Hs :  -\costrD(u) \in \partial \costp(u) \}. 
\end{equation*}
Now we are in the position that we can state the precise assumptions on \eqref{eq:opt_problem}.

\begin{assumption} For  problem \eqref{eq:opt_problem}, 
\label{ass:general} 
\begin{description}
\item[\namedlabel{ass:general_1}{\textbf{A1}}:] $\costp \colon \Hs \rightarrow \R \cup \{\pm \infty\}$ is  proper, convex,  and lower semicontinuous. 
\item[\namedlabel{ass:general_2}{\textbf{A2}}:] $\costr \colon \Hs \rightarrow \R$ is continuously Fr\'echet differentiable on $\domF$ containing  $\dom \costp$, that is,  $\dom \costp \subseteq \domF$. 
\item[\namedlabel{ass:general_3}{\textbf{A3}}:]  $\costrD \colon \Hs \rightarrow \Hs$  is globally $\LipcostrD$-Lipschitz continuous.
\item[\namedlabel{ass:general_4}{\textbf{A4}}:] The gradient  $\costrD \colon \domF \rightarrow \Hs$ is weak-to-strong sequentially continuous.
\end{description}
\end{assumption}
Conditions  \ref{ass:general_1}-\ref{ass:general_2} are standard in order to guarantee  the well-posedness of the algorithm.  \ref{ass:general_3} will be used to derive complexity results for iterations and we will discuss the relaxation of this condition for a large class of problems governed by PDEs.  We use \ref{ass:general_4} to show that  every weak sequential  accumulation point of iterates belongs to the set $\crit$.  Note that, if $ \dim(\Hs) < \infty$,  \ref{ass:general_4} follows from \ref{ass:general_2}.
 
Now we turn our attention towards formalizing \eqref{eq:update_v1} and the proposed step-size update strategy by a nonmonotone linesearch method. Therefore we introduce the prox-grad operator and the gradient mapping analogously to \cite{B17,DruLew_Error} in the Hilbert space setting.
\begin{definition}
\label{def:prox_gmap} 
For every $\alpha \in \mathbb{R}_{>0}$,   we define
\begin{enumerate}
\item  the  \textit{prox-grad} operator $\pgrad_\ssize \colon \domF \rightarrow  \dom \costp$ with $u \mapsto \prox_{\frac{1}{\ssize} \costp}(u - \frac{1}{\ssize} \costrD(u))$.  
\item the \textit{gradient mapping}  $\gmap_\ssize  \colon \domF \rightarrow \Hs$  with  $u \mapsto \ssize (u - \pgrad_\ssize(u))$.
\end{enumerate}
\end{definition}
Due to Definition \ref{def:prox_gmap} and by some simple computations,  we obtain for every $u \in \domF$ that 
\begin{equation}
\label{eq:prox}
\gmap_\ssize(u)-\costrD(u) \in \partial \costp( \pgrad_\ssize(u)). 
\end{equation}
Further,  if we define for every  $u \in \domF$,  $ w \in \Hs$,  and $\ssize  \in \mathbb{R}_{>0}$  
\begin{equation*}
\mathcal{Q}_{\ssize}(w,u):= \costr(u)+ (\costrD(u), w-u)_\Hs + \frac{\ssize}{2}\|w-u\|_\Hs^2 +  \costp(w),
\end{equation*}
then $\pgrad_{\ssize}(u)$ is the unique minimizer of $\mathcal{Q}_{\ssize}(\cdot, u)$, i.e.  
\begin{equation*}
 \pgrad_{\ssize}(u)= \argmin_{w \in \Hs} \mathcal{Q}_{\ssize}(w,u). 
\end{equation*} 
As a consequence,  we can write 
\begin{equation}
\label{eq:B59}
(\costrD(u), \pgrad_{\ssize}(u)-u)_\Hs+\frac{1}{2\ssize}\|\gmap_\ssize(u)\|_\Hs^2 +  \costp(\pgrad_{\ssize}(u)) \leq   \costp(u). 
\end{equation} 
By means of the prox-grad operator and the gradient mapping \eqref{eq:update_v1} can be reformulated as follows:
\begin{align}
\label{eq:update_v2}
u_{k+1} & = \pgrad_{\ssizek}(u_k), && \text{for } k \in \N_0,\\
\label{eq:update_v3}
u_{k+1} & = u_k - \frac{1}{\ssizek} \gmap_{\ssizek}(u_k), && \text{for } k \in \N_0.
\end{align}
In this case,   for  every $u_0 \in \domF$,  the iterations are well-defined,  that is $\{ u_k \}_{k} \subset \dom \costp \subseteq \Hs$. Further,  due to \eqref{eq:update_v3} we can see  the analogy between the proximal gradient method and classical gradient descent for smooth minimization.  We also can characterize the critical points of  \eqref{eq:opt_problem},  using the notion of the gradient mapping  and \eqref{stationary2}  as follows.
 \begin{proposition}
\label{prop:stationary_point}
For $u^* \in \domF$, it holds that $\gmap_\ssize(u^*) = 0$ for some $\ssize \in \mathbb{R}_{>0}$ if and only if $u^* \in \crit$.
\end{proposition}
Therefore   $\gmap_\ssize(\ubar) = 0$ defines a natural termination condition analogously to the smooth case.

There are many possible choices for the step-size $\ssizek$  in our iterative scheme. In this work we will consider step-size updates by a BB-type update rule or a combination  of them with a nonmonotone linesearch.  To be more precise,  as the first trial step-size within the nonmonotone linesearch,  we will choose one  of the  spectral gradient BB-types step-sizes defined by
\begin{align*}
\ssizeBBoneOld_k & \coloneqq \frac{ {(u_k - u_{k-1}, \costrD(u_k) - \costrD(u_{k-1}))}_\Hs}{{(u_k - u_{k-1}, u_k - u_{k-1})}_\Hs},\\ 
\ssizeBBtwoOld_k & \coloneqq \frac{{(\costrD(u_k) - \costrD(u_{k-1}), \costrD(u_k) - \costrD(u_{k-1}))}_\Hs}{{(u_k - u_{k-1}, \costrD(u_k) - \costrD(u_{k-1})}_\Hs},\\
\ssizeBBoneNew_k & \coloneqq \frac{{(u_k - u_{k-1}, \gmap_{\ssize_{k-1}}(u_k) - \gmap_{\ssize_{k-1}}(u_{k-1}))}_\Hs}{{(u_k - u_{k-1}, u_k - u_{k-1})}_\Hs},\\ 
\ssizeBBtwoNew_k & \coloneqq \frac{{(\gmap_{\ssize_{k-1}}(u_k) - \gmap_{\ssize_{k-1}}(u_{k-1}), \gmap_{\ssize_{k-1}}(u_k) - \gmap_{\ssize_{k-1}}(u_{k-1}))}_\Hs}{{(u_k - u_{k-1}, \gmap_{\ssize_{k-1}}(u_k) - \gmap_{\ssize_{k-1}}(u_{k-1})}_\Hs}.
\end{align*}
The first two strategies correspond to the BB-method presented e.g. in \cite{Behzad2021}. The last two novel BB-methods modify the classical BB-method and try to incorporate full first order information by using the gradient mapping and not only the gradient of $\costr$. Furthermore we will use so called alternating BB-type update rules given by
\begin{align*}
\ssizeABBOld & \coloneqq \ssizeBBoneOld \text{ for } k \text{ even and } \ssizeBBtwoOld \text{ for } k \text{ odd},\\
\ssizeABBNew & \coloneqq \ssizeBBoneNew \text{ for } k \text{ even and } \ssizeBBtwoNew \text{ for } k \text{ odd}.
\end{align*}
\begin{remark}
In the case of $\costp = 0$ , we have $\gmap_{l}(u) = \costrD(u)$ for every $l>0$ and $u \in H$.  Therefore,  many of the introduced step-size updates above are identical, i.e. $\ssizeBBoneNew = \ssizeBBoneOld, \ssizeBBtwoNew = \ssizeBBtwoOld$ and $\ssizeABBNew = \ssizeABBOld$.
\end{remark}
For the nonmonotone linesearch update,  we choose
\begin{align}
\label{eq:ls_ineq}
\cost(u_{k+1}) \leq \max\limits_{0 \leq j \leq m(k)} \cost(u_{k - j}) - \frac{\delta}{\ssizek} \|\gmap_{\ssizek}(u_k)\|_\Hs^2,
\end{align}
with $0 < \delta < 1$ and memory $m \colon \N_0 \rightarrow \N_0$ satisfying
\begin{align}
\label{eq:ls_memory}
& m(0) = 0, && m(k) = \min\{m(k-1) + 1, \mmax\} && \text{for } k \in \N,
\end{align}
with upper bound $\mmax \in \N_0$.   Similarly to \cite{MR1915930},  we also define the functions $\ell:\mathbb{N}_0 \to \mathbb{N}_0$ and  $\nu:\mathbb{N}_0 \to \mathbb{N}_0$ with 
\begin{align*}
\ell(k)& := k -  \argmax\limits_{0 \leq j \leq m(k)}  \cost(u_{k - j})     \quad \text{ for }  k\geq 0,   \text{ with }   k- m(k) \leq \ell(k) \leq k, \\
\nu(k)& := \ell(k \mmax+k),       \quad \text{ for }  k\geq 0.
\end{align*}
Thus,  by these notations,  we have $\max\limits_{0 \leq j \leq m(k)} \cost(u_{k - j}) = \cost(u_{\ell(k)})$  and  \eqref{eq:ls_ineq} can be rewritten as 
\begin{equation}
\label{eq:B55b}
\cost(u_{k+1}) \leq \cost(u_{\ell(k)}) - \frac{\delta}{\ssizek} \|\gmap_{\ssizek}(u_k)\|_\Hs^2.
\end{equation}
Further we set
\begin{align}
\label{eq:ls_update}
\ssizek = \ssizeinit \eta^{i_k},
\end{align}
where $\eta > 1$. The initial step-size $\ssizeinit > 0$ is chosen by the BB-method and a lower bound $\ssizelb > 0$ is given to ensure nonnegativity. To limit the initial step-size numerically, also an upper bound $\ssizeub > \ssizelb$ is employed.  In the update we choose the smallest integer $i_k \in \N_0$ in \eqref{eq:ls_update}, such that \eqref{eq:ls_ineq} is satisfied. This procedure is summarized in Algorithm \ref{algo:NMLS}.
\newline
\smallskip
\IncMargin{1em}
	\begin{algorithm}[H]
		\SetKwData{Left}{left}\SetKwData{This}{this}\SetKwData{Up}{up}
		\SetKwFunction{Union}{Union}\SetKwFunction{FindCompress}{FindCompress}
		\SetKwInOut{Input}{Require}\SetKwInOut{Output}{Return}
		\Input{ $0 < \delta < 1$,  $\mmax \in \mathbb{N}_0$,  $\eta>1$,  $ \ssizeub > \ssizelb > 0$, $u_0 \in \Hs$ and $\ssize_0 > 0$}
		\Output{Approximation $u^* \in \Hs$ of stationary point of \eqref{eq:opt_problem}}
	       Set  $k =0$ \; 
		\While{$\|\gmap_{\ssizek}(u_k)\|_\Hs > 0$}{
			  Compute $\ssizeBB_k$ according to a BB-method and set $\ssizeinit := \max \{\ssizelb, \min \{\ssizeub,  \ssizeBB_k  \} \}$\; 		
			Set  $\ssizek = \ssizeinit \eta^{i_k}$,  where $i_k \geq 0$ is the smallest integer for which  \eqref{eq:ls_ineq} holds  \;
			Set $u_{k+1} = u_k - \frac{1}{\ssizek} \gmap_{\ssizek}(u_k)$ and $k=k+1$ \;
			}
		\caption{ Nonmonotone Splitting Algorithm}\label{algo:NMLS}
	\end{algorithm}
\smallskip
\noindent
Note that by Proposition \ref{prop:stationary_point} the criterion $\|\gmap_{\ssizek}(u_k)\|_\Hs \leq \tol$ with some tolerance $\tol > 0$ is a reasonable stopping criterion in the numerical realization of Algorithm \ref{algo:NMLS}.

\section{Convergence and complexity analysis}
\label{sec:convergence_complexity}
In this main section of our work, we present a detailed convergence and complexity analysis for Algorithm  \ref{algo:NMLS} under different assumptions and conditions.

\subsection{General case}
\label{General case}
We start by summarizing useful properties of the gradient mapping.  
\begin{lemma}
\label{lem:properties_of_Grad}
Suppose that \ref{ass:general_1}-\ref{ass:general_2}  hold, then we have the following properties:
\begin{description}
 \item[\namedlabel{P1}{P1}:]  For every $ l_1 \geq l_2 >0 $  and  $u \in \domF$ it holds that  $\frac{1}{l_1} \|\gmap_{l_1}(u)\|_\Hs \leq \frac{1}{l_2} \|\gmap_{l_2}(u)\|_\Hs$. 
\item[\namedlabel{P2}{P2}:]  For every   $ l_1 \geq l_2 >0 $  and    $u \in \domF$ it holds that  $\|\gmap_{l_1}(u)\|_\Hs \geq  \|\gmap_{l_2}(u)\|_\Hs$.
\item[\namedlabel{P3}{P3}:]  Assume in addition  that \ref{ass:general_3} holds,  then 
the gradient mapping is Lipschitz continuous,  that is 
\begin{equation*}
\|\gmap_{l}(u)- \gmap_{l}(v) \|_\Hs \leq (2 l+\LipcostrD)\|u-v\|_\Hs,
\end{equation*}
for every  $ l  >0 $  and  $v, u \in \domF$.
\end{description}
\end{lemma}
\begin{proof}
The proof follows by similar arguments as e.g. given in \cite[Theorem 10.9, Lemma 10.10]{B17} for finite-dimensional problems.
 \end{proof}

Furthermore the well-known sufficient decrease condition can be formulated with respect to the gradient mapping analogously to the finite-dimensional case presented in \cite[Lemma 10.4]{B17}.
\begin{lemma}[Sufficient Decrease Lemma]\label{lem:suff_decrease}
Suppose that \ref{ass:general_1}-\ref{ass:general_3} hold,  then for every $u \in \domF$ and $l \in (\frac{\LipcostrD}{2}, \infty)$,  we have 
\begin{equation*}
\cost(\pgrad_{l}(u))  \leq  \cost(u) -\frac{l-\frac{\LipcostrD}{2}}{l^2} \|\gmap_{l}(u)\|_\Hs^2.
\end{equation*}   
\end{lemma}

The previous lemmas allow us to summarize some important properties of Algorithm \ref{algo:NMLS}.

\begin{lemma} 
\label{lem:properties_of_iter}
Suppose that \ref{ass:general_1}-\ref{ass:general_2} hold.  Then for every $k\geq 0$,  the following statements hold true:
\begin{enumerate}[label=(\roman*)]
\item For  $ \delta\in  (0,\frac{1}{2})$,  the nonmonotone linesearch is well-defined.  That is,   there exists $\underline{\alpha}>0 $ such that for every $u_k \in  \domF$  and  $\ssizek \in [ \underline{\alpha}, \infty) $ the nonmonotone rule \eqref{eq:ls_ineq} holds and thus,  the nonmonotone  linesearch terminates after finitely  many iterations.

\end{enumerate}
Assume,  in addition,  \ref{ass:general_3} holds.
\begin{enumerate}[resume,label=(\roman*)]
\item  Then the statement of  \rom{1}  is true for $\underline{\alpha} \coloneqq \frac{\LipcostrD}{2(1-\delta)}$  and every fixed  $\delta \in (0,1)$.

\item The step-sizes are uniformly bounded from above.  That is, for every $k \geq 1$,  we have $\ssizek \leq  \ssizebar $ with $\ssizebar \coloneqq \max\{ \frac{ \eta\LipcostrD}{2(1-\delta)}, \ssizeub \}$.

\item  It holds $ \|\gmap_{\ssize_{k+1}}(u_{k+1})\|_\Hs  \leq   C_G  \|\gmap_{\ssizek}(u_{k})\|_\Hs$,   where  $C_G :=\frac{3 \ssizebar + \LipcostrD}{\ssizelb}$.  
\end{enumerate}
\end{lemma}
\begin{proof}
\rom{1} If $u_k \in  \crit$, then the claim clearly holds true for every $\underline{\alpha} >0$.    Thus, we assume that $u_k \notin \crit$.   We suppose also on contrary it does not exist $\underline{\alpha} >0$ for which the claim holds true.  Then the nonmonote linesearch generates a sequence of step-sizes $\ssize_{k_i} := \ssizeinit \eta^i$ with $i \in \N_0$ satisfying $\ssize_{k_i} \to \infty$ and
\begin{equation}
\label{eq:B61}
 \frac{\delta}{\ssize_{k_i}} \|\gmap_{\ssize_{k_i}}(u_k)\|_\Hs^2  >  \cost(u_{\ell(k)}) - \cost(\pgrad_{\ssize_{k_i}}(u_k)) \geq  \cost(u_{k}) - \cost(\pgrad_{\ssize_{k_i}}(u_k)).
\end{equation}
First,  we note that $\pgrad_{\ssize_{k_i}}(u_k)\to u_k$ as $i \to \infty$.  This follows from the fact that $\prox_{\frac{1}{\ssize_{k_i}}\costp}(u_k)\to u_k$ as $i \to \infty$  (see e.g., \cite[Theorem 23.47]{MR3616647}) and
\begin{equation*}
\begin{split}
\|\pgrad_{\ssize_{k_i}}(u_k)- u_k\|_\Hs & \leq \|  \prox_{\frac{1}{\ssize_{k_i}} \costp}(u_k - \frac{1}{\ssize_{k_i}} \costrD(u_k))- \prox_{\frac{1}{\ssize_{k_i}} \costp}(u_k) \|_\Hs \\
& +\| \prox_{\frac{1}{\ssize_{k_i}} \costp}(u_k)- u_k\|_\Hs \leq \| \prox_{\frac{1}{\ssize_{k_i}} \costp}(u_k)- u_k\|_\Hs + \frac{1}{\ssize_{k_i}}\|\costrD(u_k)\|_\Hs,
\end{split}
\end{equation*}   
where we have used the firm nonexpansiveness of the proximal operator.   Further,  using  \eqref{eq:B61} and the mean value theorem for $\costr$, we obtain for every $i\in \N_0$,  that 
\begin{equation*}
\begin{split}
& \frac{\delta}{\ssize_{k_i}} \|\gmap_{\ssize_{k_i}}(u_k)\|_\Hs^2 \geq \cost(u_{k}) - \cost(\pgrad_{\ssize_{k_i}}(u_k))\\
& \geq \costp(u_{k}) - \costp(\pgrad_{\ssize_{k_i}}(u_k)) +{(\costrD(u_k+t_i(\pgrad_{\ssize_{k_i}}(u_k))-u_k)),  u_k- \pgrad_{\ssize_{k_i}}(u_k)))}_\Hs,
\end{split}
\end{equation*}
where $t_i \in (0,1]$ for all $i\in \N_0$. Using \eqref{eq:B59} we obtain that 
\begin{equation*}
\begin{split}
& \frac{\delta}{\ssize_{k_i}} \|\gmap_{\ssize_{k_i}}(u_k)\|_\Hs^2 \geq \frac{1}{2\ssize_{k_i}}\|\gmap_{\ssize_{k_i}}(u_k)\|_\Hs^2 \\
& + \left( \costrD\left(u_k+t_i(\pgrad_{\ssize_{k_i}}(u_k)-u_k)\right) -\costrD(u_k),  \pgrad_{\ssize_{k_i}}(u_k)-u_k  \right)_\Hs \\
& \geq  \frac{1}{2\ssize_{k_i}}\|\gmap_{\ssize_{k_i}}(u_k)\|_\Hs^2 \\
& - \frac{1}{\ssize_{k_i}} \|  \costrD \left(u_k+t_i(\pgrad_{\ssize_{k_i}}(u_k)-u_k)\right)-\costrD(u_k) \|_\Hs \|\gmap_{\ssize_{k_i}}(u_k) \|_\Hs.
\end{split}
\end{equation*}
Together with the fact that  $\delta <\frac{1}{2}$  we obtain 
\begin{equation*}
(\frac{1}{2}-\delta)\|\gmap_{\ssize_{k_i}}(u_k)\|_\Hs \leq  \|  \costrD\left(u_k+t_i(\pgrad_{\ssize_{k_i}}(u_k)-u_k)\right)-\costrD(u_k) \|_\Hs.
\end{equation*} 
Sending $i\to \infty$ and using the continuity of $\costrD$,  we obtain that 
 \begin{equation*}
  \lim_{i  \to  \infty}  \|\gmap_{\ssize_{k_i}}(u_k)\|_\Hs = 0.
 \end{equation*}
Finally, using \ref{P2},  we can infer that $ \|\gmap_{\ssizelb}(u_k)\|_\Hs  \leq   \|\gmap_{\ssize_{k_i}}(u_k)\|_\Hs$ and thus $\|\gmap_{\ssizelb}(u_k)\|_\Hs = 0$. Now Proposition \ref{prop:stationary_point} implies $u_k \in \crit$. This contradicts our assumption,  which concludes the proof.

\rom{2} For every given  $\ssizek \geq  \underline{\alpha}  > \frac{\LipcostrD}{2}$ we can  invoke Lemma \ref{lem:suff_decrease}  and write 
\begin{equation*}
\begin{split}
\cost(\pgrad_{\ssizek}(u_k))  &\leq  \cost(u_k) -\frac{\ssizek-\frac{\LipcostrD}{2}}{\ssizek^2} \|\gmap_{\ssizek}(u_k)\|_\Hs^2 \\\
                           &\leq  \max\limits_{0 \leq j \leq m(k)} \cost(u_{k - j}) -\frac{\ssizek-\frac{\LipcostrD}{2}}{\ssizek^2} \|\gmap_{\ssizek}(u_k)\|_\Hs^2.                         
 \end{split}
\end{equation*}
Thus,  \eqref{eq:ls_ineq} holds since $\ssizek$ satisfies  $\frac{\ssizek-\frac{\LipcostrD}{2}}{\ssizek^2} \geq  \frac{\delta}{\ssizek} $ by assumption. 
 
\rom{3} To derive $\ssizebar$,  we consider the following cases:
\begin{itemize}
\item $i_k = 0$:  In this case,  due to \eqref{eq:ls_update},  we have $\ssizek =  \ssizeinit \leq \ssizeub$.
\item $i_k \geq 1$:  In this case,  \eqref{eq:ls_ineq} holds  for $\ssizek$  and $u_k$. Then due to \eqref{eq:ls_update},  we can write
\begin{equation}
\label{eq:B10}
\begin{split}
\cost(\pgrad_{\frac{\ssizek}{\eta}}(u_k)) & >  \max\limits_{0 \leq j \leq m(k)} \cost(u_{k - j}) - \frac{\delta\eta}{\ssizek} \|\gmap_{\frac{\ssizek}{\eta}}(u_k)\|_\Hs^2 \\
&\geq \cost(u_{k}) - \frac{\delta\eta}{\ssizek} \|\gmap_{\frac{\ssizek}{\eta}}(u_k)\|_\Hs^2.
\end{split} 
\end{equation}
Further, by assuming without loss of generality that $\frac{\ssizek}{\eta} \geq \frac{\LipcostrD}{2}$,  we can use Lemma \ref{lem:suff_decrease} with $l = \frac{\ssizek}{\eta}$ and $u = u_k$ to obtain  
\begin{equation}
\label{eq:B11}
\cost(\pgrad_{\frac{\ssizek}{\eta}}(u_k))  \leq  \cost(u_k) -\frac{\frac{\ssizek}{\eta}-\frac{\LipcostrD}{2}}{(\frac{\ssizek}{\eta})^2} \|\gmap_{\frac{\ssizek}{\eta}}(u_k)\|_\Hs^2.
\end{equation}
Combining \eqref{eq:B10} and  \eqref{eq:B11},  we obtain $\frac{\delta \eta}{\ssizek} > \frac{\frac{\ssizek}{\eta}-\frac{\LipcostrD}{2}}{(\frac{\ssizek}{\eta})^2}$ and as a consequence,   $\ssizek <  \frac{\eta \LipcostrD}{2(1-\delta)}$.
\end{itemize}
Summarizing the two above cases,  we can conclude the second part with $\ssizebar \coloneqq \max\{ \frac{ \eta\LipcostrD}{2(1-\delta)}, \ssizeub \}$.

\rom{4} Finally for proving the last part, we use Lemma \ref{lem:properties_of_Grad} and \rom{3} to obtain that 
\begin{equation*}
\begin{split}
&\|\gmap_{\ssize_{k+1}}(u_{k+1})\|_\Hs  \stackrel{\text{\ref{P2}}}{\leq} \|\gmap_{ \ssizebar}(u_{k+1})\|_\Hs  \leq \|\gmap_{ \ssizebar}(u_{k+1})-\gmap_{ \ssizebar}(u_{k})\|_H+\|\gmap_{ \ssizebar}(u_{k})\|_\Hs \\&
 \stackrel{\text{\ref{P3}}}{\leq} (2 \ssizebar+\LipcostrD) \|u_{k+1}-u_{k} \|_\Hs +\|\gmap_{ \ssizebar}(u_{k})\|_\Hs   \stackrel{\text{Def. \ref{def:prox_gmap}}}{\leq}  \frac{(2 \ssizebar+\LipcostrD)}{\ssizek}  \| \gmap_{\ssizek}(u_{k}) \|_\Hs + \|\gmap_{ \ssizebar}(u_{k})\|_\Hs \\
 &  \stackrel{\text{\ref{P1}}}{\leq} \frac{(3 \ssizebar + \LipcostrD)}{\ssizek}  \| \gmap_{\ssizek}(u_{k}) \|_\Hs \leq \frac{3 \ssizebar+\LipcostrD}{\ssizelb}  \| \gmap_{\ssizek}(u_{k}) \|_\Hs.
\end{split}
\end{equation*}
Setting $C_G := \frac{3 \ssizebar+\LipcostrD}{\ssizelb}$,  the proof is complete.
\end{proof}

Before stating the main convergence result of this section, we present a lemma concerning the sequence $\{u_{\nu(k)}\}_k$ and well-posedness of Algorithm \ref{algo:NMLS}.

\begin{lemma} 
\label{lem:without_lip}
Suppose that the sequences $\{ u_k  \}_k$ and $\{\ssizek\}_k \subset \R_{> 0}$ are generated by Algorithm  \ref{algo:NMLS}.  Then the following properties hold:
\begin{description}
\item[\namedlabel{L1}{L1}:] $\{ u_{\nu(k)} \}_{k}$ is a subsequence of $ \{ u_{k} \}_{k}$  with  $\nu(k)-\nu(k-1) \leq 2m_{\max}+1$ and $\nu(k) \leq  (m_{\max}+1)k$.  Further,  for every $k \in \N$,  it holds $\cost(u_k) \leq \cost(u_{ \nu(\ceil{\frac{k}{\mmax+1}})})$.  
\item[\namedlabel{L2}{L2}:]  For every $k\geq1$,  we have 
\begin{equation}
\label{eq:B5}
 \cost(u_{\nu(k)}) \leq \cost(u_{\nu(k-1)}) - \frac{\delta}{\alpha_{\nu(k)-1}} \|\gmap_{\nu(k)-1}(u_{\nu(k)-1})\|_\Hs^2,
 \end{equation}
 and in particular the subsequence $\{ \cost(u_{\nu(k)}) \}_{k}$ is montonically decreasing.
 \item[\namedlabel{L3}{L3}:]Assume that $\cost$ is bounded from below,  i.e.  $\inf\limits_{u \in \Hs} \cost(u) > - \infty$.  Then we have
 \begin{equation}
 \label{eq:B62}
  \sum^{\infty}_{k=1}\frac{1}{\alpha_{\nu(k)-1}} \|\gmap_{\nu(k) - 1}(u_{\nu(k) - 1})\|^2_\Hs < \infty \text{ and  } \liminf_{k \to \infty} \frac{1}{\ssizek} \|\gmap_{\ssizek}(u_{k})\|^2_\Hs = 0.
 \end{equation}
 In particular,  if $\mmax = 0$,   we have    
 \begin{equation}
 \label{eq:B63}
 \sum^{\infty}_{k=0}\frac{1}{\ssizek} \|\gmap_{\ssizek}(u_{k})\|^2_\Hs < \infty   \text{ and }  \lim \frac{1}{\ssizek} \|\gmap_{\ssizek}(u_{k})\|^2_\Hs = 0.
 \end{equation}
\end{description}
\end{lemma}
\begin{proof}
(L1) Using the fact that $\ell(k) \geq k - \mmax$ for every $k$,  we obtain 
\begin{equation*}
\begin{split}
\nu(k) &= \ell(k \mmax+k)  \geq k \mmax +k -\mmax \\&
> (k-1) \mmax+ (k-1) \geq \ell((k-1)\mmax +(k-1)) = \nu(k-1),  
\end{split}
\end{equation*}
and thus $\{ u_{\nu(k)} \}_{k} \subset \{ u_{k} \}_{k}$.  Moreover, we have 
\begin{equation*}
\nu(k) = \ell( k\mmax+k) \leq k\mmax+k,
\end{equation*}
and 
\begin{equation*}
\begin{split}
\nu(k)&-\nu(k-1) = \ell(k \mmax +k)-\ell((k-1)\mmax +(k-1)) \\
                         &\leq k\mmax+k-\left((k-1)\mmax+(k-1)- \mmax \right) = 2\mmax +1.         \end{split}
\end{equation*}
Further,  for a given $k$ we have $0 \leq \ceil{ \frac{k}{\mmax+1}}(\mmax +1)-k  \leq \mmax$,  and thus 
\begin{equation*}
\cost(u_k) \leq \cost(u_{\ell( \ceil{ \frac{k}{\mmax+1}}(\mmax +1))}) = \cost(u_{\nu(\ceil{ \frac{k}{\mmax+1}})}),
\end{equation*}
and,  thus, we are finished with the verification of  \ref{L1}.

(L2) Inserting  $\nu(k)-1$ in place of $k$  in \eqref{eq:B55b},  we obtain  
\begin{equation}
\label{eq:B7}
\cost(u_{\nu(k)}) \leq \cost(u_{\ell( \nu(k)-1)})  - \frac{\delta}{\ssize_{\nu(k)-1}} \|\gmap_{\ssize_{\nu(k)-1}}(u_{\nu(k)-1})\|_\Hs^2.
\end{equation}
 Further,  we can write
\begin{equation}
\label{eq:B6}
\begin{split}
\cost(u_{\ell(k+1)})& = \max\limits_{0 \leq j \leq m(k+1)} \cost(u_{k+1 - j})  \leq \max \limits_{0 \leq j \leq m(k)+1} \cost(u_{k+1 - j})  \\
& \leq  \max \left[  \cost(u_{k+1}),   \max \limits_{1 \leq j \leq m(k)+1} \cost(u_{k+1 - j}) \right] \\
&\leq  \max \left[ \cost(u_{\ell(k)}) - \frac{\delta}{\ssizek} \|\gmap_{\ssizek}(u_k)\|_\Hs^2 ,  \cost(u_{\ell(k)})\right] \leq \cost(u_{\ell(k)}), 
\end{split} 
\end{equation}
where we have used $m(k+1) \leq m(k)+1$.  Therefore,  $ \{ \cost(u_{\ell(k)}) \}_k$ is decreasing and  we can write
\begin{equation}
\label{eq:B30}
\begin{split}
\cost(u_{\ell( \nu(k)-1)}) & = \cost(u_{\ell( \ell(k\mmax+k)-1)}) \leq \cost(u_{\ell( k\mmax +k-\mmax-1)}) \\
& =  \cost(u_{\ell( (k-1)\mmax +(k-1))}) = \cost(u_{\nu(k-1)}).
\end{split}
\end{equation}
Together with \eqref{eq:B7} we can conclude \eqref{eq:B5} and,  thus,  \ref{L2} holds.

(L3) Assume that Algorithm \ref{algo:NMLS} does not converge after finitely many iterations.  Summing \eqref{eq:B5} up for $k = 1, \dots, k'$,  we obtain
\begin{equation}
\label{eq:B51}
\begin{split}
\sum^{k'}_{k = 1} \frac{\delta}{\ssize_{\nu(k)-1}} \|\gmap_{\ssize_{ \nu(k)-1}}(u_{\nu(k)-1})\|_\Hs^2 & \leq  \sum^{k'}_{k = 1}  \cost(u_{\nu(k-1)}) -\cost(u_{\nu(k)})\\ 
& \leq  \cost(u_{\nu(0)})-\cost(u_{\nu(k')}).
\end{split}
\end{equation}
Sending $k'$ to infinity and using the fact that  $\cost$ is bounded from below,  we can conclude \eqref{eq:B62}. Similarly \eqref{eq:B63} follows by using the fact that for $\mmax = 0$ it holds $\nu(k) = \ell(k)=k$,   and, thus, using \eqref{eq:B62}, we can conclude the proof.
\end{proof}

Finally we are ready to present our main convergence result of this section.

\begin{theorem}
\label{Thm:conv}
Suppose that \ref{ass:general_1}-\ref{ass:general_3} hold and that $\cost$ is bounded from below.  Then,  for the sequence $\{ u_k  \}_k \subset \Hs$  generated by Algorithm  \ref{algo:NMLS} with $ \{\ssizek \}_k \subset \R_{>0}$,  the following statements holds true:
\begin{enumerate}[label=(\roman*)]
\item Either Algorithm \ref{algo:NMLS} terminates  after finitely many iterations with a stationary point of \eqref{eq:opt_problem} or the sequence  $\{ \|\gmap_{\ssizek}(u_{k})\|_\Hs \}_k$ converges to zero,  that is 
\begin{equation}
\label{eq:B9}
\lim_{k \to \infty}  \|\gmap_{\ssizek}(u_{k})\|_\Hs = 0.
\end{equation}
\item The following inequality holds true
\begin{equation}
\label{eq:B15}
\begin{split}
    \min_{0 \leq i  \leq  k}\|\gmap_{\ssizelb}(u_{i})\|_\Hs  \leq C^{\mmax}_G \sqrt{ \frac{ \ssizebar(\mmax +1)( \cost(u_0)-\cost(u_{k})) }{k\delta} }.
 \end{split}
\end{equation}
\item If, in addition,  \ref{ass:general_4} holds,   every weak sequential accumulation point of  $\{ u_k \}_k \subset \Hs$ is a stationary point of \eqref{eq:opt_problem}.
\end{enumerate}
\end{theorem}

\begin{proof}
\rom{1} Note that if Algorithm \ref{algo:NMLS} converges after finitely many iterations, by definition of the stopping criterion and Proposition \ref{prop:stationary_point}, a stationary point of \eqref{eq:opt_problem} has been found. Now assume that Algorithm \ref{algo:NMLS} does not converge after finitely many iterations. Then,  using \ref{L3},  and the fact that $ \ssize_k  \leq  \ssizebar$ for all $k$ by \rom{3} from Lemma \ref{lem:properties_of_iter}, we arrive at 
\begin{equation}
\label{eq:B12}
\lim_{k \to \infty}  \|\gmap_{\ssize_{ \nu(k)-1}}(u_{\nu(k)-1})\|_\Hs = 0.
\end{equation}
Now it remains to show that  \eqref{eq:B9} holds.  To show this,  we will successively use \rom{4} of Lemma \ref{lem:properties_of_iter}.   Let  $k\geq 0$ be arbitrary.   Using the fact that  $0 \leq k-\floor{ \frac{k}{m_{\max}+1}}(m_{\max} +1)  \leq \mmax$ and 
 \begin{equation*}
  \floor{ \frac{k}{\mmax+1} } (\mmax +1) - \ell \left(\floor{ \frac{k}{\mmax+1} } (\mmax +1)\right) \leq \mmax,
 \end{equation*}
we can write
\begin{equation}
\label{eq:B52}
\begin{split}
 \|\gmap_{\ssizek}(u_{k})\|_\Hs & \leq  C^{\mmax}_G\|\gmap_{\ssize_{\floor{ \frac{k}{\mmax+1} } (\mmax +1) }}(u_{\floor{ \frac{k}{\mmax+1} } (\mmax +1) })\|_\Hs\\ & \leq  C^{2\mmax}_G\|\gmap_{\ssize_{\ell \left(\floor{ \frac{k}{\mmax+1} } (\mmax +1)\right)  }}(u_{\ell \left(\floor{ \frac{k}{\mmax+1} } (\mmax +1)\right) })\|_\Hs \\ & \leq C^{2\mmax}_G\|\gmap_{\ssize_{\nu \left(\floor{ \frac{k}{\mmax+1} }\right)  }}(u_{\nu \left(\floor{ \frac{k}{\mmax+1}}\right) })\|_\Hs \\ & \leq C^{2\mmax+1}_G\|\gmap_{\ssize_{\nu \left( \floor{ \frac{k}{\mmax+1} }\right)-1 }}(u_{\nu \left(\floor{ \frac{k}{\mmax+1} }\right)-1 })\|_\Hs.
 \end{split}
\end{equation}
Finally,  sending  $k$ to $\infty$ and using \eqref{eq:B12},  we obtain \eqref{eq:B9}  and the proof of \rom{1} is complete.

\rom{2} Due to the facts that $\cost(u_{k})\leq \cost(u_{\nu(\ceil{\frac{k}{\mmax+1}})})$ (by \ref{L1}),  $\cost(u_0) = \cost(u_{\nu(0)})$,  and by successively using \eqref{eq:ls_ineq}, we obtain that 
\begin{equation}
\label{eq:B14}
\begin{split}
&\cost(u_{k}) -  \cost(u_0) \leq  \cost(u_{\nu(\ceil{\frac{k}{\mmax+1}})})-\cost(u_0)\\ & \leq  \cost(u_{\nu(\ceil{\frac{k}{\mmax+1}})}) -\cost(u_{\nu(\ceil{\frac{k}{\mmax+1}}-1)}) +\cost(u_{\nu(\ceil{\frac{k}{\mmax+1}}-1)})- \cdots\\&-\cost(u_{\nu(1)}) +\cost(u_{\nu(1)})   -\cost(u_0) \\ & \leq \sum^{\ceil{\frac{k}{\mmax+1}}}_{i =1}  - \frac{\delta}{\ssize_{\nu(i)-1}} \|\gmap_{\ssize_{\nu(i)-1}}(u_{\nu(i)-1})\|_\Hs^2 \leq  \sum^{\ceil{\frac{k}{\mmax+1}}}_{i =1}  - \frac{\delta}{ \ssizebar} \|\gmap_{\ssize_{\nu(i)-1}}(u_{\nu(i)-1})\|_\Hs^2.
\end{split}
\end{equation}
Thus,  using \ref{L1},  we can infer that 
\begin{equation*}
\begin{split}
 & \frac{k\delta}{ \ssizebar (\mmax+1)}   \min_{0 \leq i  \leq  \ceil{\frac{k}{\mmax+1}}(\mmax+1)-1}\|\gmap_{\ssize_{i}}(u_{i})\|_\Hs^2  \\ & \leq   \frac{\ceil{\frac{k}{\mmax+1}}\delta}{ \ssizebar} \min_{1 \leq i  \leq  \ceil{\frac{k}{\mmax+1}}}\|\gmap_{\ssize_{\nu(i)-1}}(u_{\nu(i)-1})\|_\Hs^2  \\
  & \leq  \sum^{\ceil{\frac{k}{\mmax+1}}}_{i =1}  \frac{\delta}{ \ssizebar} \|\gmap_{\ssize_{\nu(i)-1}}(u_{\nu(i)-1})\|_\Hs^2 \leq   \cost(u_0)-\cost(u_{k}), 
 \end{split}
\end{equation*}
and this yields 
\begin{equation}
\label{eq:B16}
\begin{split}
    \min_{0 \leq i  \leq  \ceil{\frac{k}{\mmax+1}}(\mmax+1)-1}\|\gmap_{\ssize_{i}}(u_{i})\|_\Hs   \leq \sqrt{ \frac{ \ssizebar(\mmax+1)( \cost(u_0)-\cost(u_{k}))}{k\delta} }.
 \end{split}
\end{equation}
Further,  using  the second part of Lemma \ref{lem:properties_of_iter}  and  the fact that $ \ceil{\frac{k}{\mmax+1}}(\mmax+1)-1 -k < \mmax$, we can deduce that 
\begin{equation}
\label{eq:B17}
  \min_{0 \leq i  \leq  k}\|\gmap_{\ssize_{i}}(u_{i})\|_\Hs   \leq  C^{\mmax}_G  \min_{0 \leq i  \leq  \ceil{\frac{k}{\mmax+1}}(\mmax+1)-1}\|\gmap_{\ssize_{i}}(u_{i})\|_\Hs.
\end{equation}
Finally,  \eqref{eq:B15}  follows from   \eqref{eq:B16},  \ref{P2},  and the fact that $\ssizek \geq \ssizelb$ for all $k\geq 0$. 

\rom{3} We show that every weak sequential accumulation point of  $\{ u_k \}_k \subset \Hs$ is a stationary point of \eqref{eq:opt_problem}.  In other words,  we suppose that $u_{k_i} \rightharpoonup u^*$ for a subsequence $ \{ u_{k_i} \}_i \subset   \{ u_{k} \}_k$ and $u^* \in \Hs$.  Then we show that $u^* \in \crit$.  From now on, for the sake of convenience, we use the same notation for the subsequence as for the sequence itself.  To begin,  due to \eqref{eq:B9} in \rom{1},  \ref{P2},  and the fact that $\ssizek \geq \ssizelb$,  we can conclude that 
\begin{equation}
\label{eq:B64}
\lim_{k \to \infty}  \|\gmap_{\ssizelb}(u_{k})\|_\Hs = 0  \text{  and  }  \lim_{k \to \infty} \|\pgrad_{\ssizelb}(u_k) - u_k \|_H=0.
\end{equation}        
Applying \eqref{eq:prox} for $\ssize = \ssizelb$ and $u = u_k$,  we obtain for every $k \in \N$ 
\begin{equation}
\label{eq:B65}
\gmap_{\ssizelb}(u_k)-\costrD(u_k) \in \partial \costp( \pgrad_{\ssizelb}(u_k)). 
\end{equation}
Due to  \ref{ass:general_4},  we can infer that   $u_{k} \rightharpoonup u^*$ implies  $\costrD(u_k) \to  \costrD(u^*)$ and,  thus,  using \eqref{eq:B64} we have  $\gmap_{\ssizelb}(u_k)-\costrD(u_k)  \to -\costrD(u^*)$ as $k \to \infty$.  Due to \eqref{eq:B64},  we also can conclude that $\pgrad_{\ssizelb}(u_k) \rightharpoonup u^*$ as  $k \to \infty$.  Therefore,   sending $k \to \infty$ in \eqref{eq:B65} and using the fact that the graph of $\partial \costp$  is  sequentially closed  \cite[Proposition 16.26]{MR3616647}  under the weak topology for domain and  the strong topology for codomain,  we obtain $ -\costrD(u^*) \in \partial \costp(u^*)$.  This completes the proof.
\end{proof}

In the next theorem, we derive an estimate that reflects the worst-case complexity of the required function and gradient-like evaluations of Algorithm \ref{algo:NMLS} to find an $\tol$-stationary point.  The proof is inspired by the one  given in \cite[Theorem 3.4]{zbMATH06431462} for smooth problems.

\begin{theorem}[Worst-case complexity]\label{thm:complexity}
Suppose that \ref{ass:general_1}-\ref{ass:general_3} hold and that $\cost$ is bounded from below with  $\costbar \coloneqq \inf_{u \in \Hs} \cost(u)>-\infty$. Then for a given tolerance  $\tol > 0$,  Algorithm \ref{algo:NMLS} requires at most
\begin{equation}
\label{eq:B19}
\maxfeval \coloneqq  \floor{ \frac{\maxfparam (\cost({u_0})-\costbar)}{ \tol^2} }
\end{equation}
function evaluations of $\cost$  and 
\begin{equation}
\label{eq:B20}
\maxgeval \coloneqq \floor{ \frac{\maxgparam(\cost({u_0})-\costbar)}{ \tol^2} }
\end{equation}
Gradient-like  $\gmap_{\ssize}(\cdot)$ evaluations to find an iterate $u_k$  satisfying  $\|\gmap_{\ssizek}(u_k)\|_\Hs \leq \tol$,  where 
\begin{equation*}
\maxfparam := \frac{(\mmax+1) C^{2\mmax}_G}{\decparam}  \quad \text{ and }  \quad  \maxgparam := \frac{ (\mmax+1)  \ssizebar C^{2\mmax}_G  }{\delta},
\end{equation*}
with
\begin{equation*}
\decparam \coloneqq \min \left\{  \frac{\delta}{\ssizeub} ,   \frac{2(1-\delta)\delta}{ n_1\eta \LipcostrD} \right\} \quad \text{ and } \quad n_1 := \floor{ \abs{ \log_{\eta}\left(  \frac{\eta \LipcostrD}{2\ssizelb(1-\delta)}   \right) }}.
\end{equation*}
\end{theorem}

\begin{proof}
The proof can be found in Appendix \ref{appendix:complexity}.
\end{proof}

This finishes our considerations of convergence and complexity in the general setting. 

\subsection{Convex case}
\label{Convex case}
In this section higher order convergence rates will be shown in two cases of additional structural assumptions.  Firstly, we consider the case of an additional convexity assumption on $\costr$.   Afterwards the case of a quadratic growth assumption on $\cost$ will be considered. We assume the following modified version of Assumption \ref{ass:general}.

\begin{assumption}
\label{ass:convex}
Assume that \ref{ass:general_1}-\ref{ass:general_3} hold.  Further,  instead of  \ref{ass:general_4}, assume that 
\begin{description}
\item[\namedlabel{ass:convex_4}{A'4}:]  $\costr \colon \Hs \rightarrow \R$ is  convex. 
\end{description} 
\end{assumption}
Under Assumption  \ref{ass:convex},  the whole function $\cost$ is convex.  In this case,   we can conclude that the set of minimizers of $\cost$ coincides with $\crit$ provided that  $\crit \neq \emptyset$ and that  $\crit$  is closed  and convex.  Further we have
\begin{equation*}
 \crit  = \Argmin \cost = (\partial \cost)^{-1}(0).
\end{equation*}
The associated minimal function value is denoted $\cost^* \in \R$.

Next we prove an auxiliary lemma which will be used later.
\begin{lemma}\label{lem:quadratic}
 Suppose that Assumption \ref{ass:convex} holds,  $\crit \neq \emptyset$,  and $\{ u_k \}_k$ is generated by Algorithm \ref{algo:NMLS}.  Then for every  $\lambda \in [0,1]$,  it holds
\begin{equation}
\label{eq:B44}
\begin{split}
\cost(u_{\nu(k)})-\cost^* &\leq (1-\lambda)\left(\cost(u_{\nu(k-1)})- \cost^* \right) +\frac{\ssizebar \lambda^2}{2}  \dist^2(u_{\nu(k)-1}, \crit) \\& + \frac{\tilde{C}}{\ssize_{\nu(k)-1}} \|\gmap_{\ssize_{\nu(k)-1}}(u_{\nu(k)-1})\|_\Hs^2,
\end{split}
\end{equation}
where $\tilde{C} :=\frac{\LipcostrD}{2\ssizelb}$.   Further,  we have the following inequality for the initial iterations   
\begin{equation}
\label{eq:B31}
\begin{split}
\cost(u_{\nu(1)})-\cost^* \leq C_0   \dist^2(u_0, \crit),
\end{split}
\end{equation}
with constant $C_0$ which is independent of $u_0$.
\end{lemma}
\begin{proof}
The proof can be found in Appendix \ref{appendix:quadratic}.
\end{proof}

Now we are ready to provide the main convergence result for the convex case. 
For sake of convenience in the presentation, we set 
\begin{equation*}
\mathcal{E}_k  := \cost(u_k)-\cost^*  \quad \text{ for evey } k\geq 0
\end{equation*}  
and use this notation in the remainder of this section.  

\begin{theorem}[Global convergence and $O(k^{-1})$ complexity for the convex case]
\label{Thm:glabal_conver_con}
Suppose that Assumption \ref{ass:convex} holds,  $\crit \neq \emptyset$, and  $\{ u_k \}_k$ is generated by Algorithm \ref{algo:NMLS}.  Then the following statements hold true:
\begin{enumerate}[label=(\roman*)]
\item  Every sequential weak accumulation point of $\{ u_k \}_k$ belongs to $\crit$.

\item $\{ u_k \}_k$ converges weakly to a minimizer $u^* \in \crit$ and its  "shadow" sequence converges strongly to $u^*$,  that is  $P_{\crit}u_k  \to u^*$.

\item  It holds  $\cost(u_{k}) \to \cost^*$ as $k \to \infty$ and for large enough $k\geq 0$,  there exist constants $\rho_1, \rho_2 > 0$ such that   
\begin{equation}
\label{eq:B42}
 \cost(u_k)-\cost^* \leq \frac{\rho_1}{ \rho_2+k}.
\end{equation} 
\end{enumerate}  
\end{theorem} 
\begin{proof}
\rom{1} We suppose that  a subsequence $\{u_{k_i}\}_i$ with $u_{k_i}  \rightharpoonup u^*$ is given.   We will show that  $u^* \in \crit$.  To show this,  we prove that  there exists a vanishing  sequence of subgradients $\{ w_{k_i}  \}_i \in H$, i.e.  $ w_{k_i} \to 0$,  corresponding to $\{u_{k_i}\}_i$  with  $w_{k_i} \in \partial \cost(u_{k_i})$ for every $i\in \N$.   Using \eqref{eq:prox} and \eqref{eq:subdiff_rule}, we define 
\begin{equation}
\label{eq:B47}
w_{k+1} := \gmap_{\ssizek}(u_k) +\costrD(u_{k+1}) - \costrD(u_{k}) \in \partial   \cost(u_{k+1}).
\end{equation}
Further,  since $\crit \neq \emptyset$,  $\cost$ is bounded from below,  we can use \rom{1} of Theorem \ref{Thm:conv} and, thus \eqref{eq:B9} holds.   This, together with \eqref{eq:B47} and the boundedness of $\alpha_k$, implies  
\begin{align*}
& \|w_{k+1}\|_\Hs \leq \|\gmap_{\alpha_k}(u_k)\|_\Hs +\| \costrD(u_{k+1}) - \costrD(u_{k}) \|_\Hs \leq (1+\frac{\LipcostrD}{\alpha_k}) \|\gmap_{\alpha_k}(u_k)\|_\Hs \to 0, && k \to \infty
\end{align*}
In particular,  we can infer that $w_{k_i} \to 0$ as $i \to \infty$.  Using the fact that the graph of $\partial \cost$  is  sequentially closed  \cite[Proposition 16.26]{MR3616647}  under the weak topology for domain and  the strong topology for codomain together with  $w_{k_i} \to 0$ and $u_{k_i}  \rightharpoonup u^*$,  we arrive at $0 \in \partial \cost (u^*)$ and therefore $u^* \in \crit$.  

\rom{2} We show that $u_k  \rightharpoonup u^*$  with  $u^* \in \crit$.  In this matter,  we show that $\{ u_k \}_k$ is a quasi-Fej\'er sequence with respect  to $\crit \neq \emptyset$.  Using \eqref{eq:prox}, we can write   
\begin{equation}
\label{eq:B66}
0 \in \ssizek (u_{k+1} -u_k) + \partial \costp(u_{k+1})+\costrD(u_k).
\end{equation} 
Further,  since $\costp$ is convex and  $\costrD$ is Lipschitz continuous,  by the Haddad-Bailon Theorem \cite[Corollary 18.16, p.~270]{MR3616647},  we have 
\begin{equation*}
( \costrD(v)- \costrD(w), v-w)_\Hs  \geq  \LipcostrD^{-1}  \| \costrD(v)- \costrD(w)\|^2_{\Hs}.
\end{equation*}   
Further,  we can write for every $v,w, z \in \domF$ that 
\begin{equation}
\label{eq:B49}
\begin{split}
& {( \costrD(v)- \costrD(w), z-w)}_\Hs = {(\costrD(v)-\costrD(w), v-w)}_\Hs  + {(\costrD(v)-\costrD(w), z-v)}_\Hs \\& \geq \LipcostrD^{-1}  \| \costrD(v)- \costrD(w)\|_{\Hs}^2- \|\costrD(v)- \costrD(w) \|_{\Hs}\|z-v\|_{\Hs}  \geq  -\frac{\LipcostrD}{4} \|z-v\|^2_{\Hs},
\end{split} 
\end{equation} 
where in the last line we have used the fact that $f(x):= \LipcostrD^{-1} x^2- x \|z-v\|_{\Hs}$ is strictly convex and attains its global minimum at 
$x^*=\frac{\LipcostrD}{2} \|z-v\|_{\Hs}$. 

Using  \eqref{eq:B49} for $u_{k+1}$,  $u_k$,  and any $u^* \in \crit$  in place of  $v$, $z$,  and $w$,  respectively,  we obtain
\begin{equation*}
\begin{split}
( \costrD(u_{k})- \costrD(u^*),  u_{k+1}-u^*)_\Hs \geq -\frac{\LipcostrD}{4} \|u_{k+1}-u_k\|^2_{\Hs}.
\end{split} 
\end{equation*}  
Together with the fact that $ \partial \costp$ is  monotone, cf. \cite{MR3616647}, and \eqref{stationary},  we can write 
\begin{equation}
\label{eq:B67}
(\partial \costp(u_{k+1})+\costrD(u_k) ,  u_{k+1}-u^*)_\Hs \geq -\frac{\LipcostrD}{4}\|u_{k+1}-u_k\|^2_{\Hs}.
\end{equation}
Using \eqref{eq:B66} and \eqref{eq:B67},  we can deduce that 
\begin{equation}
\label{eq:B50}
( u_{k+1}-u_k , u_{k+1}-u^*)_\Hs \leq \frac{\LipcostrD}{4 \alpha_k }\|u_{k+1}-u_k\|^2_{\Hs}.
 \end{equation}
Further,  using  \eqref{eq:B50} and the fact that 
\begin{equation*}
(w-v,w-z)_\Hs = \frac{1}{2}\| w-v\|^2_{\Hs}- \frac{1}{2}\| v-z\|^2_{\Hs}+\frac{1}{2} \| w-z\|^2_{\Hs}    \quad \text{ for all }   z,w,v \in \Hs,
\end{equation*}
we can deduce that
\begin{equation}
\label{eq:B55}
\begin{split}
\frac{1}{2}\| u_{k+1}-u^* \|^2_{\Hs} - \frac{1}{2}\| u_{k}-u^* \|^2_{\Hs}&\leq   \left(\frac{\LipcostrD}{4\alpha_k}- \frac{1}{2} \right)\|u_{k+1}-u_k\|^2_{\Hs}  \leq \frac{\LipcostrD}{4\ssizelb^3}\|\gmap_{\ssizek}(u_{k})\|_\Hs^2,
\end{split}
\end{equation} 
 where it can be seen from \eqref{eq:B51} and \eqref{eq:B52} that 
\begin{equation*}
 \sum^{\infty}_{k=\mmax+1}  \|\gmap_{\ssizek}(u_{k})\|_\Hs^2 \leq (\mmax+1)C^{4\mmax+2}_G  \sum^{\infty}_{k=1}  \| \gmap_{\ssize_{\nu \left(k\right)-1 }}(u_{\nu \left(k\right)-1 })\|_\Hs < \infty.
\end{equation*}
Therefore,   the sequence  $\{ u_k\}_k \subset \Hs$  is quasi-Fej\'er  monotone with respect to  $\crit$ and since,   due to \rom{1},   every sequential weak accumulation point of $\{ u_k\}_k \subset \Hs$ belongs to $\crit \neq \emptyset$,   we can conclude by  \cite[Proposition 1(3)]{iusem_convergence_2003}  that  $\{ u_k\}_k$ is weakly convergent and it has a unique accumulation point.

In addition,  since $\crit$ is closed and convex,  we can conclude,  due to  \cite[Proposition 3.6 \rom{4}]{COMBETTES2001115},  that  $ \{ P_{\crit} u_n  \}_n$  converges strongly to a point $\hat{u} \in \crit$.   Moreover,  since  $u^*- P_{\crit}u_n  \to  u^* - \hat{u}$  and   $u_n -P_{\crit}u_n \rightharpoonup  u^* - \hat{u}$,   it follows from the definition of orthogonal projection  that   $ \|u^* -\hat{u} \|^2_{\Hs} = \lim_{k \to \infty} (u^* - P_{\crit}u_n,u_n - P_{\crit}u_n)_\Hs \leq 0$.   Hence,  we obtain that $u^* = \hat{u}$.

\rom{3}  The proof of this part is inspired by the one in \cite[Theorem 3.2.]{MR2792408}.   First,  we show that  $\cost(u_{k}) \to \cost^*$.  Due to \rom{2},   $u_k  \rightharpoonup u^*$ for some $u^* \in \crit$.   As in the proof of \rom{1},  there exists a sequence  $w_{k} \to 0$ with $w_{k} \in \partial \cost(u_{k})$.  Therefore,  we can write 
\begin{equation}
\label{eq:B45}
\cost(u_{k}) \leq   \cost^* + (w_{k}, u_{k}-u^* )_\Hs     \quad  \text{ for every } k\in \mathbb{N}.
 \end{equation}  
Sending  $k\to  \infty$ in \eqref{eq:B45} and using the facts that  $u_{k}  \rightharpoonup u^*$  and $w_{k} \to 0$ and the sequential weak  lower semicontinuity of $\cost$,
we can conclude that 
\begin{equation*}
\cost^* \leq \liminf_{k \to \infty} \cost(u_{k}) \leq \limsup_{k \to \infty} \cost(u_{k})  \leq \cost^*.
\end{equation*}
Hence, $\cost(u_{k}) \to \cost^*$.

Next, we turn to the verification of \eqref{eq:B42} for a large enough $k$.  Due to \eqref{eq:B44} of Lemma \ref{lem:quadratic},  for every $\lambda \in [0,1]$,  it holds
\begin{equation}
\label{eq:B40}
\begin{split}
\mathcal{E}_{\nu(k)} \leq (1-\lambda)\mathcal{E}_{\nu(k-1)}  +  \frac{\ssizebar \lambda^2}{2}  \dist^2(u_{\nu(k)-1}, \crit)  +\frac{\tilde{C}}{ \ssize_{\nu(k)-1}}  \|\gmap_{\ssize_{\nu(k)-1}}(u_{\nu(k)-1})\|_\Hs^2.
\end{split}
\end{equation}
Since $\{u_k\}_k$ is quasi-F\'ejer monotone with respect to $\crit \neq \emptyset$,  due to \cite[Proposition 3.6 \rom{2}]{COMBETTES2001115},  the sequence $\dist^2(u_{k}, \crit)$ is convergent.  Therefore,  we have  
\begin{equation}
\label{eq:B41}
 \dist^2(u_{\nu(k)-1}, \crit)  \leq \kappa,   \quad  \text{ for all  }  k\geq  1,
\end{equation}
for a positive constant  $\kappa >0$.  Further, using \eqref{eq:B40}  and  \ref{L2},  we can write 
\begin{equation}
\label{eq:B38}
\begin{split}
&\mathcal{E}_{\nu(k)}  \leq (1-\lambda)\mathcal{E}_{\nu(k-1)}+ \frac{\ssizebar \lambda^2 \kappa}{2}  +\frac{\tilde{C}}{ \delta }  \left( \mathcal{E}_{\nu(k-1)} - \mathcal{E}_{\nu(k)} \right).
\end{split}
\end{equation}

The expression on the right hand side is strictly convex in $\lambda$,  since
\begin{equation*}
\frac{\mathrm{d^2}}{\mathrm{d} \lambda^2} \left((1 - \lambda) \mathcal{E}_{\nu(k-1)} + \frac{\ssizebar \lambda^2 \kappa}{2}\right) = \ssizebar \kappa > 0.
\end{equation*}
Thus, it possesses the unique minimizer $\lambda = \frac{\mathcal{E}_{\nu(k-1)}}{\ssizebar \kappa}$.  Since  $\{\mathcal{E}_{\nu(k)}\}_k \to 0$,  this implies that for large enough $k$, we can set $\lambda = \frac{\mathcal{E}_{\nu(k-1)}}{\ssizebar \kappa} \leq 1$ and obtain in \eqref{eq:B38} that
\begin{equation*}
\begin{split}
&\mathcal{E}_{\nu(k)}  \leq \mathcal{E}_{\nu(k-1)}-\frac{\mathcal{E}^2_{\nu(k-1)}}{2\ssizebar\kappa}  +\frac{\tilde{C}}{ \delta } \left( \mathcal{E}_{\nu(k-1)} - \mathcal{E}_{\nu(k)} \right).
\end{split}
\end{equation*}
Together with the fact that $\mathcal{E}_{\nu(k)}  \leq \mathcal{E}_{\nu(k-1)}$,  we can write  
\begin{equation*}
\begin{split}
&\mathcal{E}_{\nu(k)}  \leq \mathcal{E}_{\nu(k-1)}- \frac{\mathcal{E}_{\nu(k-1)}\mathcal{E}_{\nu(k)}}{2\ssizebar\kappa}  +\frac{\tilde{C}}{ \delta } \left( \mathcal{E}_{\nu(k-1)} - \mathcal{E}_{\nu(k)} \right)
\end{split}
\end{equation*}
and, thus, obtain
\begin{equation}
\label{eq:B39}
\begin{split}
&\mathcal{E}_{\nu(k)}   \leq\left(1+   \frac{ \delta}{2\ssizebar\kappa(\tilde{C} +\delta)} \mathcal{E}_{\nu(k-1)}\right)^{-1} \mathcal{E}_{\nu(k-1)}.
\end{split}
\end{equation}
Further,  \eqref{eq:B39} can be expressed as
\begin{equation*}
\begin{split}
\frac{1}{\mathcal{E}_{\nu(k)}} \geq \frac{1}{\mathcal{E}_{\nu(k-1)}}+ \frac{ \delta}{2\ssizebar\kappa(\tilde{C} +\delta)}.
\end{split}
\end{equation*}
Applying this inequality recursively  for integers $k_1$ and $k_2$ with $k_2 \geq k_1$ and large enough $k_1$ satisfying  $\frac{\mathcal{E}_{\nu(k_1)}}{\ssizebar\kappa} \leq 1$,  we obtain 
\begin{equation*}
\begin{split}
\frac{1}{\mathcal{E}_{\nu(k_2)}} \geq \frac{1}{\mathcal{E}_{\nu(k_1)}}+ \frac{ \delta (k_2-k_1) }{2\ssizebar\kappa(\tilde{C} +\delta)},
\end{split}
\end{equation*}
 and by easy computations,   also
 \begin{equation*}
\begin{split}
\mathcal{E}_{\nu(k_2)} \leq  \frac{2\ssizebar\kappa(\tilde{C} +\delta)\mathcal{E}_{\nu(k_1)}}{2\ssizebar\kappa(\tilde{C} +\delta)+\mathcal{E}_{\nu(k_1)} \delta (k_2-k_1) }.
\end{split}
 \end{equation*}
Then for $k\geq 0$  large enough, we set $k_2 = \ceil{ \frac{k}{\mmax+1}}$ and obtain
\begin{equation*}
\begin{split}
\mathcal{E}_{k} &\leq \mathcal{E}_{\nu(\ceil{ \frac{k}{\mmax+1}})}  \leq  \frac{2\ssizebar\kappa(\tilde{C} +\delta)\mathcal{E}_{\nu(k_1)}}{2\ssizebar\kappa(\tilde{C} +\delta)+\mathcal{E}_{\nu(k_1)} \delta (\ceil{ \frac{k}{\mmax+1}}-k_1) }\\ 
&\leq \frac{2\ssizebar\kappa(\tilde{C} +\delta)\mathcal{E}_{\nu(k_1)}}{-k_1 \mathcal{E}_{\nu(k_1)} \delta +2\ssizebar\kappa(\tilde{C} +\delta)+\mathcal{E}_{\nu(k_1)} \delta (\frac{k}{\mmax+1}) }.
\end{split}
 \end{equation*}
Therefore,   \eqref{eq:B42}  follows by setting 
\begin{equation*}
\rho_1 := (\mmax+1)\delta^{-1}2\ssizebar\kappa(\tilde{C} +\delta)
\end{equation*}
and 
\begin{equation*}
\rho_2:=(\mmax+1)(\delta\mathcal{E}_{\nu(k_1)})^{-1}\left(2\ssizebar\kappa(\tilde{C} +\delta)-k_1 \mathcal{E}_{\nu(k_1)} \delta  \right), 
\end{equation*}
thus,  the proof is complete.
\end{proof}
Comparing Theorem \ref{Thm:glabal_conver_con} to Theorem \ref{Thm:conv}, one obtains weak convergence of the whole sequence and convergence the associated cost function evaluations with rate $1 / k$.
\subsection{Convergence under quadratic growth conditions}\label{Quadratic}
In this subsection, we turn our attention towards quadratic growth type conditions and study convergence of  Algorithm \ref{algo:NMLS} under these conditions. 

\begin{definition}[Quadratic Growth Condition] We say that  $\cost$ satisfies the \textit{quadratic growth condition}, if
 \begin{equation}
\label{eq:B46}
\cost(u)-\cost^* \geq  \quadconst  \dist^2(u, \crit)    \quad  \text{ for all  }     u \in \quadset \cap \dom \cost 
\end{equation}
holds, with a set $\quadset \subset \Hs$,  a constant $\quadconst > 0$, and $\crit \neq \emptyset$.  We refer to this notion as global if $\quadset = \Hs$,  and as local, if for $u^* \in \crit$,  $r \in (0, \infty]$, and $\omega>0$ we have $\quadset = \mathbf{B}_{r} (u^*)\cap [ \cost^* < \cost +\omega ]$.   Additionally,  $\cost$ is said to satisfy the \textit{strong quadratic growth condition} at $u^*$ if $\crit = \{ u^* \}$ on $\quadset$.  That is,
\begin{equation}
\label{eq:B26}
\cost(u)-\cost(u^*) \geq  \quadconst \|u-u^*\|^2_\Hs    \quad \text{ for all  }  u \in \quadset \cap \dom \cost.
\end{equation} 
 \end{definition}
The quadratic growth condition is a geometrical
assumption which describes the \textit{flatness} of the objective function around its minimizers.  Roughly speaking,  this condition is considered as a relaxation of the strong convexity condition and allows us to   
obtain faster rates of convergence (linear) and also convergence in the strong topology for the iteration sequence.   It is also closely  related  to the notion of Tikhonov well-posedness \cite{zbMATH00422112}. The relationship between the quadratic growth condition and the so-called \textit{metric subregularity of the subdifferential} has been investigated  e.g.  in \cite{artacho2008characterization,zbMATH06285806,aze_nonlinear_2014,MR3707370,zbMATH06409515}. 
The strong quadratic growth condition \eqref{eq:B26} is said to be the \textit{quadratic functional growth property} in \cite{necoara_linear_2019} provided that  $\cost$ is  continuously differentiable over a closed convex set.  In \cite{garrigos_thresholding_2020,garrigos_convergence_2022}, $\cost$ is also called \textit{$2$-conditional on $\quadset$} if it satisfies the quadratic growth condition \eqref{eq:B46}.  This property was recently proved in \cite[Theorem 5]{MR3707370} to be equivalent with the case where $\cost$ satisfies the Kurdyka-Łojasiewicz inequality with order $1/2$.

\begin{theorem}
\label{Thm: QuadGrowth}
Suppose that Assumption \ref{ass:convex}  and the quadratic growth condition \eqref{eq:B46} hold for $\quadset :=[ \cost  < \cost^* +\omega ] $ with $\omega > 0$.  Then,  for the  sequence of iterates $\{ u_k \}_k$ generated by Algorithm \ref{algo:NMLS},  there exists  $\bar{k} \in \N$ such that for every $k \geq \bar{k}$  it holds
\begin{equation}
\label{eq:B48}
\cost(u_{k}) - \cost^*  \leq C_{c} \sigma^{k},
\end{equation} 
and 
\begin{equation}
\label{eq:B53}
\dist^2(u_k,\crit)  \leq  C_{d} \sigma^{k},
\end{equation}
where the constants  $C_{c}$,  $C_{d} > 0$,  and  $0 < \sigma <1$ are independent of $u_0$,  $u^*$,  and $k$. 

Further, there exists  $u^*\in  \crit$  such that $u_{k} \to  u^*$ and we have 
\begin{equation}
\label{eq:B54}
\| u_k-u^* \|_{\Hs}^2 \leq    C_{p}  \sigma^{k}, 
\end{equation}
with a constant $C_{p} > 0$ which is independent of $u_0$,  $u^*$,  and $k$. 
\end{theorem}

\begin{proof}
First,  due to \ref{L2} and  \rom{3} from Theorem \ref{Thm:glabal_conver_con},  the sequence $\{ \cost(u_{\nu(k)})\}_k$  is monotonically decreasing and converges to $\cost^*$.   Further,  using  \eqref{eq:B30},  we can deduce for $k\geq 1$  that  
\begin{equation}
\label{eq:B75}
\cost(u_{\nu(k)-1}) \leq   \cost( u_{\ell(\nu(k)-1)})  \leq  \cost( u_{\nu(k-1)}).
\end{equation}
Thus,  for given $\omega >0$,  there exists $\bar{k}_{\omega} \in \N $ such that 
\begin{equation*}
\cost(u_{\nu(k)-1)}) \in [ \cost  < \cost^* +\omega ]   \quad  \text{ for all }   k\geq \bar{k}_{\omega}.
\end{equation*}
Next, we show that 
\begin{equation}
\label{eq:B21}
\mathcal{E}_{\nu(k)}  \leq \theta \mathcal{E}_{\nu(k-1)}   \quad \text{ for all }  k \geq \bar{k}_{\omega} 
\end{equation}  
with  a constant  $\theta \in (0,1) $ independent of $k$.    

Let an arbitrary  $ k \geq \bar{k}_{\omega}$  be given.  To show \eqref{eq:B21}, we choose an arbitrary $\zeta$ with  
\begin{equation*}
0<  \zeta < \min \{  \frac{1}{\delta},   \frac{1}{2\tilde{C}}, \frac{\quadconst}{2\ssizebar \tilde{C}}\}. 
\end{equation*}
Now we consider the following cases:
 \begin{itemize}
\item The inequality $\frac{1}{\ssize_{\nu(k)-1}} \|\gmap_{\nu(k)-1}(u_{\nu(k)-1})\|_\Hs^2 \geq \zeta \mathcal{E}_{\nu(k-1)}   $ holds. 
 In this case,  using \ref{L2}, we obtain
 \begin{equation*}
 \begin{split}
 \mathcal{E}_{\nu(k)}   & \leq \mathcal{E}_{\nu(k-1)}  - \frac{\delta}{\ssize_{\nu(k)-1}} \|\gmap_{\nu(k)-1}(u_{\nu(k)-1})\|_\Hs^2\leq (1-\zeta\delta) \mathcal{E}_{\nu(k-1)},
 \end{split}
 \end{equation*}
 where due to the choice of $\zeta$ we have $ 1-\zeta\delta<1$.
  \item The inequality 
 \begin{equation}
 \label{eq:B27}
  \frac{1}{\ssize_{\nu(k)-1}} \|\gmap_{\nu(k)-1}(u_{\nu(k)-1})\|_\Hs^2 < \zeta \mathcal{E}_{\nu(k-1)}  
 \end{equation}
holds.   This case is more delicate.   Using  the quadratic growth condition \eqref{eq:B46} and \eqref{eq:B75}, we obtain  for $k\geq \bar{k}_{\omega}$  that 
 \begin{equation}
  \label{eq:B28}
 \begin{split}
  \dist^2(u_{\nu(k)-1} , \crit)  \leq \frac{1}{\quadconst} \mathcal{E}_{\nu(k)-1}   \leq  \frac{1}{\quadconst} \mathcal{E}_{\nu(k-1)}.
 \end{split}
 \end{equation}
Now,  using Lemma \ref{lem:quadratic},   \eqref{eq:B44},  and  \eqref{eq:B28},  we can write for every $\lambda \in [0,1]$ that  
\begin{equation}
\label{eq:B29}
\begin{split}
\mathcal{E}_{\nu(k)} \leq \left( \tilde{C}\zeta +  1 -\lambda +\frac{ \ssizebar}{2 \quadconst} \lambda^2 \right) \mathcal{E}_{\nu(k-1)} ,
\end{split}
\end{equation}
where it can be easily seen that the minimum of $ \left( \tilde{C}\zeta+  1 -\lambda + \frac{ \ssizebar}{2\quadconst} \lambda^2 \right)$ attained at $ \lambda^*: =  \min \{ 1, \frac{\quadconst}{\ssizebar} \}$ is strictly smaller than $1$ and therefore \eqref{eq:B21} holds for  $k \geq \bar{k}_{\omega}$.
 \end{itemize}

Let now $k\geq \bar{k}:=\bar{k}_{\omega} (\mmax+1)$ be given.  By successively applying \eqref{eq:B21},  we obtain 
\begin{equation}
\label{eq:B33}
\begin{split}
\mathcal{E}_{k} &  \stackrel{\text{\ref{L1}}}{\leq}  \mathcal{E}_{\nu(\ceil{\frac{k}{\mmax+1}})} \leq   \theta^{1- \bar{k}_{\omega}}\theta^{\ceil{\frac{k}{\mmax+1}}} \mathcal{E}_{\nu( \bar{k}_{\omega}-1)} \\ &   \leq \theta^{1- \bar{k}_{\omega}}\theta^{{\frac{k}{\mmax+1}}} \mathcal{E}_{\nu(\bar{k}_{\omega}-1)}  \leq \theta^{1- \bar{k}_{\omega}}\left(\theta^{{\frac{1}{\mmax+1}}} \right)^k \mathcal{E}_{\nu( \bar{k}_{\omega}-1)} .
\end{split} 
\end{equation}
Together with the quadratic growth condition \eqref{eq:B26},  we have 
\begin{equation}
\label{eq:B34}
\begin{split}
\dist^2( u_k,\crit) &\leq \frac{1}{\quadconst} \mathcal{E}_{k} \leq  \frac{\theta^{1-\bar{k}_{\omega}}}{\quadconst} \left(\theta^{{\frac{1}{\mmax+1}}} \right)^k  \mathcal{E}_{\nu(\bar{k}_{\omega}-1)}.
\end{split} 
\end{equation} 
Thus,  for  every  $k \geq  \bar{k}$, the  iterate  $u_k$ stays in  $\quadset$.  Setting $C_{c}:=\theta^{1-\bar{k}_{\omega}} \mathcal{E}_{\nu(\bar{k}_{\omega}-1)}$,   $C_{d}:= \quadconst^{-1} \theta^{1-\bar{k}_{\omega}} \mathcal{E}_{\nu(\bar{k}_{\omega}-1)}$,    and   $\sigma := \theta^{{\frac{1}{\mmax+1}}} <1$, we are finished with the verification of \eqref{eq:B48}  and \eqref{eq:B53}.  
  
Next,  we show that $u_k \to u^*$ for $u_k \rightharpoonup u^*$  with  $u^* \in \crit$ given in Theorem \ref{Thm:glabal_conver_con} \rom{2}.  Due to \eqref{eq:B53},  $\dist(u_k,\crit) \to 0$ and we can infer that  $u_k-P_{\crit}u_k  \to 0$.  This together with fact that  $P_{\crit}u_k \to u^*$ (see $(ii)$ in Theorem  \ref{Thm:glabal_conver_con}) leads to $u_k \to u^*$. 

Finally,  we show that \eqref{eq:B54} holds true.  To see this,  let $p \in \N$ be arbitrary, using Young's inequality we have
\begin{equation}
\label{eq:B56}
\begin{split}
\| u_k- u_{k+p}\|^2_{\Hs} & \leq 2\left( \|  u_k-P_{\crit}u_k\|^2_{\Hs}+\|u_{k+p}- P_{\crit}u_k \|_\Hs^2  \right) \\
& = 2\left( \dist^2(u_k, \crit)+\|u_{k+p}- P_{\crit}u_k \|_\Hs^2  \right).
\end{split}
\end{equation}
 Using \eqref{eq:B55},  we obtain for an arbitrary $u^* \in \crit$ that 
 \begin{equation*}
 \begin{split}
\| u_{k+p}-u^* \|^2_{\Hs} \leq \| u_{k}-u^* \|^2_{\Hs} +  \frac{\LipcostrD}{2\ssizelb^2} \sum_{j =k}^{k+p-1}  \|\gmap_{\ssize_j}(u_{j})\|_\Hs^2.
\end{split}
 \end{equation*}
In particular,  we have for $u^* = P_{\crit}u_k$ that 
\begin{equation}
\label{eq:B57}
 \begin{split}
\| u_{k+p}-P_{\crit}u_k \|^2_{\Hs}  \leq  \dist^2( u_{k}, \crit)+ \frac{\LipcostrD}{2\ssizelb^2} \sum_{j =k}^{k+p-1}  \|\gmap_{\alpha_j}(u_{j})\|_\Hs^2.
\end{split}
 \end{equation}
Further,  using \ref{L2} and \eqref{eq:B52},  we can write for  large enough $k$ that   
\begin{equation}
\label{eq:B68}
\begin{split}
 &\sum_{j =k}^{k+p-1}  \|\gmap_{\alpha_j}(u_{j})\|_\Hs^2 \leq C^{4\mmax+2}_G  \sum_{j =k}^{k+p-1} \|\gmap_{\ssize_{\nu \left( \floor{ \frac{j}{\mmax+1} }\right)-1 }}(u_{\nu \left(\floor{ \frac{j}{\mmax+1} }\right)-1 })\|^2_\Hs \\ & \leq   C^{4\mmax+2}_G \ssizebar \delta^{-1} \sum_{j =k}^{k+p-1}  \frac{\delta}{\alpha_{\nu \left(\floor{ \frac{j}{\mmax+1} }\right)-1 }}\|\gmap_{\ssize_{\nu \left( \floor{ \frac{j}{\mmax+1} }\right)-1 }}(u_{\nu \left(\floor{ \frac{j}{\mmax+1} }\right)-1 })\|^2_\Hs \\ & \leq  C^{4\mmax+2}_G \ssizebar \delta^{-1} \left(  \mathcal{E}_{\nu \left( \floor{ \frac{k}{\mmax+1} }-1 \right)} -  \mathcal{E}_{\nu \left(  \floor{ \frac{k+p}{\mmax+1} }  \right)}     \right).
 \end{split}
\end{equation}
Combining  \eqref{eq:B56},  \eqref{eq:B57},  \eqref{eq:B68},  and setting   $ \bar{C}_p: = \frac{C^{4\mmax+2}_G \ssizebar\LipcostrD}{\delta\ssizelb^2}$,  we arrive at 
\begin{equation*}
 \| u_k- u_{k+p}\|^2_{\Hs} \leq  4 \dist^2(u_k, \crit)+\bar{C}_p \left(  \mathcal{E}_{\nu \left( \floor{ \frac{k}{\mmax+1} }-1 \right)} -  \mathcal{E}_{\nu \left(  \floor{ \frac{k+p}{\mmax+1} }  \right)}  \right).
\end{equation*}
Sending $p\to \infty$  and using Theorem \ref{Thm:glabal_conver_con} \rom{3}, \eqref{eq:B53},  and  similar computations as in \eqref{eq:B33},  we obtain for every $k\geq \bar{k}$ that
\begin{equation*}
\begin{split}
 \| u_k- u^*\|^2_{\Hs} \leq  4 \dist^2(u_k, \crit)+\bar{C}_p  \mathcal{E}_{\nu \left( \floor{ \frac{k}{\mmax+1} }-1 \right)} \leq 4C_d \sigma^k+ \bar{C}_p\theta^{-1- \bar{k}_{\omega}}\sigma^k \mathcal{E}_{\nu( \bar{k}_{\omega}-1)} = C_p \sigma^k.
\end{split}
\end{equation*}
Thus,   \eqref{eq:B54} holds true with  $C_p :=4 C_d +\bar{C}_p\theta^{-1- \bar{k}_{\omega}}\mathcal{E}_{\nu( \bar{k}_{\omega}-1)}$ and this completes the proof.
\end{proof}

\begin{corollary}
\label{cor:quadGrowthLocal}
Suppose that Assumption \ref{ass:convex}  holds and the quadratic growth  condition is satisfied for $\quadset \coloneqq \mathbf{B}_{r} (u^*)\cap [ \cost  < \cost^* +\omega ] $ with  $r , \omega \in  (0,\infty)$ and  $\{u_k\}_k$  generated by Algorithm \ref{algo:NMLS},  converges  strongly to some $u^* \in \crit $.   Then there exists  $\bar{k} \in \N$ such that \eqref{eq:B48}-\eqref{eq:B54} hold with constants $C_{c}, C_{d}, C_p > 0$,  and  $0 < \sigma<1$, which are independent of $u_0$,  $u^*$,  and $k$. 
\end{corollary}
\begin{proof}
The proof proceeds along the lines of the proof of Theorem \ref{Thm: QuadGrowth} with the  difference that  one also needs to be sure that $\{u_k \}_k \subset \mathbf{B}_{r} (u^*)$ for a large enough  $\bar{k} \in \mathbb{N}$.  This follows from the fact that  $u_k \to u^*$.
\end{proof}

\begin{remark}
Due to the equivalence of weak and strong convergence in finite-dimensional spaces,  the assumption of  Corollary  \eqref{cor:quadGrowthLocal} automatically holds if $\dim (\Hs) < \infty$.   For the case that $\dim (\Hs) = \infty$,  it is not clear how to guarantee that $\{u_k \}_k \subset  \mathbf{B}_{r}(u^*)$ for large enough $k \in \N$.
\end{remark}

In many cases of PDE-constrained optimization, the assumptions of Corollary \ref{cor:quadGrowthLocal} or the following corollary are very likely satisfied.

\begin{corollary}
\label{cor:quadGrowthStrong}
Suppose that  \ref{ass:general_1}-\ref{ass:general_3} in Assumption \ref{ass:general} hold and that $\costr$ is convex on $\mathbf{B}_{r} (u^*)$ with $u^* \in \crit$ and $r \in (0, \infty)$.  Further,  we assume that the strong quadratic growth condition holds for $\quadset \coloneqq \mathbf{B}_{r} (u^*)\cap [ \cost  < \cost^* +\omega ] $ with $\omega \in  (0,\infty)$.  Then,  the sequence of iterations $\{ u_k \}_k$ generated by Algorithm \ref{algo:NMLS} converges locally R-linear with respect to the strong topology.   In other words,  there exists a radius $r_0 \leq r$  such that for every $u_0 \in  \mathbf{B}_{r_0} (u^*)\cap [ \cost  < \cost^* +\omega ] $ and $k\geq 1$,  we have 
\begin{equation} 
\label{eq:B35}
\| u_k-u^* \|^2_\Hs \leq C_{R} \sigma^k \| u_0-u^* \|^2_\Hs  
\end{equation} 
and 
\begin{equation}
\label{eq:B36}
\cost(u_{k}) - \cost(u^*)  \leq  \sigma^k\left( \cost(u_{0}) - \cost^* \right),
\end{equation} 
where the constants $0 < C_{R}$ and  $0 < \sigma < 1$ are independent of $u_0$,  $u^*$,  and $k$. 
\end{corollary}
\begin{proof}
By similar argument as in the proof of Theorem \ref{Thm: QuadGrowth},  we can show that for any  $u_{\nu(k)-1} \in  \mathbf{B}_{r} (u^*)\cap [ \cost  < \cost^* +\omega ]$ with $k \geq 1$,    it holds 
\begin{equation}
\label{eq:B21b}
\mathcal{E}_{\nu(k)}  \leq \theta \mathcal{E}_{\nu(k-1)},  
\end{equation}  
with  $\theta \in (0,1) $ independent of $k$.    

Let now $k\geq 1$ be given.  Assuming  $u_i \in \mathbf{B}_{r} (u^*)\cap [ \cost  < \cost^* +\omega ]$ for $i=0,\dots,k-1$  and successively applying \eqref{eq:B21b},  we obtain 
\begin{equation}
\label{eq:B33b}
\begin{split}
\mathcal{E}_{k} &  \stackrel{\text{\ref{L2}}}{\leq} \mathcal{E}_{\nu(\ceil{\frac{k}{\mmax+1}})} \leq   \theta^{-1}\theta^{\ceil{\frac{k}{\mmax+1}}} \mathcal{E}_{\nu(1)}   \leq \theta^{-1}\theta^{{\frac{k}{\mmax+1}}} \mathcal{E}_{\nu(1)}  \leq \theta^{-1} \left(\theta^{{\frac{1}{\mmax+1}}} \right)^k \mathcal{E}_{\nu(1)} .
\end{split} 
\end{equation}
Together with the quadratic growth condition \eqref{eq:B26} and \eqref{eq:B31}  from Lemma \ref{lem:quadratic},  we have 
\begin{equation}
\label{eq:B34b}
\begin{split}
\| u_k-u^* \|^2_H &\leq \frac{1}{\quadconst} \mathcal{E}_{k} \leq  \frac{1}{\theta\quadconst} \left(\theta^{{\frac{1}{\mmax+1}}} \right)^k  \mathcal{E}_{\nu(1)} \leq \frac{C_0}{\theta\quadconst} \left(\theta^{{\frac{1}{\mmax+1}}} \right)^k \| u_0-u^* \|^2_H.
\end{split} 
\end{equation} 
Thus,  for  every  $u_0 \in \mathbf{B}_{r_0} (u^*)$ with $r_0 \leq \left( \frac{\theta\quadconst}{C_0}\right)^{\frac{1}{2}}r$,   iterates  $u_k$ stay in  $\mathbf{B}_{r} (u^*)$ for $k\geq 0$  and,  as a consequence,  the inequalities \eqref{eq:B33b}  and \eqref{eq:B34b}  are well-defined. By setting $C_{R}:=\frac{C_0}{\theta\quadconst} $  and   $\sigma := \theta^{{\frac{1}{\mmax+1}}} <1  $ we are finished with the verification of \eqref{eq:B35}.  
  
As in  \eqref{eq:B33b},  we can also write for every  $k\geq 0$ and  $u_0 \in \mathbf{B}_{r_0} (u^*)\cap [ \cost  < \cost^* +\omega ] $ that 
\begin{equation*}
\begin{split}
\mathcal{E}_{k} &\leq  \mathcal{E}_{\nu(\ceil{\frac{k}{\mmax+1}})}  \leq \theta^{\ceil{\frac{k}{\mmax+1}}}\mathcal{E}_{\nu(0)} \leq \theta^{{\frac{k}{\mmax+1}}}  \mathcal{E}_{\nu(0)}  \leq  \sigma^k \mathcal{E}_{\nu(0)} . 
\end{split} 
\end{equation*}
Thus \eqref{eq:B36} holds also true and this completes the proof.
\end{proof}

Compared to Corollary \ref{cor:quadGrowthLocal},  we do not need to assume strong convergence of the sequence $\{u_k\}_k$ in   Corollary \ref{cor:quadGrowthStrong}  and the following corollary.
\begin{corollary}
\label{cor:StrongCon}
 Suppose that  \ref{ass:general_1}-\ref{ass:general_3} in Assumption \ref{ass:general} hold and  $\cost$ is locally strongly convex  with a constant $\kappa >0$ on $\mathbf{B}_{r} (u^*)$ with some   $r \in (0, \infty)$ and $u^* \in \crit$.   Then,  the sequence of iterates $\{ u_k \}_k$ generated by Algorithm \ref{algo:NMLS} converges locally R-linear in  the strong topology  to $u^*$.   In other words,  there exists  $r_0 \leq r$ such that  \eqref{eq:B35} and \eqref{eq:B36} hold for every $u_0 \in  \mathbf{B}_{r_0} (u^*)$.  
\end{corollary}
\begin{proof} Since $\cost$ is locally strongly convex,   using  \cite[Proposition 3.23]{MR3310025},  we can write 
\begin{equation}
\label{eq:B58}
\cost(u)-\cost(v) \geq ( w, u-v)_\Hs + \frac{\kappa}{2} \|u-v\|_\Hs^2  \quad  \text{ for all }  u, v \in \mathbf{B}_{r} (u^*)  \text{ and }  w \in \partial \cost(v).
\end{equation}
 Setting $v= u^*$ and $w = 0$,  we can easily see that the strong quadratic growth property \eqref{eq:B26} holds for $\quadset \coloneqq \mathbf{B}_{r} (u^*)$ and $\quadconst \coloneqq \frac{\kappa}{2}$.   Then,  the local R-linear convergence follows by similar arguments as in the proof of Corollary \ref{cor:quadGrowthStrong}.  
\end{proof}

\begin{remark}
Note that, if the strong convexity in Corollary \ref{cor:StrongCon} holds globally,  then R-linear convergence is also obtained globally.  More precisely,  \eqref{eq:B35} and \eqref{eq:B36} hold for every $u_0 \in  \Hs$. Furthermore if $\costr$ is locally strongly convex, then also $\cost$ is.
\end{remark}

\subsection{On relaxing the Lipschitz continuity of  $\costrD$ for problems governed by PDEs}
In this section,  we analyze a list of assumptions satisfied for a large class of optimization problems governed by PDEs.  We will study the applicability of this assumptions in Section \ref{sec:PDE}.
\begin{assumption}
\label{ass:general3}
For problem  \eqref{eq:opt_problem}, 
\begin{description}
\item[\namedlabel{ass2:general_1}{H1}:]   $\cost \colon \Hs \rightarrow \R \cup \{\pm \infty\}$ is bounded from below and radially unbounded,  i.e. $\lim_{\|u\|_{\Hs} \to \infty} \cost(u)= \infty$.
\item[\namedlabel{ass2:general_2}{H2}:]   $\costp \colon \Hs \rightarrow \R \cup \{\pm \infty\}$ is  proper, convex,  and lower semicontinuous.
\item[\namedlabel{ass2:general_3}{H3}:]   $\costr \colon \Hs \rightarrow \R$ is continuously Fr\'echet differentiable on $\domF$ containing  $\dom \costp$, that is,  $\dom \costp \subseteq \domF$. 
\item[\namedlabel{ass2:general_4}{H4}:] $\costrD \colon \Hs \rightarrow \Hs$  is $\LipcostrD$-Lipschitz continuous on every sequentially weakly compact subset of $\dom \costp$.
\item[\namedlabel{ass2:general_5}{H5}:] $\costp$  is weakly sequentially lower semicontinuous (wlsc) and $\costrD \colon \domF  \rightarrow \Hs$ is weak-to-strong sequentially continuous.
\end{description}
\end{assumption}
Here, compared to Assumption \ref{ass:general}, we do not impose a global Lipschitz condition on $\costrD$. Nevertheless, Assumption \ref{ass:general3} will be sufficient to reproduce the results of Section \ref{General case}. Furthermore \ref{ass2:general_1}, \ref{ass2:general_2} and \ref{ass2:general_5} impose a set of conditions that ensure the existence of a global minimizer to \eqref{eq:opt_problem}, as will be shown in the following.

\begin{proposition}
Suppose that \ref{ass2:general_1}, \ref{ass2:general_2}, and \ref{ass2:general_5} hold. Then,  problem \eqref{eq:opt_problem} possesses a global minimizer $\ubar \in \Hs$,  and as a consequence,  $\crit \neq \emptyset$.  Further,  every level set of  $\cost$ is sequentially weakly compact. 
\end{proposition}
\begin{proof}
The proof uses standard arguments based on the direct methods in calculus of variations: By \ref{ass2:general_1},   there exists 
\[\costbar \coloneqq \inf\limits_{u \in \Hs} \cost(u)\]
 and a minimizing sequence $\{u_n\}_{n} \subset \Hs$ with $\cost(u_n) \rightarrow \costbar$ for $n \rightarrow \infty$.   Due to \ref{ass2:general_1} and  \cite[Proposition 11.11]{MR3616647}) every level set of $\cost$ is bounded and, thus,   $\{u_n\}_{n}$ admits a weakly convergent subsequence $(u_{n_k})_{k}$ with $u_{n_k} \rightharpoonup \ubar \in \Hs$ on  
$[ \cost \leq \cost(\tilde{u}) ]$ for some  $\tilde{u} \in  \dom \costp$.  Properties  \ref{ass2:general_2} and \ref{ass2:general_5} imply that $\cost$ is wlsc,  and thus,  $\cost(\ubar) \leq \liminf_{k \in \N} \cost(u_{n_k}) = \costbar$.   This shows that $\cost(\ubar) = \costbar$, i.e. $\ubar \in \Hs$ is a global minimizer of \eqref{eq:opt_problem}.  Hence $\crit \neq \emptyset$. 
Further,  since $\cost$ is wlsc,  every level set of $\cost$ is sequentially weakly closed.  Together,  with the boundedness of the level sets,  we conclude that every level set of  $\cost$ is sequentially weakly compact. 
\end{proof}

Comparing the properties given in Assumptions \ref{ass:general}  and  Assumption \ref{ass:general3} in detail, we realize that besides the global Lipschitz continuity of $\costrD$, the remaining properties of Assumption \ref{ass:general}  follow from those of Assumption  \ref{ass:general3}.  Further,  due to Lemma  \ref{lem:properties_of_iter}\rom{1} and Lemma  \ref{lem:without_lip},   Algorithm \ref{algo:NMLS} is well-defined, even without global Lipschitz continuity of $\costrD$.  Further,  it can be seen from \eqref{eq:B55b} and  Definition \ref{def:prox_gmap},  that for every $u_0 \in \dom \costp$, the whole sequence $\{ u_k \}_k$ generated by Algorithm  \ref{algo:NMLS} stays in $U_0 :=[ \cost \leq \cost(u_0) ]$.   Since $U_0$ is sequentially weakly compact  and $U_0 \subset  \dom \costp$,   using \ref{ass2:general_4} we can deduce that $\costrD$ is $\LipcostrD$-Lipschitz continuous on  $U_0$  with  $\{ u_k \}_k  \subset U_0$.   Therefore all the results and estimates in Section \ref{General case} can be easily applied in the presence of Assumption \ref{ass:general3}.   Further,  similar observations are also valid for the results given in Section \ref{Convex case}, provided that we replace  \ref{ass2:general_5} by \ref{ass:convex_4} and analogously also for Section \ref{Quadratic}.

\section{Application to PDE-constrained optimization}
\label{sec:PDE}
In this section,  we investigate the applicability of our theoretical results to two problems governed by semilinear elliptic and parabolic PDEs.  For each example,  we discuss the justification of Assumption \ref{ass:general3}.

\subsection{Elliptic model problem}
\label{sec:PDE:elliptic}
As a first example,  we consider the following semilinear elliptic model problem
\begin{align}
\tag{\textbf E}
\label{eq:opt_problem_pde_e}
&\min_{(y, u)}\hspace{1.4mm} \frac{1}{2} \|y - \yd \|_{L^2(\Omega)}^2 + \frac{\sigma}{2}\|u\|_{L^2(\Omega)}^2 + \lambda \|u\|_{L^1(\Omega)} \\ 
&\text{subject to }\label{eq:elliptic_state}
\begin{cases} 
 - \kappa \Delta y + \exp(y) = u &  \text{in } \Omega, \\
y = 0 & \text{on } \partial \Omega,\\
\ulb \leq u \leq \uub & \text{a.e.  in } \Omega, 
\end{cases}
\end{align}
on a bounded domain  $\Omega \subset \R^n$ with $ n= 2, 3$, which is either convex or possesses a $C^{1,1}$-boundary $\partial \Omega$.  Here the parameter $\kappa > 0$ stands for the diffusion,   the parameters $\sigma, \lambda > 0$  weigh the cost terms,   $\ulb, \uub \in \R$ with $\ulb < 0 < \uub$ are control bounds,  and  $\yd \in L^2(\Omega)$ denotes the desired state.   First,  we show that problem \eqref{eq:opt_problem_pde_e} can be rewritten in the form \eqref{eq:opt_problem}.  It is well-known, cf.  \cite[Theorems 2.7 and 2.12]{MR4162938},  that for given $u \in \Hs:= L^2(\Omega)$,  the state equation \eqref{eq:elliptic_state} is uniquely solvable in the weak sense,   i.e.,  $y(u) \in \Ws\cap C(\overline{\Omega})$  with  $\Ws := H_0^1(\Omega)$,  and  the control-to-state operator $u \mapsto y(u)$ is well-defined and twice continuously Fr\'echet differentiable from $\Hs$ to $\Ws \cap C(\overline{\Omega})$.  This operator is typically constructed by applying the implicit function theorem on the equality constraints  
\begin{equation*}
 \PDEcon(y, u)= 0  \text{ in }  Z':= H^{-1}(\Omega)   \quad  \text{ with  }  \quad  \PDEcon(y, u) :=  - \kappa \Delta y + \exp(y) - u.
\end{equation*}
Here,  $\PDEcon : W \times H \to Z'$  is twice continuously Fr\'echet differentiable,  for every $u \in \Hs$,  $\PDEcon_y(y, u) \in \mathcal{L}(\Ws, Z')$ has a bounded inverse,  and $\PDEcon_u(y, u) \in \mathcal{L}(\Hs, Z') $ is continuous.  Then,  by defining 
\begin{equation}
\label{eq:B71}
\costr(u) = \frac{1}{2} \|y(u)- \yd \|_{L^2(\Omega)}^2,  \qquad  \costp(u) \coloneqq  \frac{\sigma}{2} \|u\|_{L^2(\Omega)}^2 + \lambda \|u\|_{L^1(\Omega)}+\delta_{\mathcal{U}}(u) ,
\end{equation}
with the indicator function $\delta_{\mathcal{U}}$ of $\mathcal{U} \coloneqq \{ u: \ulb  \leq    u \leq \uub \text{ a.e.  in } \Omega \} \subset \Hs$,  problem \eqref{eq:opt_problem_pde_e} has the form \eqref{eq:opt_problem}.  

Further,  by standard computations as in \cite[Chapter 1.6]{H08},   it can be shown  that  
\begin{equation}
\label{eq:B72}
\costrD(u) = -p(u)    \quad  \text{ in }   \Hs,  
\end{equation}
where $p = p(u) \in W$ (see e.g., \cite[Theorem 3.2]{MR4162938}) is the weak solution of the adjoint problem
\begin{equation}
\label{eq:elliptic_adjoint}
 \begin{cases} 
- \kappa \Delta p + \exp(y) p  = - (y - \yd) & \text{ in } \Omega, \\
p = 0 & \text{ on } \partial \Omega,
\end{cases}
 \end{equation}
with $y = y(u)$. Now it remains to verify Assumption \ref{ass:general3} for the reduced problem \eqref{eq:opt_problem} with \eqref{eq:B71}. Property \ref{ass2:general_1} and \ref{ass2:general_2} are clearly satisfied.  \ref{ass2:general_3} follows using the chain rule and the continuous Fr\'echet differentiability of the control-to-state mapping $y(u)$.  Note that 
\begin{equation}
\label{eq:B70}
\costrD(u) := (y'(u))^* \costyu'(y(u))     \text{ with   }  \costyu(y) = \frac{1}{2} \|y - \yd \|_{L^2(\Omega)}^2,
\end{equation}
where  $y'(u) h = -\PDEcon^{-1}_y(y(u),u)\PDEcon_u(y(u),u) h$ for every $h \in \Hs$ and the superscript "*" denotes the adjoint operator.   It remains to verify \ref{ass2:general_4} and  \ref{ass2:general_5}.  Since  the control-to-state operator $u\to y(u)$ and  $\costyu$ are twice continuously Fr\'echet  differentiable,  using \eqref{eq:B70},  and  the compact embedding $W\xhookrightarrow{c} H $,  one obtains that the second Fr\'echet derivative of $\costr$ is bounded on bounded sets.   
Therefore $\costrD$ is  Lipschitz continuous on any bounded set in $\Hs$.  Thus,  \ref{ass2:general_4} holds.  Finally using the weak formulation of the state \eqref{eq:elliptic_state} and adjoint \eqref{eq:elliptic_adjoint} equation and  \cite[Theorem 2.11]{MR4162938},  it can be shown that  $y(u_k) \rightharpoonup  y(\bar{u})$ and $p(u_k) \rightharpoonup  p(\bar{u})$ in  $\Ws$ as $u_k \rightharpoonup  \bar{u}$ in $\Hs$.  This together with  $W\xhookrightarrow{c} H $,  \eqref{eq:B72},  and \eqref{eq:B71},  implies that 
$\costr(u_k) \to \costr(\bar{u})$ and $\costrD(u_k) \to \costrD(\bar{u})$  in $\Ws$ as  $u_k \rightharpoonup  \bar{u}$ in $\Hs$.  Hence,  \ref{ass2:general_5} holds. This finishes the verification of Assumption \ref{ass:general3}.

Note that using  \eqref{eq:B72} and \eqref{stationary2},  the Fermat principle can be expressed as 
\begin{equation*}
\label{}
u^* = \prox_{\frac{1}{\ssize} \costp}(u^*+\frac{1}{\ssize}p(u^*)) \quad  \text{ for some }  \ssize >0,
\end{equation*}
where  $p(u^*)$ is the solution of  \eqref{eq:elliptic_adjoint} for  $u = u^*$  and $\prox_{\frac{1}{\ssize} \costp}$ can be characterized in a pointwise a.e. sense.  Due to \cite[Proposition 24.13]{MR3616647},  we have for any $\ssize >0$
\begin{equation*}
\left[\prox_{\frac{1}{\ssize} \costp}(u)\right] (x) = \prox_{\frac{1}{\ssize} \costp}(u(x))  \quad  \text{ for almost all  } x\in \Omega,
\end{equation*}
where a  pointwise a.e. closed form representation of the  proximal operator is given by 
\begin{align*}
\left[\prox_{\frac{1}{\ssize} \costp}(u)\right](x) = \min\{\max\{
\begin{cases}
\frac{1}{C_1}(u(x) - C_2), & u(x) > C_2,\\
\frac{1}{C_1}(u(x) + C_2), & u(x) < - C_2,\\
0, & \text{otherwise},
\end{cases}, \ulb\}, \uub\},
\end{align*}
with  $C_1 := 1 + \sigma / \alpha$ and $C_2 := \lambda / \alpha$. The calculations are made similarly to \cite[Section 6]{parikh2014proximal}.

\subsection{Parabolic model problem}
\label{sec:PDE:parabolic}
As a second example,  we will consider the following semilinear parabolic model problem
\begin{align}
\label{eq:opt_problem_pde_p}
\tag{\textbf P}
&\min_{(y, u)}\hspace{1.4mm} \frac{1}{2} \|y - \yd \|_{L^2(0,T;L^2(\Omega))}^2 + \lambda \|u\|_{L^1(0,T;L^1(\Omega))} \\ 
&\text{subject to }
\begin{cases} 
\dot{y} - \kappa \Delta y + y^3 = u   & \text{in } (0, T) \times \Omega, \\
y = 0 & \text{on } (0, T) \times  \partial \Omega,\\
y(0)=y_0  & \text{in }  \Omega,\\
\ulb \leq u \leq \uub & \text{a.e.  in } \Omega,
\end{cases}
\end{align}
where  $\Omega \subset \R^n$ with $n =2,3 $ is a bounded domain with Lipschitz boundary $\partial \Omega$ and $T > 0$.  Further,  $\kappa > 0$,  $\lambda \geq 0$,  and  the control bounds  $\ulb, \uub \in \R$ are defined as in problem  \eqref{eq:opt_problem_pde_e}.  We also consider the desired state  $\yd \in L^2(0,T;L^2(\Omega))$  and  initial function $y_0 \in H^1_0(\Omega)$. 

Similarly to the problem \eqref{eq:opt_problem_pde_e},  we define the control-to-state operator $u \mapsto y(u)$ from  $\Hs := L^2(0,T; L^2(\Omega))$ to $\Ws: = L^2(0,T;  H^2(\Omega) \cap H_0^1(\Omega)) \cap H^1(0,T; L^2(\Omega))$.  The well-posedness and  twice continuous  Fr\'echet differentiability of the control-to-state operator can be established  by similar arguments as given in \cite{AzmKunRod,KunRod}.  Compared to the elliptic problem  \eqref{eq:opt_problem_pde_e},  we define 
\begin{equation*}
\costr(u) = \frac{1}{2} \|y(u)- \yd \|_{L^2(0,T; L^2(\Omega))}^2,  \qquad  \costp(u) \coloneqq  \ \lambda \|u\|_{L^1(0,T;L^1(\Omega))}+\delta_{\mathcal{U}}(u), 
\end{equation*}
with  $\mathcal{U} \coloneqq \{ u: \ulb  \leq  u \leq \uub \text{ a.e.  in }(0,T) \times \Omega \} \subset \Hs$ and also consider
\begin{equation*}
\PDEcon(y, u) =  \begin{pmatrix} & \dot{y} - \kappa \Delta y + y^3 - u \\ & y(0) \end{pmatrix}  \quad  \text{ with }  \quad  Z':= L^2(0,T; L^2(\Omega))\times H^1_0(\Omega).
\end{equation*}
 In this case,   \eqref{eq:B72} holds with $p = p(u) \in \Ws$ as the strong solution of the adjoint problem
 \begin{equation}
\label{eq:parabolic_adjoint}
 \begin{cases} 
-\dot{p} - \kappa \Delta p + 3 y^2 p  = - (y - \yd) & \text{ in } (0,T) \times \Omega,\\
p = 0 & \text{ on } \partial \Omega, \\
p(T)  = 0  & \text{ in } \Omega,
\end{cases}
\end{equation} 
with $y = y(u)$.
By similar arguments and the compact embedding  $W\xhookrightarrow{c} L^2(0,T;H^1_0(\Omega)) $  and consequently  $W\xhookrightarrow{c} H $,  it can be shown that Assumption \ref{ass:general3} holds also for \eqref{eq:opt_problem_pde_p}.  Moreover,  the closed form pointwise a.e. representation of the proximal operator is given by
\begin{align*}
\prox_{\frac{1}{\ssize} \costp}(u)(t,x) = \min\{\max\{
\begin{cases}
u(t,x) - C_2, & u(t,x) > C_2,\\
u(t,x) + C_2, & u(t,x) < - C_2,\\
0, & \text{otherwise},
\end{cases}, \ulb\}, \uub\},
\end{align*}
where $C_2 = \lambda / \ssize$. 
 
\subsection{Discussion on the quadratic-growth condition}
In the remainder of the section,  we present a situation in which $\cost$ is  locally strongly convex and,  as a consequence,  Algorithm \ref{algo:NMLS} converges locally $R$-linear by Corollary \ref{cor:StrongCon}.  We consider the case that for  $u^* \in S_*$,   $\costr$ is twice continuously Fr\'echet differentiable and it satisfies  
\begin{align}
\label{eq:second_order_growth}
\costrDD(u^*)(h,h) \geq C \|h\|_\Hs^2   \quad  \text{ for all }  h \in  \Hs,
\end{align}
with some $C \in \R $.   More concretely,  for the two previous problem,  we express $  \costrDD(u^*)$ in terms of solutions to PDEs and discuss when $\cost$ is locally strongly convex.

For both previous model  problems,  the control-to-state operator $u \mapsto y(u)$ from $\Hs$ to $\Ws$ is twice continuously Fréchet differentiable,  and by similar computation as in \cite{MR1871460},  its second derivative can be expressed as
\begin{equation*}
y''(u)(\delta u, \delta v)=-\PDEcon^{-1}_y(y(u),u))\PDEcon_{yy}(y(u),u)(y'(u)\delta u,y'(u) \delta v)  \quad  \text{  for all }  \delta u, \delta v \in \Hs,
\end{equation*}
since the control operator is linear.   Then,  using the chain rule,  we obtain
 \begin{equation}
 \label{e146}
 \begin{split}
  \costrDD(u)&(\delta u,\delta v) = \langle \costyu''(y(u))y'(u)\delta u, y'(u)\delta v\rangle_{W',W} \\
 &+\langle -\PDEcon^{-*}_y(y(u),u)\costyu'(y(u)),\PDEcon_{yy}(y(u),u)(y'(u)\delta u,y'(u) \delta v)\rangle_{Z,Z'},
\end{split}
\end{equation} 
where  $\langle \cdot , \cdot \rangle$ stands for the dual pairing.   

For problem \eqref{eq:opt_problem_pde_e},   $\costrDD(u^*)$ with $u^* \in S_*$   can be expressed as 
\begin{equation}
\label{eq:quadratic_growth_elliptic}
 \costrDD(u^*)(h, h) = \|y^h\|_{L^2(\Omega)}^2 + \int_{\Omega}  p^* \exp(y^*)(y^h)^2 \mathrm{d}x,
\end{equation}
where $y^* = y(u^*), p^* = p(u^*)$,  and $y^h \in \Ws$ is the weak solution of the linearized state equation
\begin{equation*}
\begin{cases}
-\kappa \Delta y^h + \exp(y^*) y^h  = h & \text{in } \Omega,\\
y^h  = 0 & \text{on } \partial \Omega.
\end{cases}
\end{equation*}
Now,   due to \eqref{eq:B71},   $\cost$ is locally strongly convex provided that
\begin{equation}
\label{eq:B73}
\sigma \| h \|^2_{\Hs} + \costrDD(u^*)(h, h)  \geq   \eta \| h \|^2_{\Hs}  \quad  \text{ for all }   h \in \Hs
\end{equation}
 for a $\eta>0 $, where the  term $\sigma \| h \|^2_H$ stands for the second directional derivative of $\frac{\sigma}{2} \|u\|_{L^2(\Omega)}^2$  in direction $h$.  Further,  using \eqref{eq:quadratic_growth_elliptic},  \eqref{eq:B73} can be equivalently rewritten as
 \begin{equation}
 \label{eq:B74}
 \sigma \| h \|^2_{L^2(\Omega)} +  \|y^h\|_{L^2(\Omega)}^2 + \int_{\Omega}  p^* \exp(y^*)(y^h)^2 \mathrm{d}x  \geq   \eta \| h \|^2_{L^2(\Omega)}  \quad  \text{ for all }   h \in \Hs.
 \end{equation}
We observe that the only term in \eqref{eq:B74} that can spoil the positive definiteness of  $\sigma \mathcal{I}+ \mathcal{F}''(u^*)$ and,  hence,  local strong convexity of $\cost$, is the term involving $p^*$. This term originates from the nonlinearity in the state equation.  
Note that either a small enough adjoint $p^*$ or a large enough parameter $\sigma$ ensure that \eqref{eq:B73} holds.  A small enough adjoint can occur,  for instance,  if $\|y^* - \yd\|_{L^2(\Omega)}^2$ is sufficiently small.

Similarly,  for the second example \eqref{eq:opt_problem_pde_p}, we obtain
\begin{align}
\label{eq:quadratic_growth_parabolic}
\costrDD(u^*)(h, h) = \|y^h\|_{L^2(0,T;L^2(\Omega))}^2 + \int\limits_0^T \int\limits_\Omega 6 p^* y^* (y^h)^2 \mathrm{d}x \mathrm{d}t,
\end{align}
where $y^h \in W$ is the weak solution to the linearized state equation
 \begin{equation*}
\begin{cases}
\dot{y}^h - \kappa \Delta y^h + 3 (y^*)^2 y^h  = h & \text{in } (0,T) \times \Omega,\\
y^h  = 0 & \text{on } (0,T) \times \partial \Omega,\\
y^h(T) = 0 & \text{in } \Omega,
\end{cases}
\end{equation*}
and $p^*$ solves the weak formulation of \eqref{eq:parabolic_adjoint} for $y^*$.   For this problem,  due to the absence  of the term  $ \frac{\sigma}{2}\| \cdot\|_{\Hs}^2$ in the objective function of   \eqref{eq:quadratic_growth_parabolic}, the local strong convexity is not clear. For the elliptic model problem \eqref{eq:opt_problem_pde_e} the authors in  \cite{C17} have studied a weaker condition which implies that the quadratic growth condition \eqref{eq:B26} is satisfied. Furthermore from \cite[Section 3]{C12} it especially follows that for $\sigma = 0$ condition \eqref{eq:B74} cannot be satisfied for the elliptic model problem \eqref{eq:opt_problem_pde_e}.

\section{Numerical Experiments}
\label{sec:numerics}

In this section,  we report on the numerical experiments for the problems \eqref{eq:opt_problem_pde_e} and \eqref{eq:opt_problem_pde_p} in order to verify the capabilities of Algorithm  \ref{algo:NMLS} numerically.  Throughout,  we use $\|\gmap_{\ssizek}(u_k)\|_\Hs \leq \tol$ with some tolerance $\tol > 0$ as termination condition,  as it is proposed in Section \ref{sec:problem_algorithm}.  Our codes are implemented in Python 3 and use FEniCS (see \cite{fenics}) for the matrix assembly.   Sparse memory management and computations are implemented with SciPy (see \cite{scipy}). All computations below are run on an Ubuntu 22.04 notebook with 32 GB main memory and an Intel Core i7-8565U CPU.  

We will also compare different step-size  approaches for the iterations \eqref{eq:update_v1}  with respect to gradient-like evaluations and function evaluations as introduced in Theorem \ref{thm:complexity}.   We consider a fixed step-size, different combination of  BB-type step-sizes presented in Section \ref{sec:problem_algorithm} and BB-type step-sizes incorporated  with the (non-)monotone linesearch approach.

Note that for problems governed by nonlinear PDEs such as \eqref{eq:opt_problem_pde_e} and \eqref{eq:opt_problem_pde_p},  any gradient-like  $\gmap_\ssize$  evaluation requires solving a nonlinear state equation and a linear adjoint equation and any function $\cost$ evaluation  is involved with solving a nonlinear state equation.  Furthermore,  the number of gradient-like evaluations corresponds to the number of iterations $k$ of Algorithm \ref{algo:NMLS}.

\begin{example}[Elliptic model problem]
\label{sec:numerics:elliptic}
In this example,   we consider problem \eqref{eq:opt_problem_pde_e}.  For  the spatial discretization,  we follow a discretize-before-optimize approach and use $P_1$-type finite elements on a Friedrichs-Keller triangulation of the spatial spatial domain $\Omega$. To efficiently evaluate the nonlinearity, we resort to mass lumping.  For the numerical tests,  we choose the parameters summarized in Table \ref{tab:elliptic_parameters_general}. Note that for the fixed step-size approach $\alpha = \alpha_0$ is used.
\begin{table}[!htb]
\begin{center}
\resizebox{\linewidth}{!}{
\begin{tabular}{>{$}c<{$} >{$}c<{$} >{$}c<{$} >{$}c<{$} >{$}c<{$} >{$}c<{$} >{$}c<{$} >{$}c<{$} >{$}c<{$} >{$}c<{$} >{$}c<{$} >{$}c<{$} >{$}c<{$} >{$}c<{$}}
\toprule
\multicolumn{7}{c}{optimization problem} & \multicolumn{7}{c}{Algorithm \ref{algo:NMLS}}\\ \hline
\multicolumn{1}{c}{$\Omega$} & \multicolumn{1}{c}{$\kappa$} & \multicolumn{1}{c}{$\sigma$} & \multicolumn{1}{c}{$\lambda$} & \multicolumn{1}{c}{$y_d$} & \multicolumn{1}{c}{$\ulb$} & \multicolumn{1}{c}{$\uub$} & \multicolumn{1}{c}{$\tol$} & \multicolumn{1}{c}{$\ssizelb$} & \multicolumn{1}{c}{$\ssizeub$} & \multicolumn{1}{c}{$\ssize_0$} & \multicolumn{1}{c}{$\eta$} & \multicolumn{1}{c}{$\delta$} & \multicolumn{1}{c}{$\mmax$}\\ \hline\hline
(0,1)^2  & 10^{-2} & 10^{-4} & 10^{-3} & \mathrm{see\; Figure\; \ref{fig:elliptic_desired}} & -3 & 2 & 10^{6} & 10^{-4} & 10^{2} & 10 & 8 & 0.9 & 8\\
\bottomrule
\end{tabular}
}
\caption{Example \ref{sec:numerics:elliptic}: Parameter setting.}
\label{tab:elliptic_parameters_general}
\end{center}
\end{table}
The desired state,  the optimal state,  and the optimal control (see Section \ref{sec:problem_algorithm}) are illustrated in Figure \ref{fig:elliptic_state_control_desired}.  We can see that the bounds are active for the control,  though no strong sparsity is promoted,  due to the choices of $\lambda$ and $\sigma$.     
\begin{figure}
 \centering
\subfigure[Desired State]
	{
		\includegraphics[height=3cm,width=3.7cm]{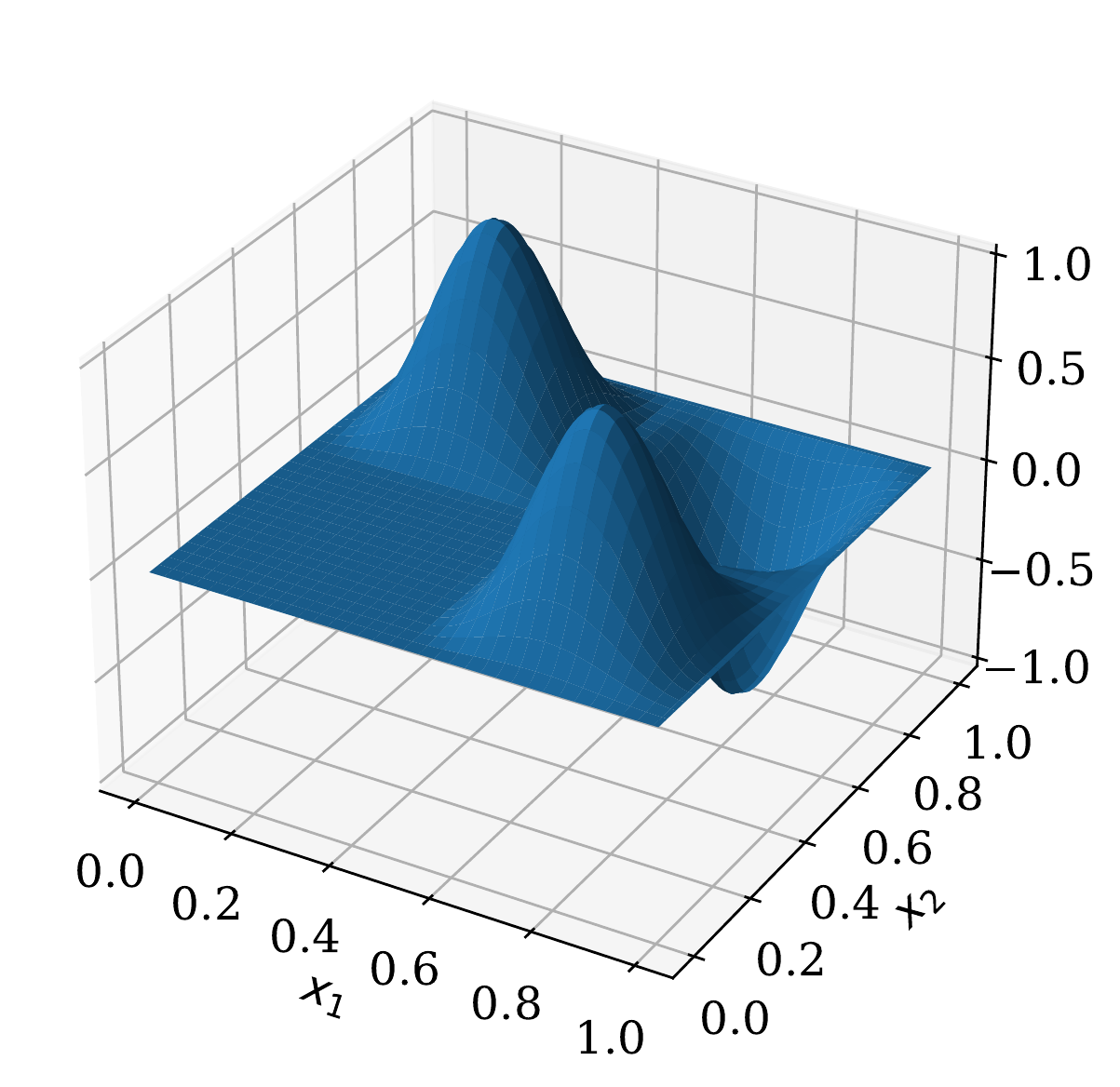}
		\label{fig:elliptic_desired}

	}
\subfigure[Optimal State]
	{
		\includegraphics[height=3cm,width=3.7cm]{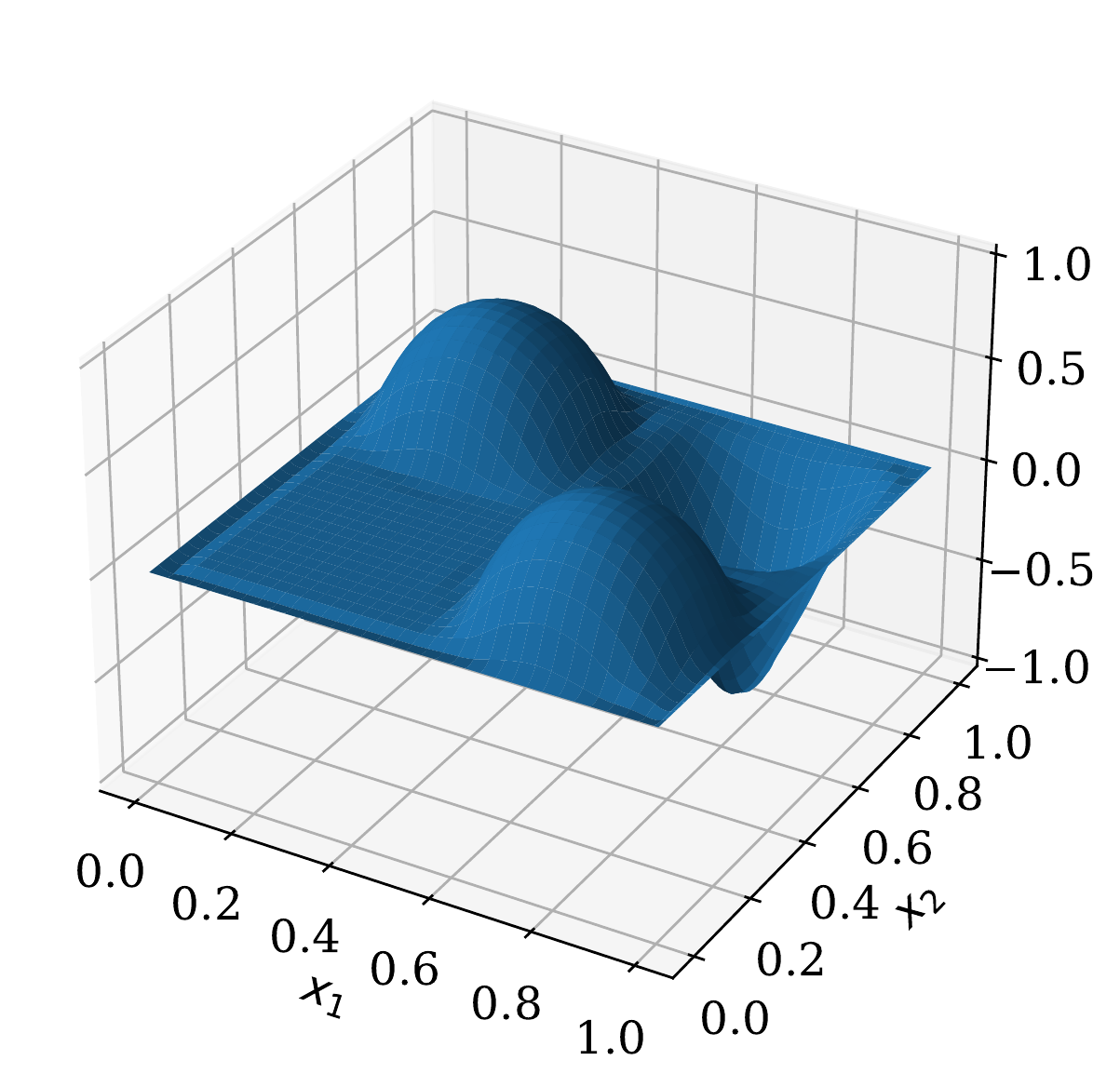}
	}
\subfigure[Optimal Control]
	{
		\includegraphics[height=3cm,width=3.7cm]{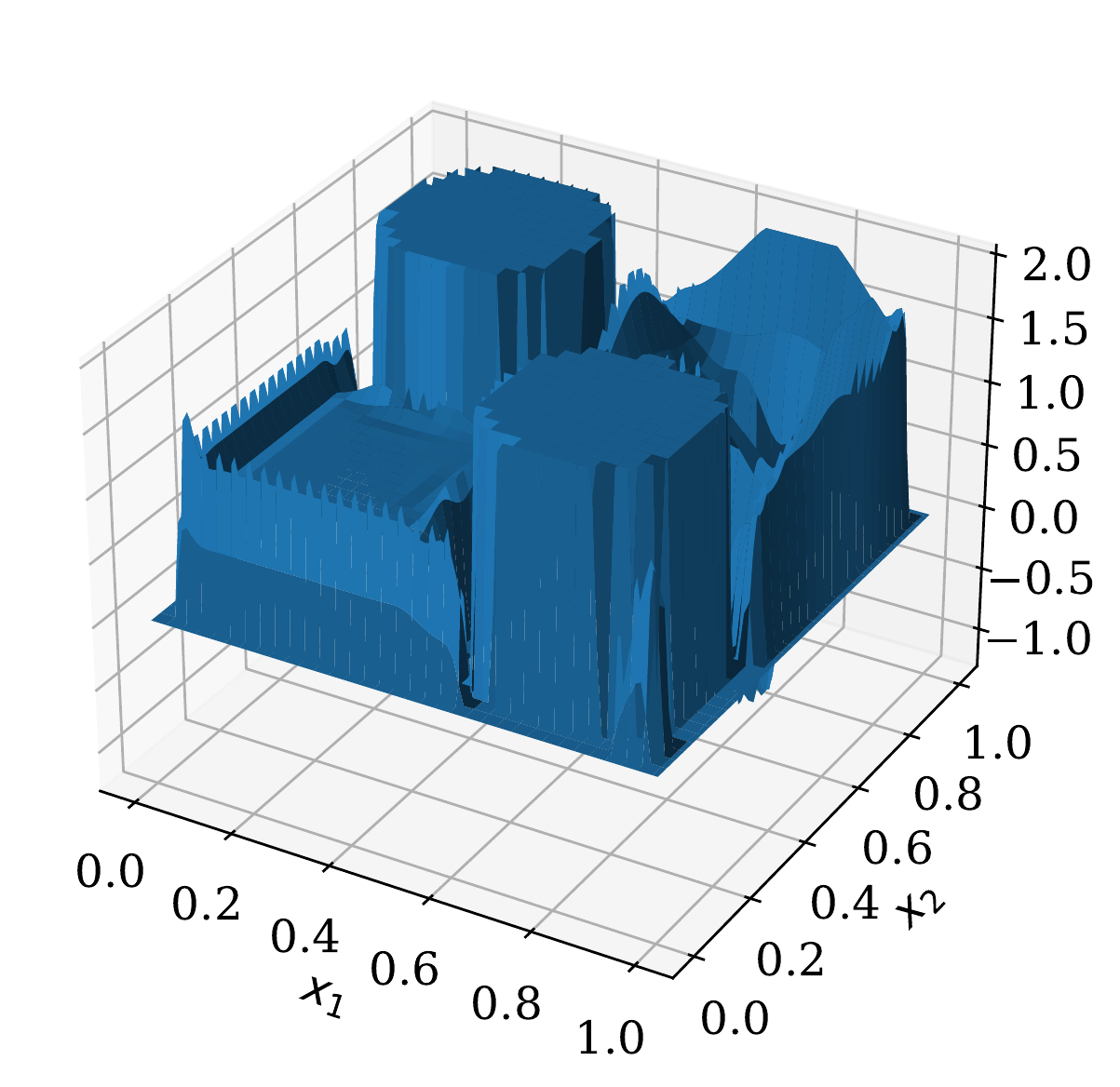}
	}
 \caption{Example \ref{sec:numerics:elliptic}:  The desired state $\yd$,   the optimal state,  and the optimal control (left to right).}
\label{fig:elliptic_state_control_desired}
\end{figure}
We compare the different BB-type step-sizes presented in Section \ref{sec:problem_algorithm} with the baseline approach of using a fixed step-size in \eqref{eq:update_v1} and with a (non-)monotone linesearch approach. The results regarding computational time,  function evaluations,  and gradient-like evaluations are gathered in Table \ref{tab:elliptic_results_BB}.  For the linesearch methods, the most volatile of the novel BB-type step-size updates is used as baseline,  i.e.  BB1b.

\begin{table}[!htb]
\begin{center}
\resizebox{\linewidth}{!}{
\begin{tabular}{>{$}c<{$} >{$}c<{$} >{$}c<{$} >{$}c<{$} >{$}c<{$} >{$}c<{$} >{$}c<{$} >{$}c<{$} >{$}c<{$} >{$}c<{$} >{$}c<{$} >{$}c<{$} >{$}c<{$}}
\toprule
 & \multicolumn{1}{c}{fixed} & \multicolumn{1}{c}{BB1a} & \multicolumn{1}{c}{BB2a} & \multicolumn{1}{c}{ABBa} & \multicolumn{1}{c}{BB1b} & \multicolumn{1}{c}{BB2b} & \multicolumn{1}{c}{ABBb} & \multicolumn{1}{c}{nonmon. LS (BB1b)} & \multicolumn{1}{c}{mon. LS (BB1b)}\\ \hline\hline
\text{grad.-like eval.}  & 153662 & 618 & 1046 & 571 & 941 & 608 & 383 & 697 & 991 \\
\text{fun. eval.}  & 0 & 0 & 0 & 0 & 0 & 0 & 0 & 887 & 1527 \\
\text{time} [s] & 1.56 \cdot 10^3 & 1.01 \cdot 10^1 & 1.36 \cdot 10^1 & 1.55 \cdot 10^1 & 8.39 \cdot 10^0 & 7.92 \cdot 10^0 & 5.65 \cdot 10^0 & 2.21 \cdot 10^1 & 3.30 \cdot 10^1 \\
\bottomrule
\end{tabular}}
\caption{ Example \ref{sec:numerics:elliptic}: Numerical results for fixed step-size,  different spectral gradient methods as introduced in Section \ref{sec:problem_algorithm},  and a (non-)monotone linesearch method with respect to BB1b.}
\label{tab:elliptic_results_BB}
\end{center}
\end{table}
As can be seen from Table \eqref{tab:elliptic_results_BB},  for this example,  all other approaches outperform the one with fixed step-size by a huge margin of about two orders of magnitude.  The alternating BB-methods appear to be more efficient compared to   the single BB-updates for both the old and novel step-sizes.  Furthermore,  the novel step-sizes do not always outperform the old ones,  but in the superior alternating BB-method, they do by a significant margin.  In other words, they are competitive to the old ones.  All of these considerations are valid for both computational time and gradient-like evaluations.  Function evaluations are only needed if a linesearch method is used.  Compared to the baseline method BB1b,  the nonmonotone linesearch method needs about $250$ fewer gradient-like evaluations,  but this comes at the cost of $887$ additional function evaluations for the linesearch and this results in an increased overall computational time.  Compared to a monotone linesearch, the nonmonotone approach performs significantly better. The convergence behavior is also visualized in Figure \ref{fig:elliptic_results}.
\begin{figure}
 \centering
\subfigure
	{
		\includegraphics[height=3.5cm,width=4.5cm]{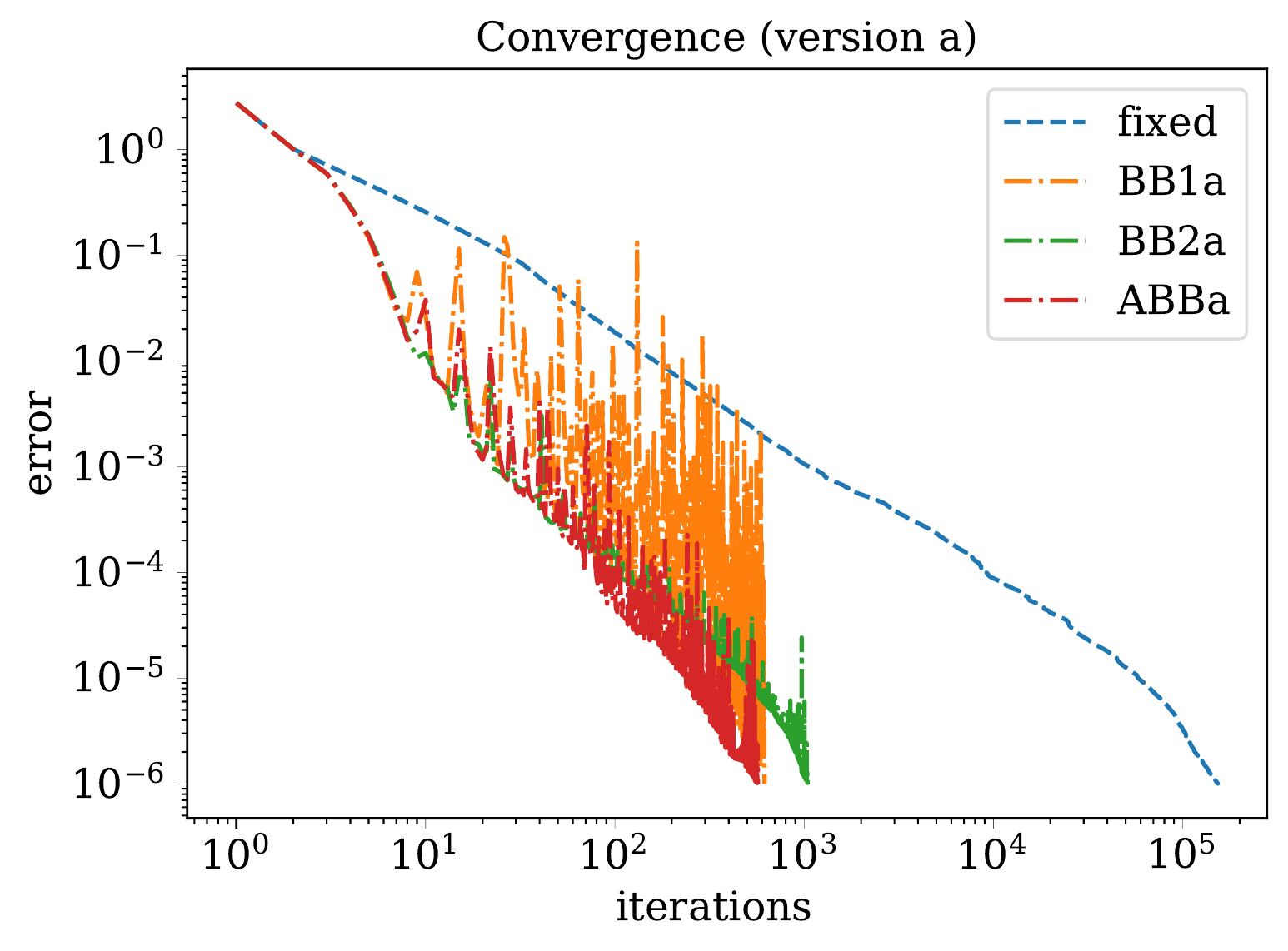}
	}
\subfigure
	{
		\includegraphics[height=3.5cm,width=4.5cm]{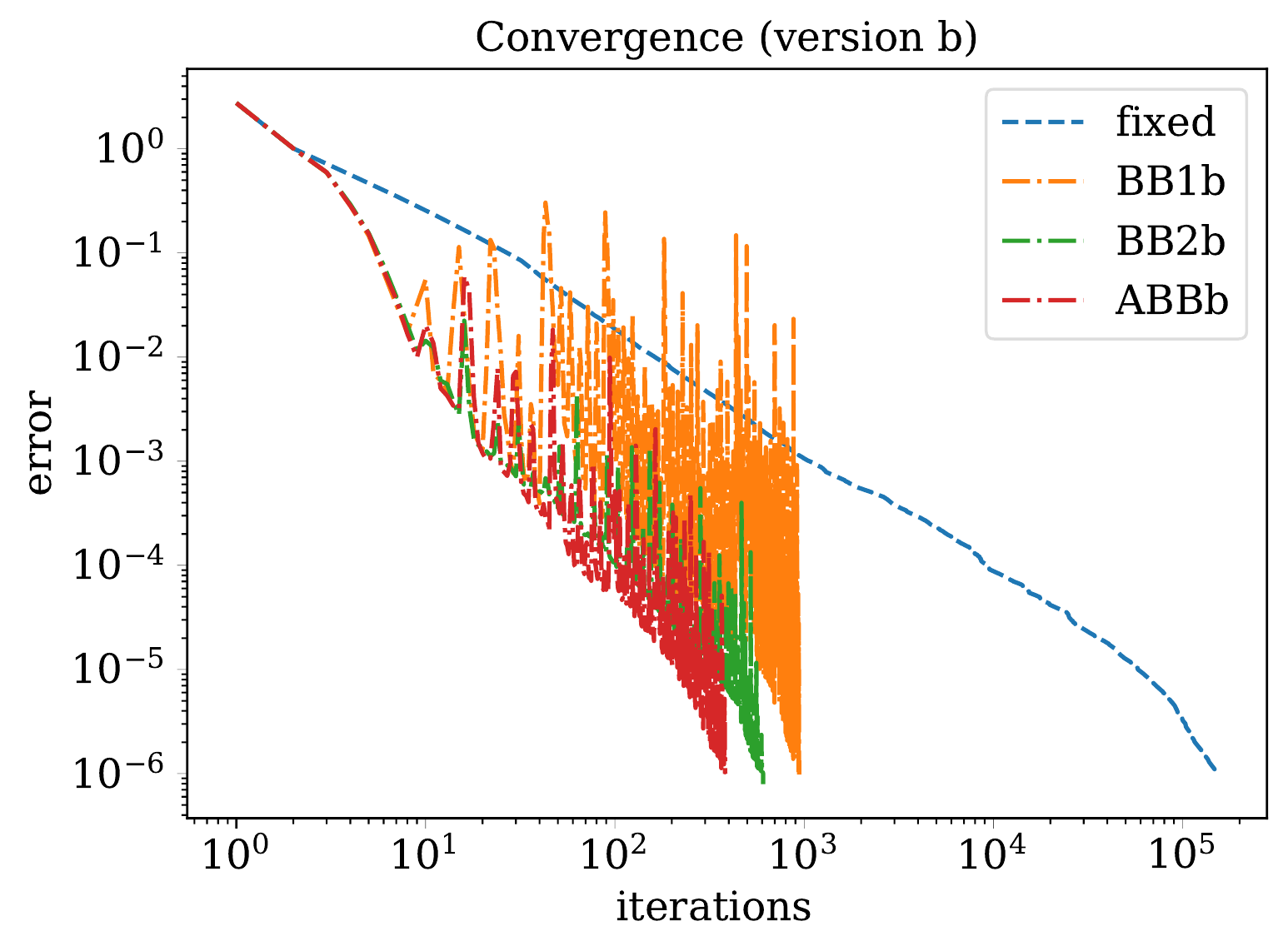}
	}
\subfigure
	{
		\includegraphics[height=3.5cm,width=4.5cm]{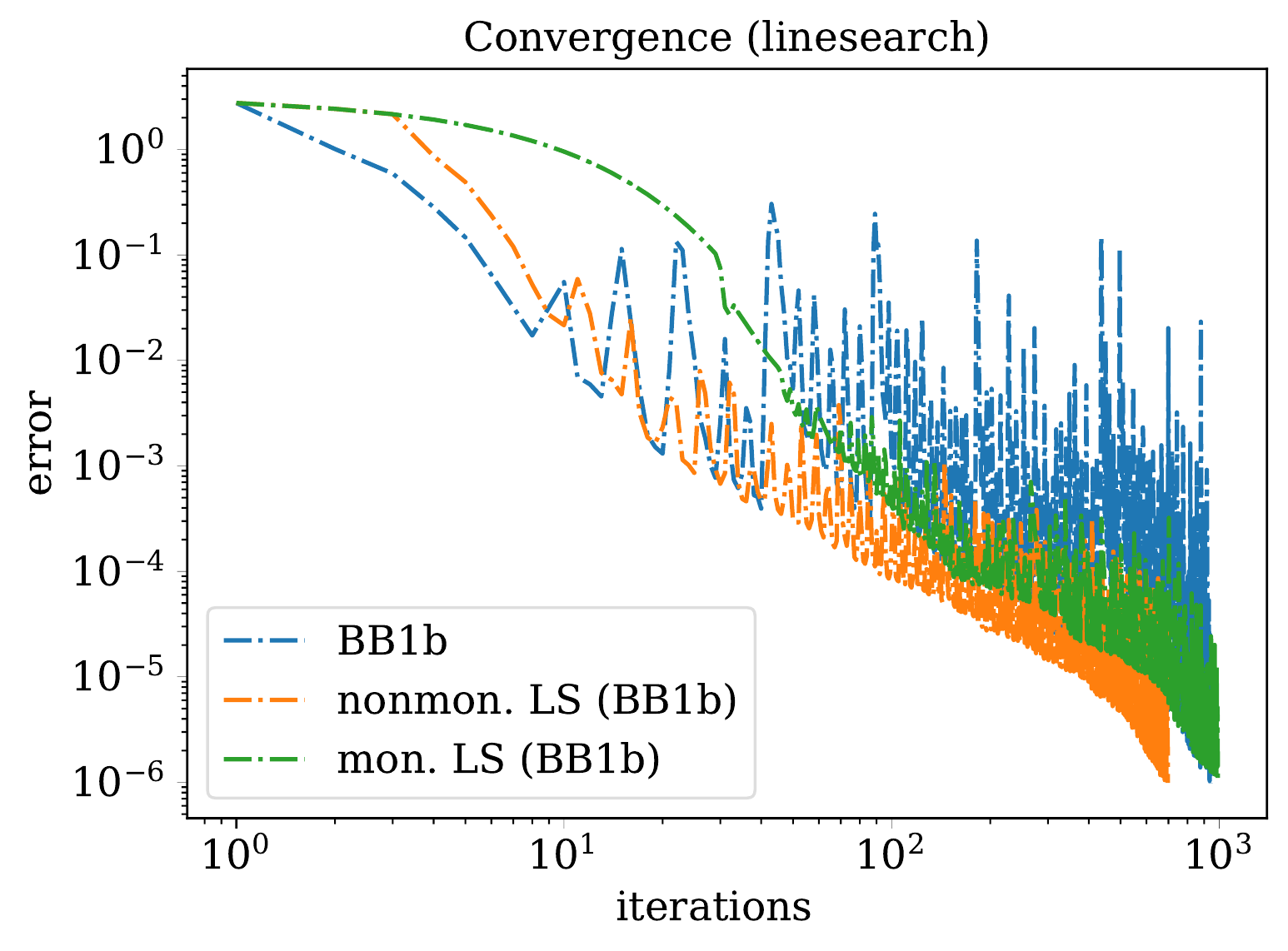}
		\label{Fig:eliptic_line}
	}
 \caption{Example \ref{sec:numerics:elliptic}:  Convergence of Algorithm \ref{algo:NMLS}. "Error" refers to $\|\gmap_{\ssizek}(u_k)\|_\Hs$ at the current iterate.}
\label{fig:elliptic_results}
\end{figure}
Figure \ref{fig:elliptic_noncon} presents an example,   where a BB-type step-size update without linesearch fails to converge.  In this example,  we start the algorithms  with  $\ssize_0 = 1$ instead of $\ssize_0 = 10$.   This confirms the necessity of incorporating a linesearch strategy to ensure convergence.
\begin{figure}
 \centering
		\includegraphics[height=3.5cm,width=4.5cm]{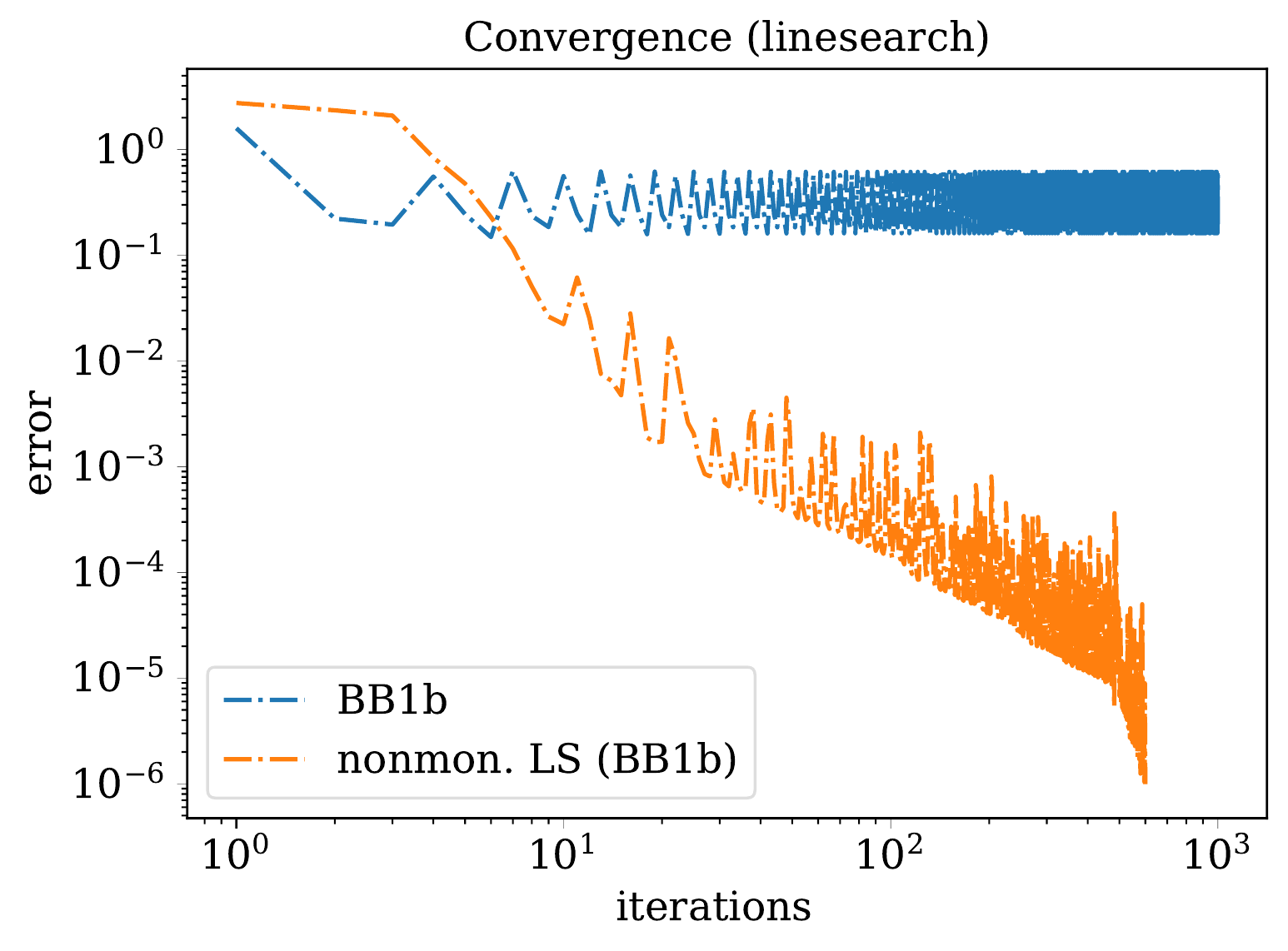}
	\caption{Example \ref{sec:numerics:elliptic}: Example of non-convergence without linesearch.  "Error" refers to $\|\gmap_{\ssizek}(u_k)\|_\Hs$ at the current iterate.}
\label{fig:elliptic_noncon}
\end{figure}
\end{example}

\begin{example}[Parabolic model problem]
\label{sec:numerics:parabolic}
In this example,  we consider the model problem \eqref{eq:opt_problem_pde_p}.  Here,  the spatial domain is discretized in the same manner as the one of the elliptic problem in Example \ref{sec:numerics:elliptic}.   Moreover, for the temporal discretization,  we use the Crank Nicolson/Adams Bashforth scheme  \cite{MR2034880}.  In this scheme,  the implicit Crank Nicolson scheme is used  except for the nonlinear terms which are treated using  the explicit Adams Bashforth scheme and mass lumping.  For the numerical tests we choose the parameters summarized in Table \ref{tab:parabolic_parameters_general}. Note that again for the fixed step-size approach $\ssize = \ssize_0$ is used.
\begin{table}[!htb]
\begin{center}
\resizebox{\linewidth}{!}{
\begin{tabular}{>{$}c<{$} >{$}c<{$} >{$}c<{$} >{$}c<{$} >{$}c<{$} >{$}c<{$} >{$}c<{$} >{$}c<{$} >{$}c<{$} >{$}c<{$} >{$}c<{$} >{$}c<{$} >{$}c<{$} >{$}c<{$} >{$}c<{$}}
\toprule
\multicolumn{7}{c}{optimization problem} & \multicolumn{7}{c}{Algorithm \ref{algo:NMLS}}\\ \hline
\multicolumn{1}{c}{$\Omega$} & \multicolumn{1}{c}{$T$} & \multicolumn{1}{c}{$\kappa$} & \multicolumn{1}{c}{$\lambda$} & \multicolumn{1}{c}{$y_d$} & \multicolumn{1}{c}{$\ulb$} & \multicolumn{1}{c}{$\uub$} & \multicolumn{1}{c}{$\tol$} & \multicolumn{1}{c}{$\ssizelb$} & \multicolumn{1}{c}{$\ssizeub$} & \multicolumn{1}{c}{$\ssize_0$} & \multicolumn{1}{c}{$\eta$} & \multicolumn{1}{c}{$\delta$} & \multicolumn{1}{c}{$\mmax$}\\ \hline\hline
(0,1)^2  & 1 & 10^{-2} & 10^{-2} & \mathrm{see\; Figure\; \ref{fig:parabolic_desired}} & -100 & 100 & 10^{-6} & 10^{-4} & 10^{2} & 10 & 4 & 0.8 & 4\\
\bottomrule
\end{tabular}
}
\caption{Example \ref{sec:numerics:parabolic}: Parameter setting.}
\label{tab:parabolic_parameters_general}
\end{center}
\end{table}
The desired state,  the optimal state and the optimal control  at time instances $t = 0, 0.25, 0.5, 0.75$ are depicted  in Figures \ref{fig:parabolic_desired}, \ref{fig:parabolic_state},  and \ref{fig:parabolic_control},  respectively.   We can clearly see sparsity in space for the control. Note hat also sparsity in time can be observed in the sense that the control stays zero on an interval between time instances $t =0.25$ and $t=0.75$.
\begin{figure}
 \centering
\subfigure[$t = 0$]
	{
		\includegraphics[height=3.2cm,width=3.2cm]{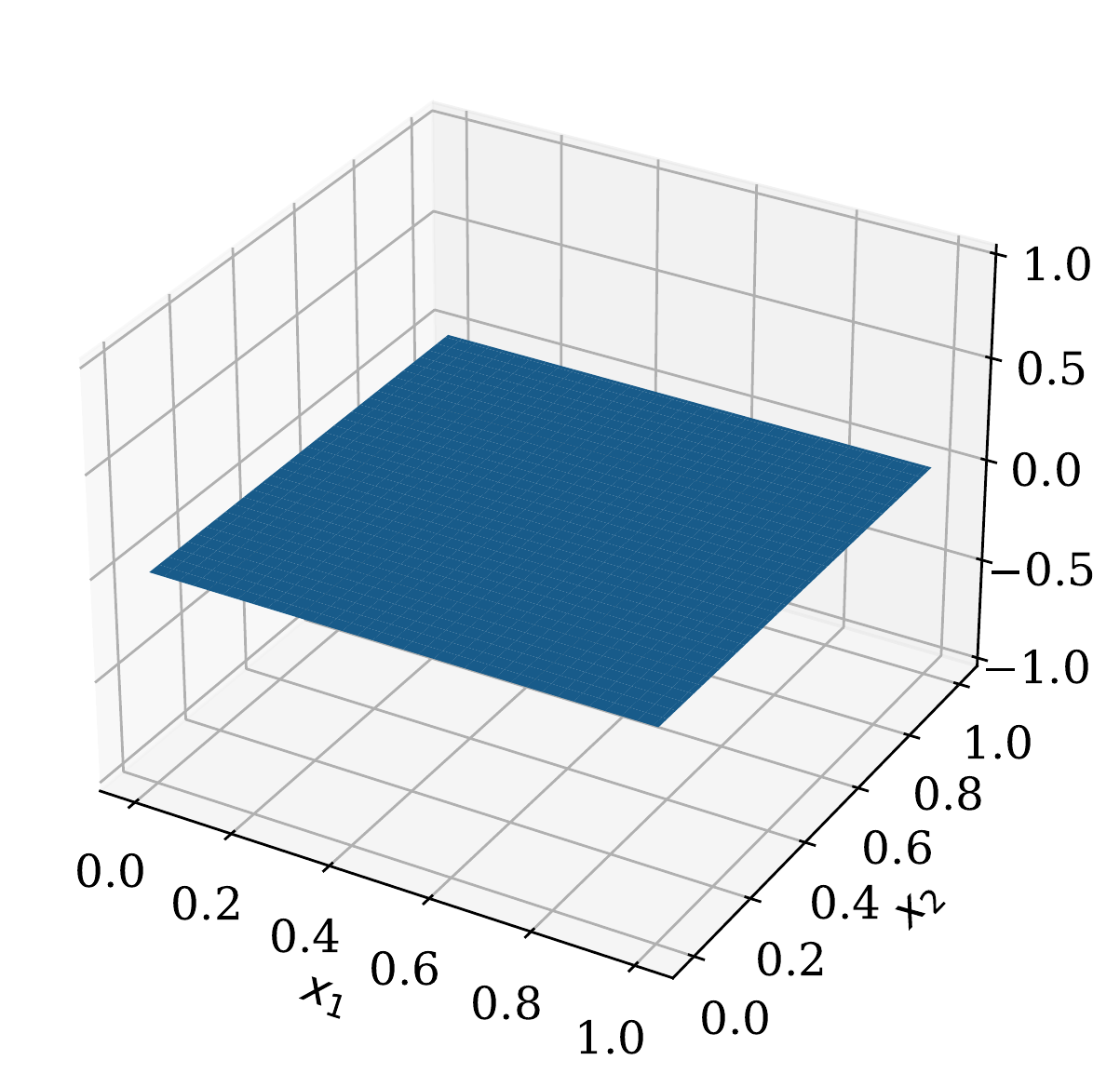}
	}
\subfigure[$t =0.25$]	
	{
		\includegraphics[height=3.2cm,width=3.2cm]{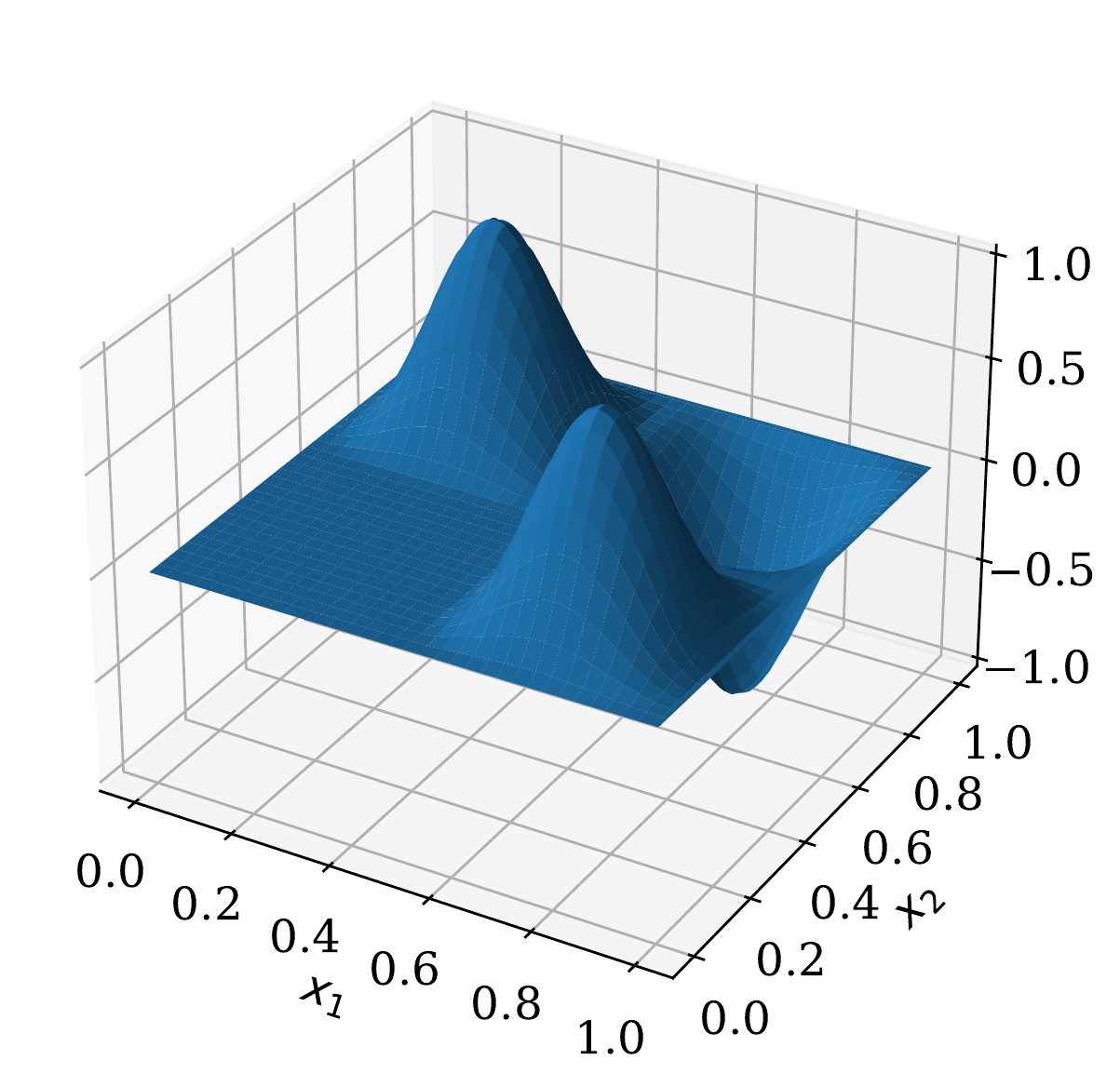}
	}
\subfigure[$t =0.5$]
	{
		\includegraphics[height=3.2cm,width=3.2cm]{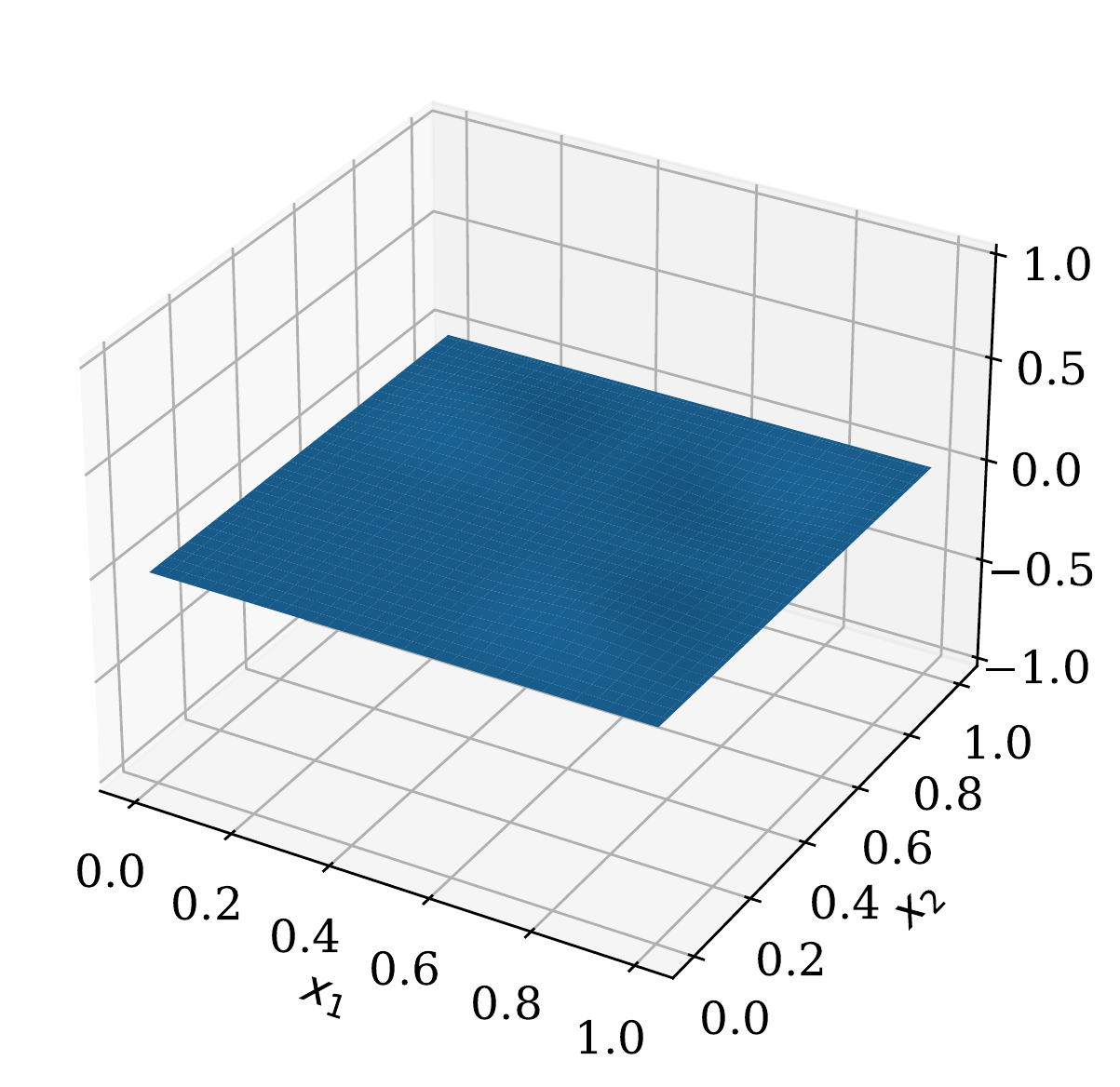}
	}
\subfigure[$t =0.75$]	
	{
		\includegraphics[height=3.2cm,width=3.2cm]{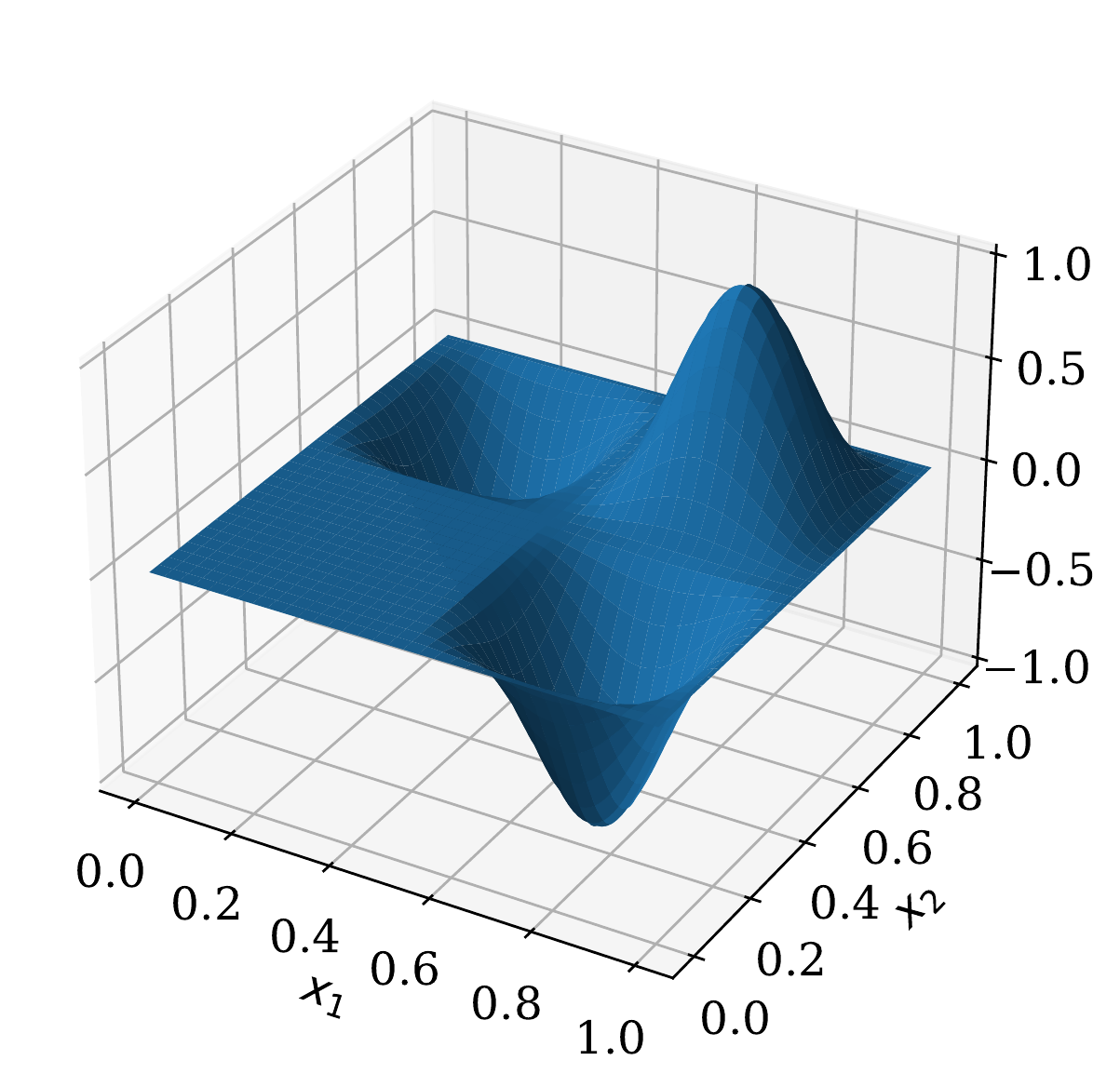}
	}
	\caption{Example \ref{sec:numerics:parabolic}: Snapshots of the desired state $\yd$ at time instances $0, 0.25, 0.5$ and $0.75$.}
	\label{fig:parabolic_desired}
\subfigure[$t = 0$]
	{
		\includegraphics[height=3.2cm,width=3.2cm]{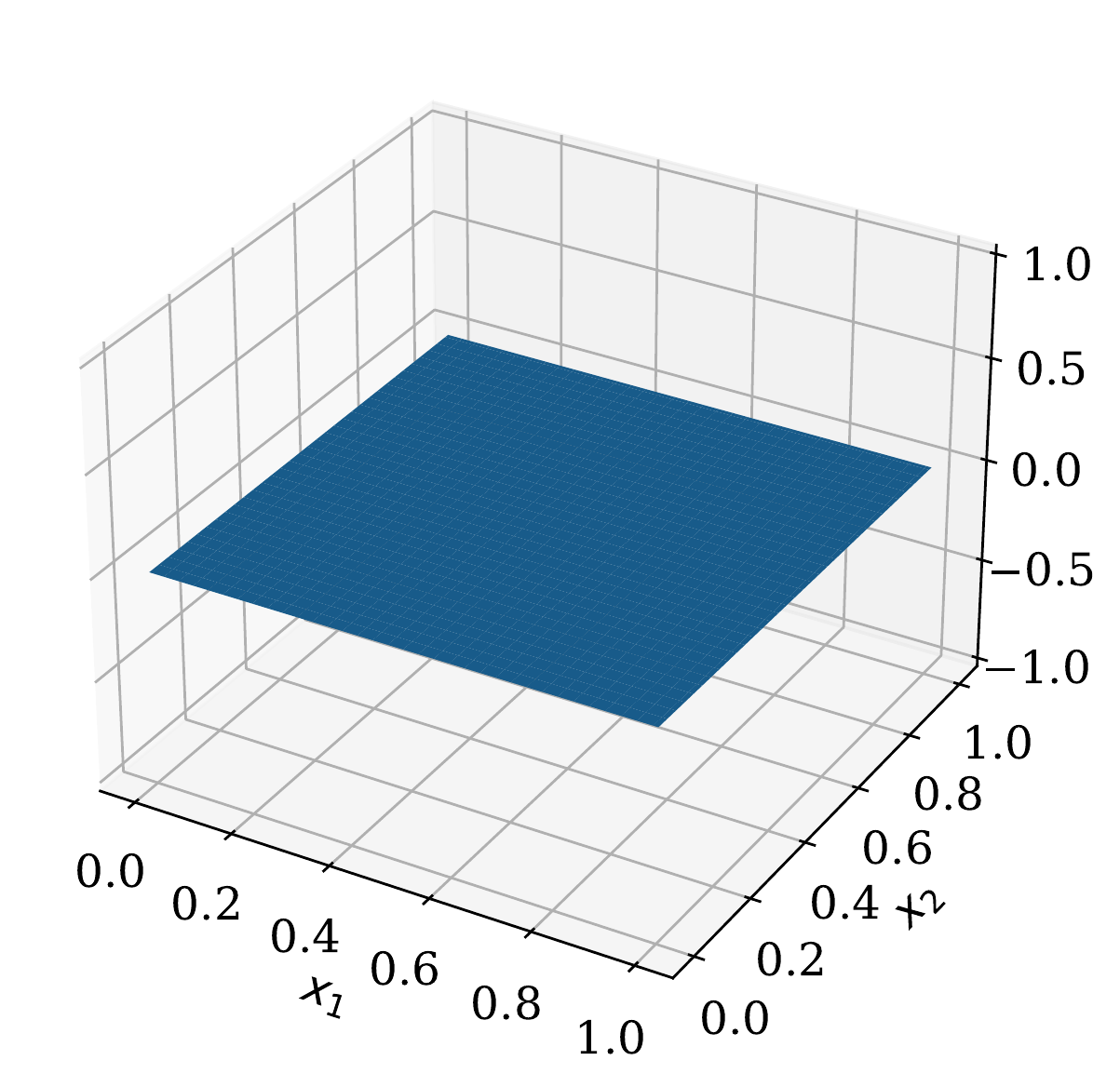}
	}
\subfigure[$t =0.25$]	
	{
		\includegraphics[height=3.2cm,width=3.2cm]{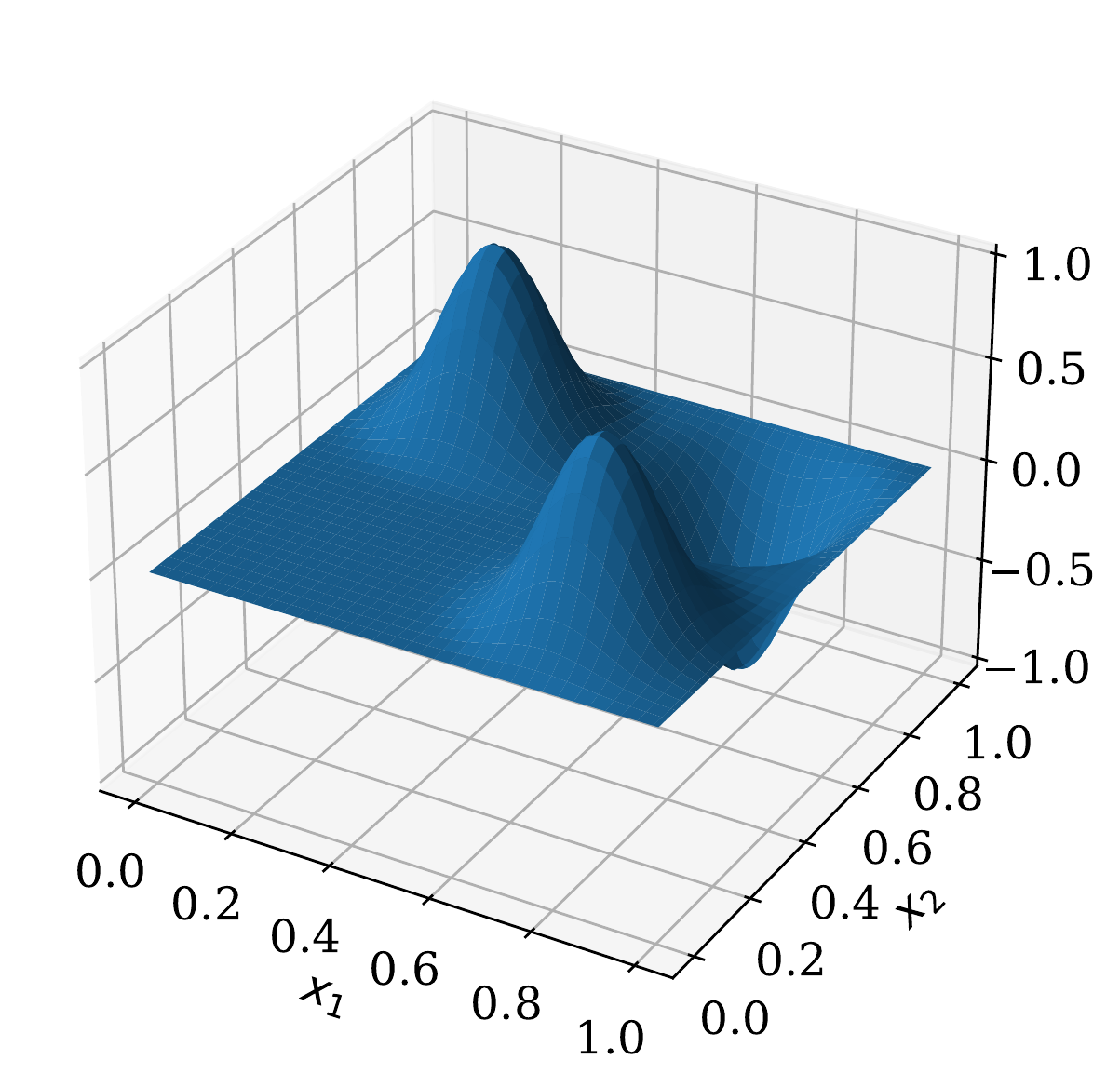}
	}
\subfigure[$t =0.5$]	
	{
		\includegraphics[height=3.2cm,width=3.2cm]{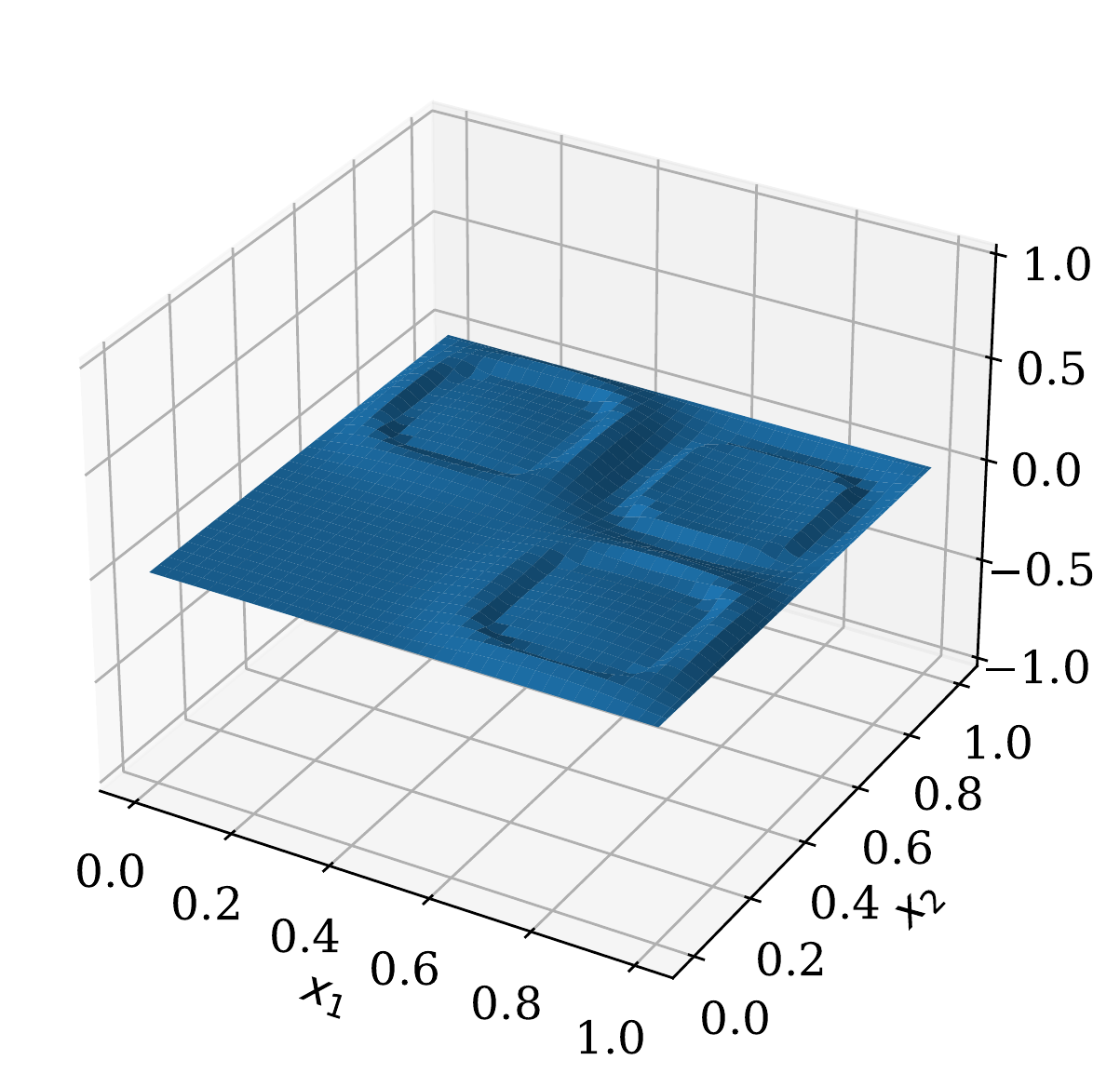}
	}
\subfigure[$t =0.75$]	
	{
		\includegraphics[height=3.2cm,width=3.2cm]{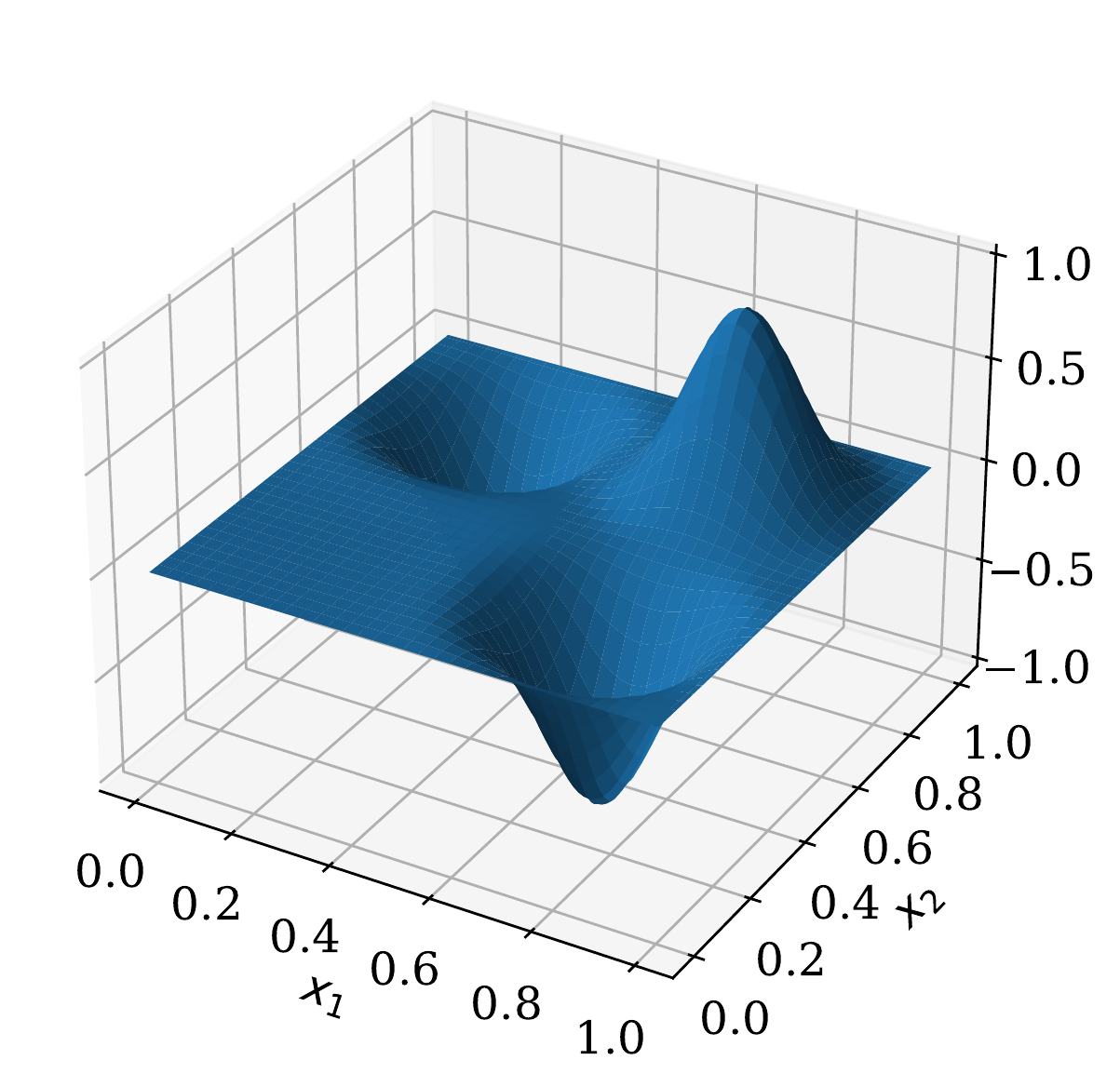}
	}
	\caption{Example \ref{sec:numerics:parabolic}: Snapshots of the optimal state at time instances $0, 0.25, 0.5$ and $0.75$.}
		\label{fig:parabolic_state}
\subfigure[$t = 0$]
	{
		\includegraphics[height=3.2cm,width=3.2cm]{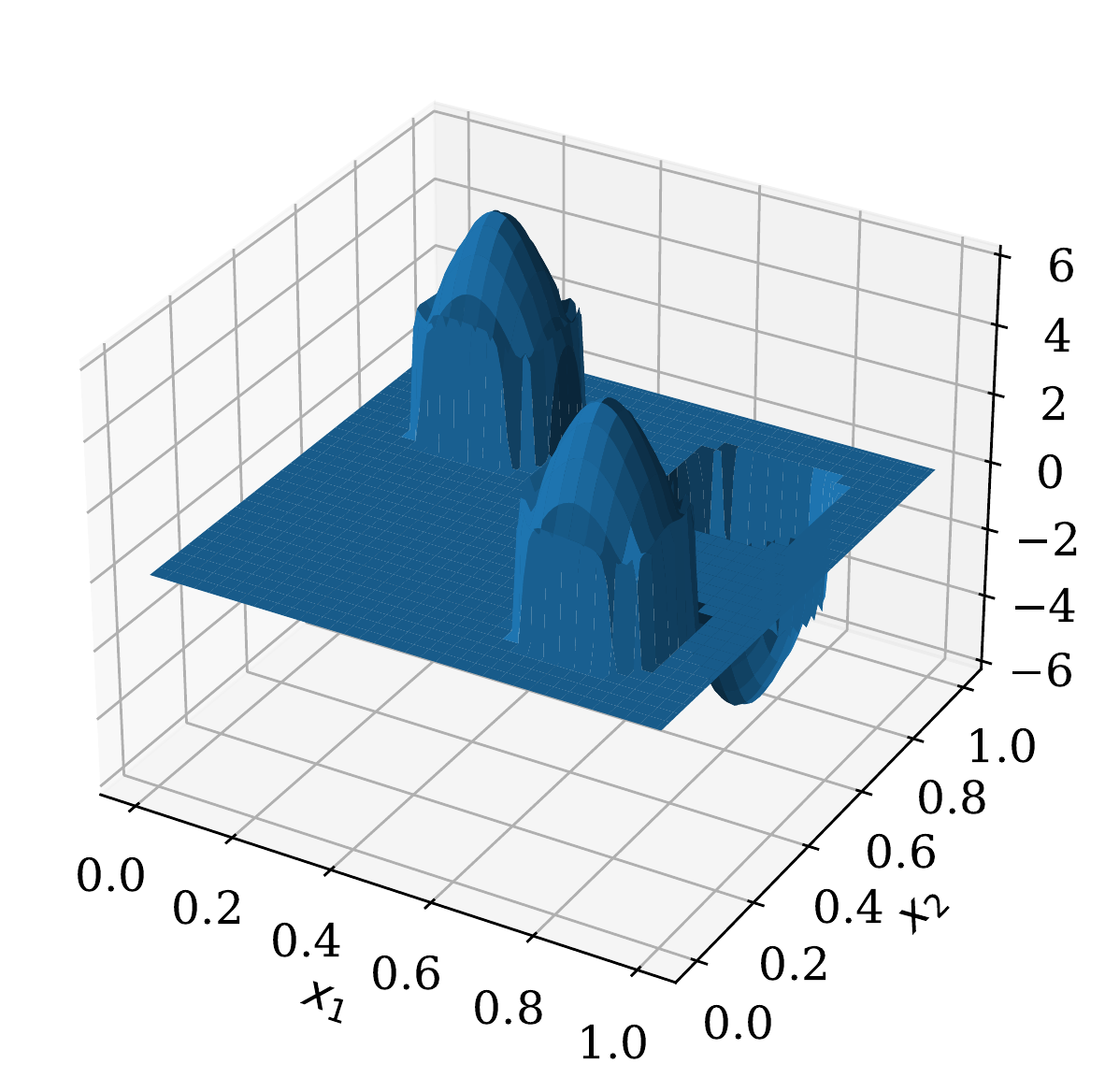}
	}
\subfigure[$t =0.25$]	
	{
		\includegraphics[height=3.2cm,width=3.2cm]{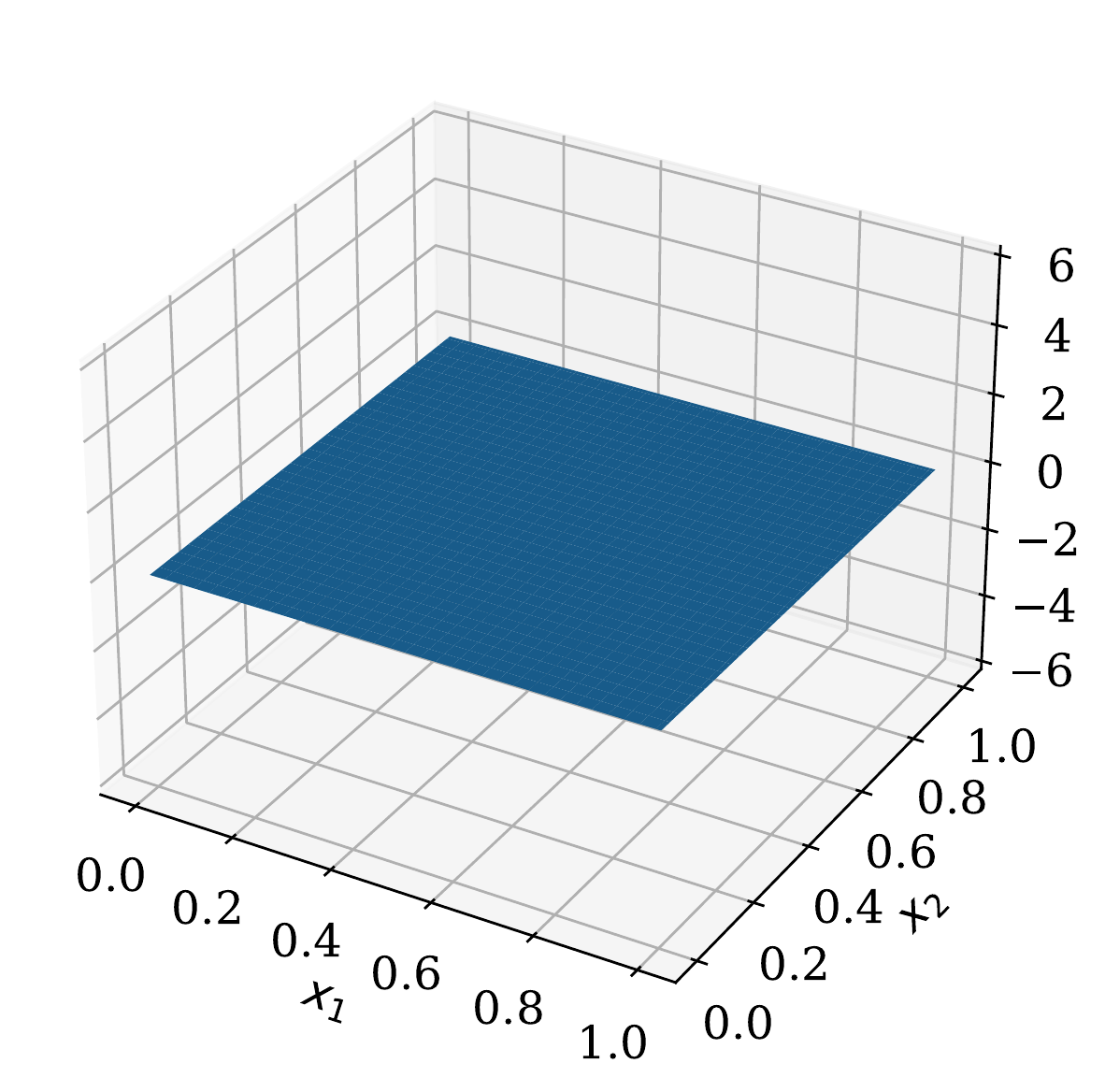}
	}
\subfigure[$t =0.5$]	
	{
		\includegraphics[height=3.2cm,width=3.2cm]{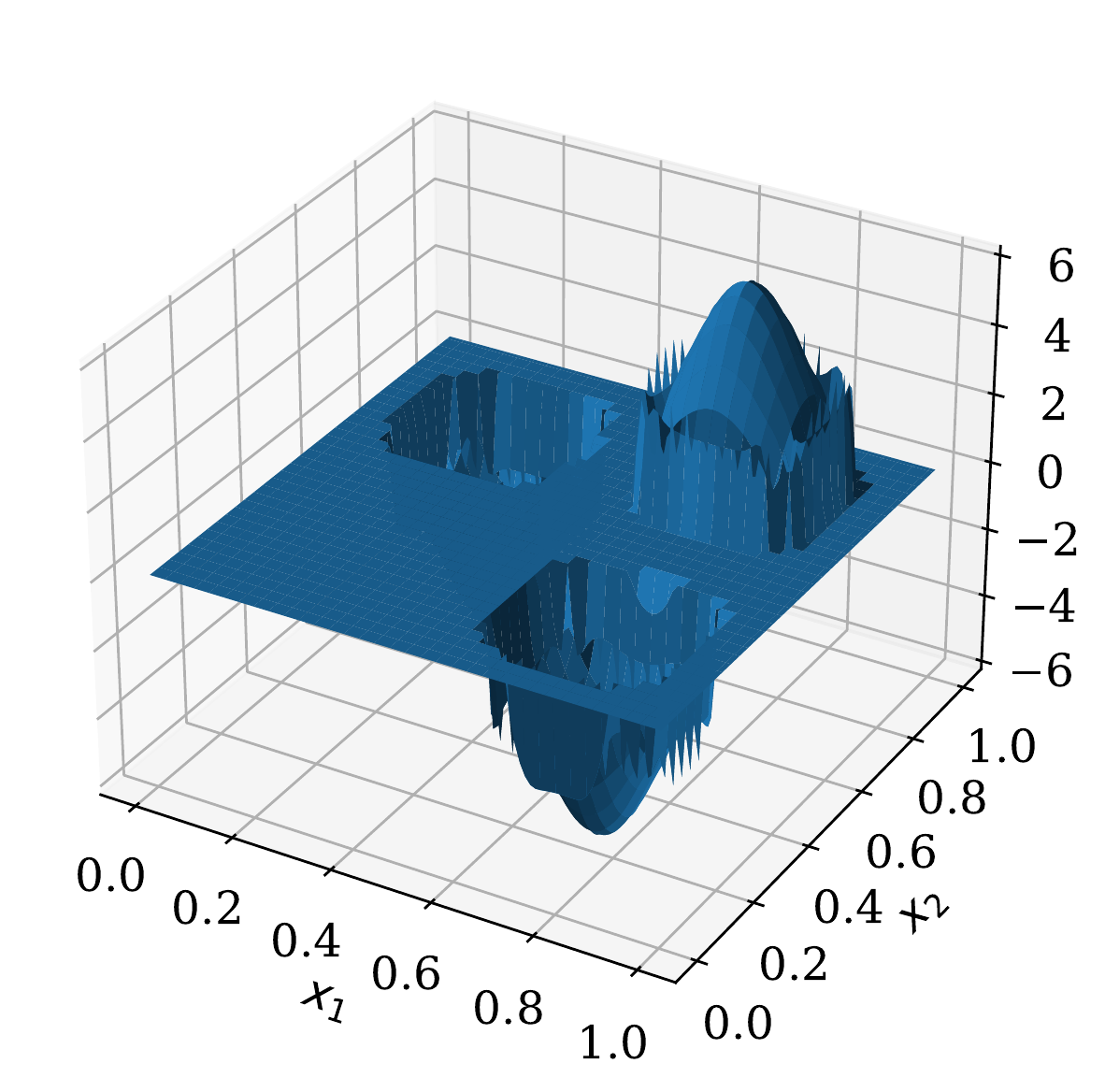}
	}
\subfigure[$t =0.75$]	
	{
		\includegraphics[height=3.2cm,width=3.2cm]{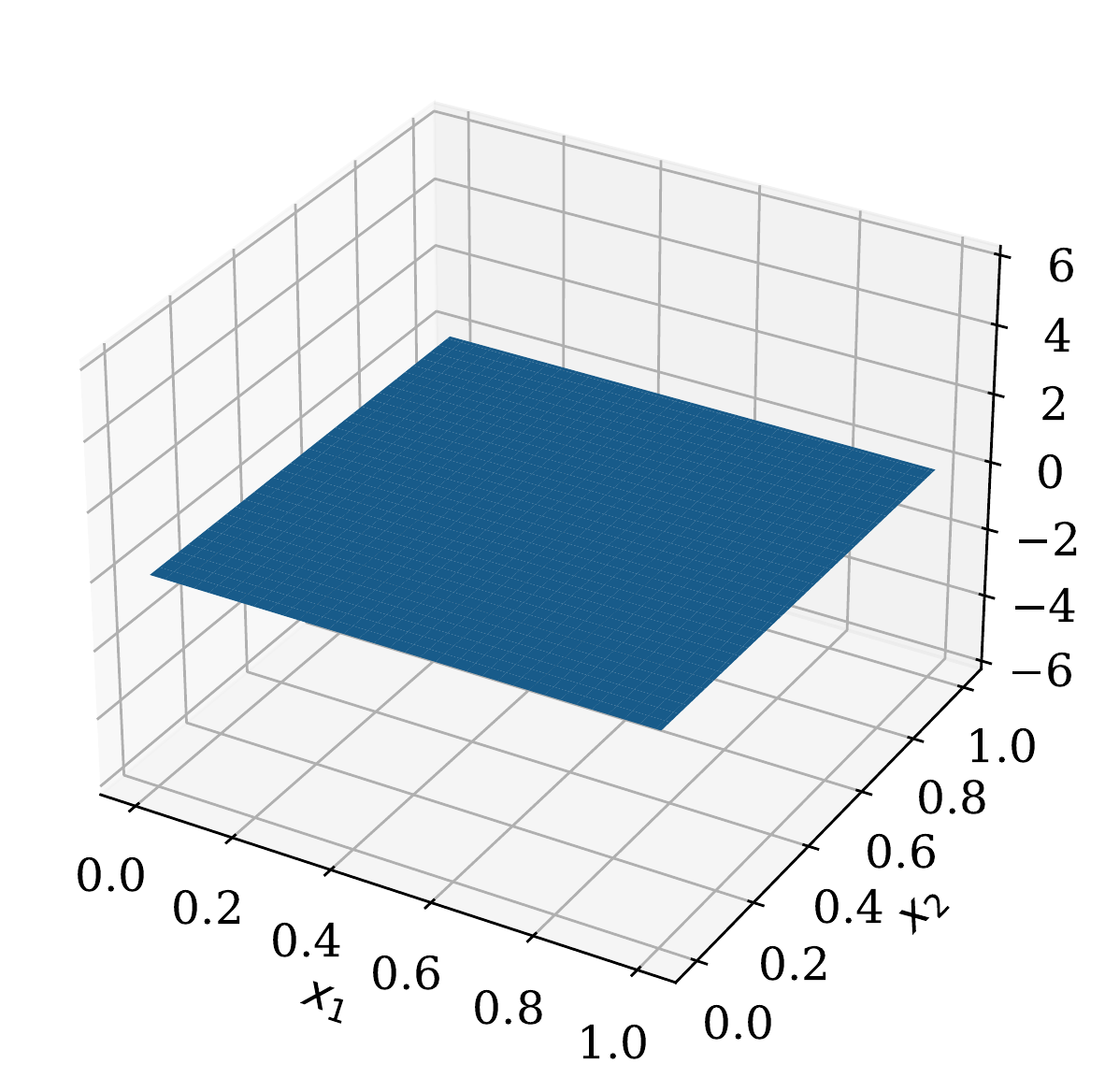}
	}
 \caption{Example \ref{sec:numerics:parabolic}: Snapshots of the optimal control at time instances $0, 0.25, 0.5$ and $0.75$.}
\label{fig:parabolic_control}
\end{figure}
Similarly to Exmaple \ref{sec:numerics:elliptic},    we compare the different BB-type step-sizes presented in Section \ref{sec:problem_algorithm} with the baseline approach of using a fixed step-size and with a (non-)monotone linesearch approach,  and the results regarding computational time, function evaluations and gradient-like evaluations are presented in Table \ref{tab:parabolic_results_BB}.  To reveal the efficiency of the linesearch strategy,  we incorporated the least efficient BB-type step-size,  namely BB1b  with the linesearch strategy.
\begin{table}[!htb]
\begin{center}
\resizebox{\linewidth}{!}{
\begin{tabular}{>{$}c<{$} >{$}c<{$} >{$}c<{$} >{$}c<{$} >{$}c<{$} >{$}c<{$} >{$}c<{$} >{$}c<{$} >{$}c<{$} >{$}c<{$} >{$}c<{$} >{$}c<{$} >{$}c<{$}}
\toprule
 & \multicolumn{1}{c}{fixed} & \multicolumn{1}{c}{BB1a} & \multicolumn{1}{c}{BB2a} & \multicolumn{1}{c}{ABBa} & \multicolumn{1}{c}{BB1b} & \multicolumn{1}{c}{BB2b} & \multicolumn{1}{c}{ABBb} & \multicolumn{1}{c}{nonmon. LS (BB1b)} & \multicolumn{1}{c}{mon. LS (BB1b)}\\ \hline\hline
\text{grad.-like eval.}  & 515044 & 375 & 758 & 784 & 470 & 445 & 221 & 463 & 471 \\
\text{fun. eval.}  & 0 & 0 & 0 & 0 & 0 & 0 & 0 & 639 & 713 \\
\text{time} [s] & 5.66 \cdot 10^4 & 4.28 \cdot 10^1 & 8.70 \cdot 10^1 & 8.89 \cdot 10^1 & 5.35 \cdot 10^1 & 5.08 \cdot 10^1 & 2.50 \cdot 10^1 & 8.65 \cdot 10^1 & 9.57 \cdot 10^1 \\
\bottomrule
\end{tabular}}
\caption{Example \ref{sec:numerics:parabolic}: Numerical results for fixed step-size,  different spectral gradient methods as introduced in Section \ref{sec:problem_algorithm},  and a (non-)monotone linesearch method with respect to BB1b.}
\label{tab:parabolic_results_BB}
\end{center}
\end{table}
All considerations and observations from Example \ref{sec:numerics:elliptic} hold also true for this example.  That is, the iterations \eqref{eq:update_v1} for the choice of the BB-type step-sizes significantly outperform those with fixed step-size, as is to be expected. The novel alternating BB-method outperforms the other approaches. The nonmonotone linesearch method outperforms the monotone version. It also outperforms the BB1b version with respect to gradient-like evaluations,  but the cost of additional $639$ function evaluations again does not pay off with respect to overall computational time.  The convergence behavior is also visualized in Figure \ref{fig:parabolic_results}.
\begin{figure}
 \centering
\subfigure
	{
		\includegraphics[height=3.5cm,width=4.5cm]{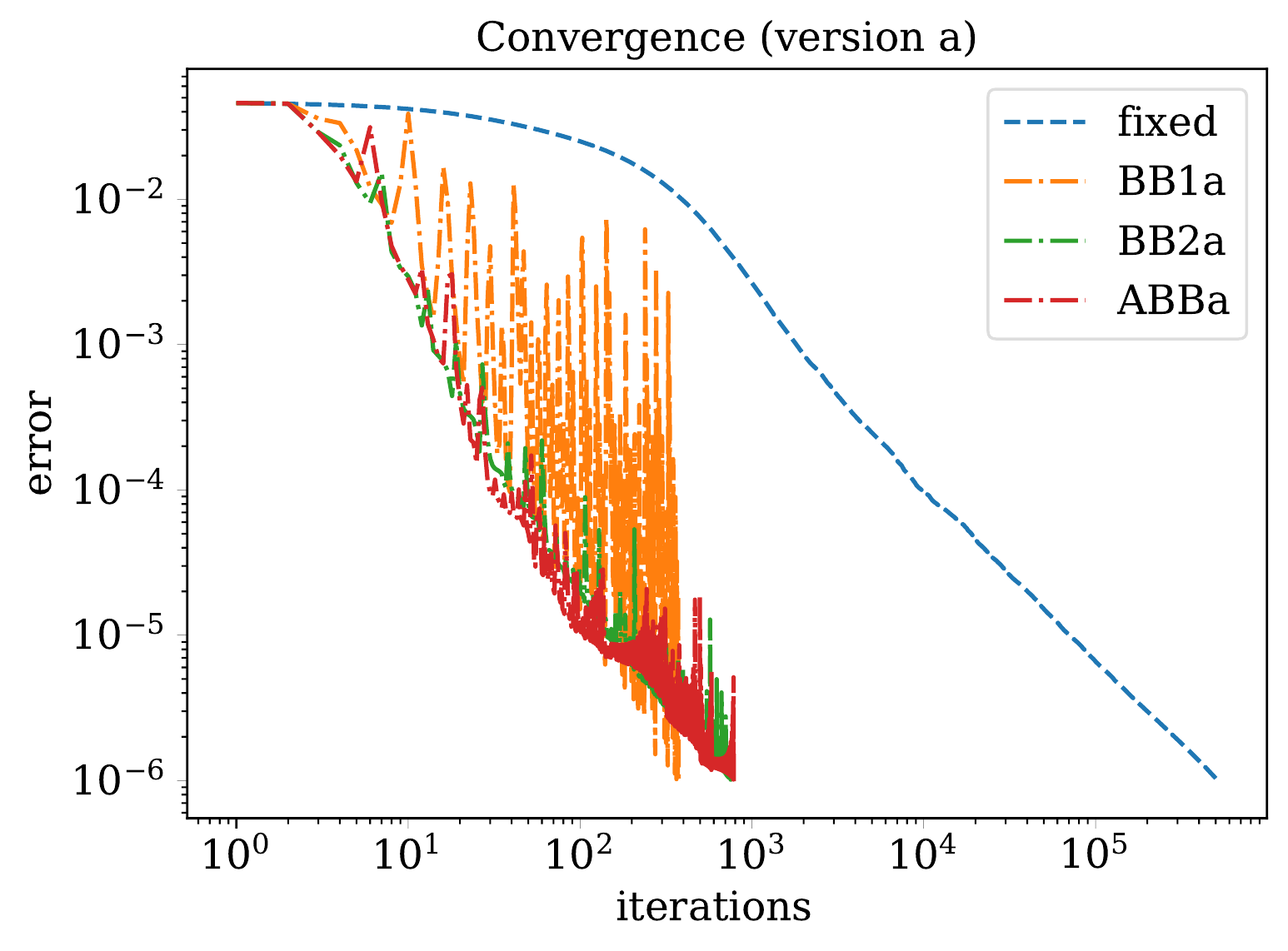}
	}
\subfigure
	{
		\includegraphics[height=3.5cm,width=4.5cm]{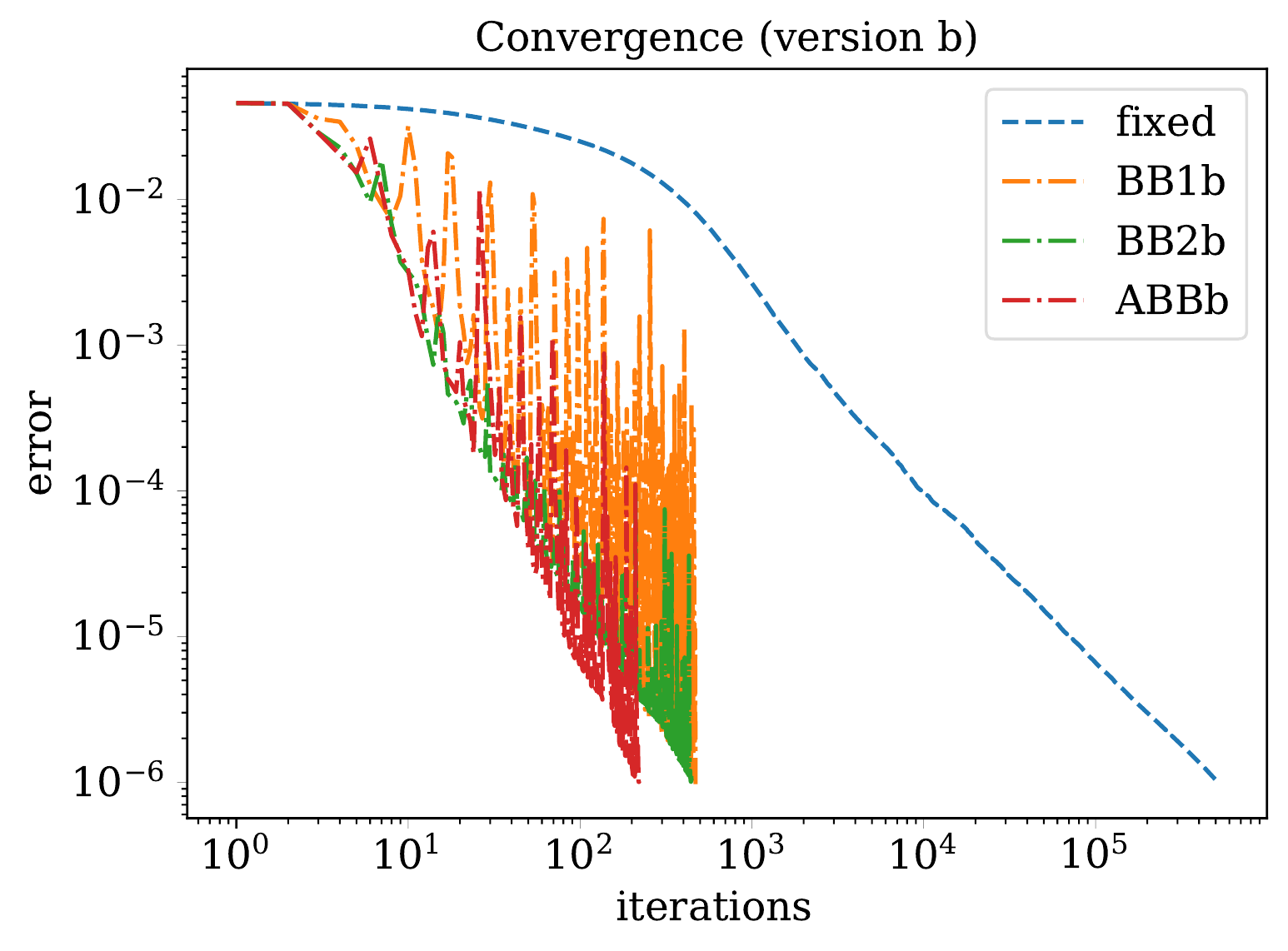}
	}
\subfigure
	{
		\includegraphics[height=3.5cm,width=4.5cm]{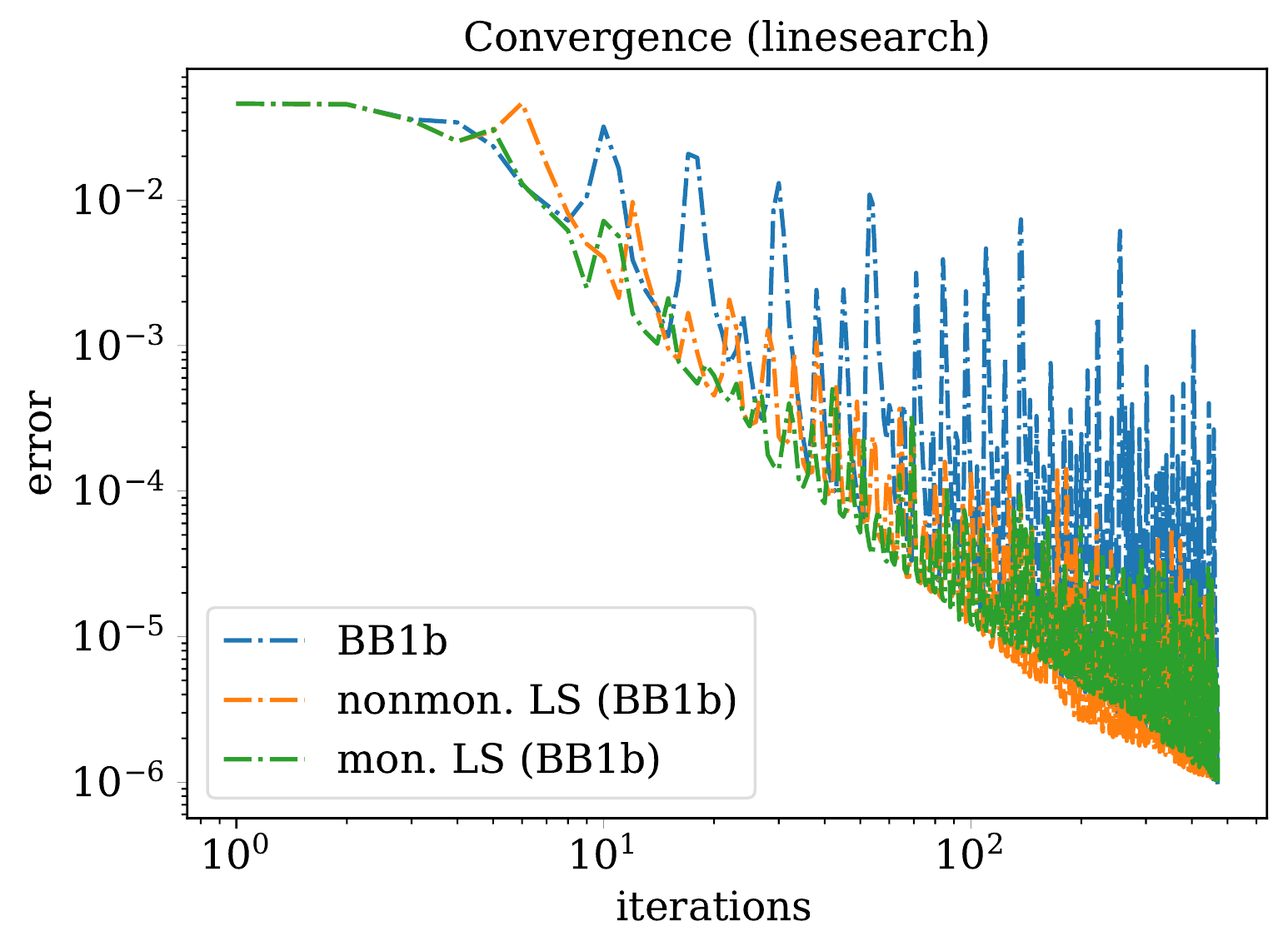}
	}
 \caption{Example \ref{sec:numerics:parabolic}:  convergence of Algorithm \ref{algo:NMLS}. "Error" refers to $\|\gmap_{\ssizek}(u_k)\|_\Hs$ at the current iterate.}
\label{fig:parabolic_results}
\end{figure}
\end{example}

To summarize,  all the numerical results from the above examples show the capabilities of Algorithm \ref{algo:NMLS} and the necessity for the use of a linesearch method.   From the above example,  we can conclude that the incorporation of nonmonotone linesearch and BB step-sizes leads to an efficient algorithm for a class of non-smooth PDE-constrained optimization problems.   Moreover, the convergence behavior is far better than what our worst-case complexity results suggest.  In particular,  when combined with BB methods, this is typical behavior,  see e.g.,  \cite{azmi_analysis_2020,MR967848}.  Whether the strong quadratic growth condition helps to  accelerate convergence in the elliptic example \eqref{eq:opt_problem_pde_e} towards the end is not entirely clear,  although the orange curves in Figures \ref{Fig:eliptic_line} and  \ref{fig:elliptic_noncon} suggest it.

\section*{Conclusion}
\label{sec:conclusion}
We have studied the nonmonotone FBS method for a class of infinite-dimensional composite problems in Hilbert spaces. Most prominently, we have established global convergence with complexity $(1 / \sqrt{k})$ and also provide a complexity analysis. Under additional convexity assumption, convergence is now sublinear of order $(1 / k)$ in function values. If additionally, a quadratic growth-type condition is satisfied, we have shown $R$-linear convergence, both in function values and iterates. Additional difficulties arising in the transition from weak to strong convergence in the infinite-dimensional setting have been discussed in detail.  Finally, the nonmonotone FBS and the novel BB step-size rules exploiting the nonsmooth part were successfully tested for elliptic and parabolic PDE-constrained problems.


\begin{acknowledgements}
The work of Marco Bernreuther was supported by the German Research Foundation (DFG) under grant number VO 1658/5-2
within the priority program ``Non-smooth and Complementarity-based Distributed Parameter
Systems: Simulation and Hierarchical Optimization'' (SPP 1962).
\end{acknowledgements}
\appendix
\section{Appendix}\label{Appendix}
\subsection{Proof of Theorem \ref{thm:complexity}}\label{appendix:complexity}
Here,  we restrict ourselves mainly to deriving  \eqref{eq:B19}.  Justification of \eqref{eq:B20} follows by similar arguments.  The proof is carried out by determining the minimum reduction of the objective function between iteration step $u_{k+1}$  and its maximal predecessor , i.e. $u_{\ell(k)}$,  with respect to \eqref{eq:ls_update} as long as Algorithm \ref{algo:NMLS} has not been terminated.  Note that $u_{\ell(k)}$ can be calculated without further evaluations of the cost function by storing previous iterates.  When using Algorithm \ref{algo:NMLS},  we have the following two cases:
\begin{itemize}
\item Assume that \eqref{eq:ls_ineq} in Step 4 of Algorithm \eqref{algo:NMLS} holds for $\ssizek := \ssizeinit $ and therefore $i_k = 0$.  Then we have a decrease 
\begin{equation*}
\cost(u_{\ell(k)})-\cost(u_{k+1}) \geq \frac{\delta}{\ssizeinit} \|\gmap_{\ssizek}(u_k)\|_\Hs^2 \geq  \frac{\delta}{\ssizeub} \|\gmap_{\ssizek}(u_k)\|_\Hs^2.  
\end{equation*} 
In this case,  only one additional function evaluation is necessary to obtain this decrease. 
\item Assume that \eqref{eq:ls_ineq} in Step 4 of Algorithm \ref{algo:NMLS} fails to hold for $\ssizeinit$ and therefore $i_k\geq 1$ backtracking steps are performed.  Due to \rom{3}  of Lemma \ref{lem:properties_of_iter}  we have 
\begin{equation*}
 \ssizek = \ssizeinit \eta^{i_k} < \frac{\eta \LipcostrD}{2(1-\delta)},
\end{equation*}
and,  as a consequence,  we can bound $i_k$ as follows
\begin{equation*}
i_k < \abs{\log_{\eta}\left(  \frac{\eta \LipcostrD}{2\ssizeinit (1-\delta)}   \right) }. 
\end{equation*} 
Hence,  Step 4 requires at most $n_1  \coloneqq  \floor{  \abs{\log_{\eta}\left(  \frac{\eta \LipcostrD}{2\ssizelb(1-\delta)}   \right)}} $ function evaluations.  At the time that the linesearch strategy terminates in this step,  we have the decrease
\begin{equation*}
\cost(u_{\ell(k)})-\cost(u_{k+1}) \geq \frac{\delta}{\ssizek} \|\gmap_{\ssizek}(u_k)\|_\Hs^2 >  \frac{2(1-\delta)\delta}{\eta \LipcostrD}\|\gmap_{\ssizek}(u_k)\|_\Hs^2.
\end{equation*} 
\end{itemize}
 Gathering the two above cases,  we can see that function value decrease per function evaluation is given,  in the worst-case,  by
\begin{equation*}
\cost(u_{\ell(k)})-\cost(u_{k+1}) \geq \decparam  \|\gmap_{\ssizek}(u_k)\|_\Hs^2.   
\end{equation*}
with  $\decparam \coloneqq \min \left\{  \frac{\delta}{\ssizeub} ,   \frac{2(1-\delta)\delta}{ n_1\eta \LipcostrD} \right\}$.  Further, using similar arguments as in the proof of Theorem  \ref{Thm:conv},  we obtain
\begin{equation}
\label{eq:B18}
\begin{split}
&\cost({u_0})-\costbar \geq \cost({u_0})-\cost(u_{k}) \geq \cost{(u_{\nu(0)})}-\cost(u_{\nu(\ceil{\frac{k}{\mmax+1}})})\\ 
& = \sum^{\ceil{\frac{k}{\mmax+1}}}_{i =1}  \cost(u_{\nu(i-1)}) -\cost(u_{\nu(i)}) \geq  \sum^{\ceil{\frac{k}{\mmax+1}}}_{i =1}   \decparam  \|\gmap_{\ssize_{\nu(i)-1}}(u_{\nu(i)-1)})\|_\Hs^2 \\
&\geq \decparam  \ceil{\frac{k}{\mmax+1}} \min_{1 \leq i  \leq  \ceil{\frac{k}{\mmax+1}}}\|\gmap_{\ssize_{\nu(i)-1}}(u_{\nu(i)-1})\|^2_\Hs    \\   
&\geq \left(\frac{\decparam  k}{\mmax+1}\right) \min_{0 \leq  i  \leq  \ceil{\frac{k}{\mmax+1}}(\mmax+1)-1}\|\gmap_{\ssize_{i}}(u_{i})\|^2_\Hs   \\   
&\geq \left(\frac{  \decparam k}{C^{2\mmax}_G(\mmax+1)}\right) \min_{0 \leq  i  \leq  k}\|\gmap_{\ssize_{i}}(u_{i})\|^2_\Hs.
\end{split}
\end{equation}
To find an $\tol$-stationary point, we assume that up to the current iteration $k$ Algorithm \ref{algo:NMLS} has not been terminated, i.e. $\|\gmap_{\ssize_{i}}(u_{i})\|_\Hs > \epsilon$ for all $i =0,\dots,k$.   In this case,  using  \eqref{eq:B18},  we can write 
\begin{equation*}
\begin{split}
&\cost({u_0})-\costbar > \left(\frac{ \decparam k}{C^{2\mmax}_G(\mmax+1)}\right)\tol^2,
\end{split}
\end{equation*}
and,  therefore,  the total number of function evaluations of $\cost$ in Algorithm \ref{algo:NMLS}
is bounded from above by
\begin{equation*}
\maxfeval \leq \floor{ \frac{ C^{2\mmax}_G (\mmax+1)(\cost({u_0})-\costbar)}{\decparam \tol^2} } = \floor{ \frac{\maxfparam (\cost({u_0})-\costbar)}{ \tol^2} }.
\end{equation*}
Thus,  we are finished with the verification  of  \eqref{eq:B19}.
Similarly,  using  \eqref{eq:B15},  it can be easily show that the total number of $\gmap_{\ssizek}(u_k)$ evaluations is bounded by
\begin{equation*}
\maxgeval \leq \floor{ \frac{  C^{2\mmax}_G  \ssizebar (\mmax+1)(\cost({u_0})-\costbar)}{\delta  \tol^2} } = \floor{ \frac{\maxgparam (\cost({u_0})-\costbar)}{ \tol^2} },
\end{equation*}
and,  thus,  the proof is complete. 

\subsection{Proof of Lemma \ref{lem:quadratic}}\label{appendix:quadratic}
First, we show  for every $k\geq 1$ that 
\begin{equation}
\label{eq:B22}
\begin{split}
\cost(u_{k})\leq  \min_{w \in \Hs} \mathcal{Q}_{\ssize_{k-1}}(w, u_{k-1}) + \frac{\LipcostrD}{2\ssizelb \ssize_{k-1}} \|\gmap_{\ssize_{k-1}}(u_{k-1})\|_\Hs^2. 
\end{split}
\end{equation}
Due to \ref{ass:general_3}, the descent lemma, see \cite[Lemma 1.30]{MR3310025}, implies 
\begin{equation*}
\costr( u_{k}) \leq \costr(u_{k-1}) + ( \costrD(u_k),u_k-u_{k-1})_\Hs + \frac{\LipcostrD}{2 \ssize^2_{k-1}} \|\gmap_{\ssize_{k-1}}(u_{k-1})\|_\Hs^2. 
\end{equation*}
Using the definition of $\mathcal{Q}_{\ssize_{k-1}}$  and the fact that $u_k$ is the minimizer of $\mathcal{Q}_{\ssize_{k-1}}(\cdot, u_{k-1})$,  we obtain
\begin{equation*}
\begin{split}
\cost(u_{k})& \leq  \min_{w \in \Hs} \mathcal{Q}_{\ssize_{k-1}}(w, u_{k-1}) + \left(\frac{\LipcostrD}{2 \ssize^2_{k-1}} -\frac{1}{2\ssize_{k-1}}\right)\|\gmap_{\ssize_{k-1}}(u_{k-1})\|_\Hs^2 \\
 &\leq  \min_{w \in \Hs} \mathcal{Q}_{\ssize_{k-1}}(w, u_{k-1}) + \frac{\LipcostrD}{2 \ssize^2_{k-1}} \|\gmap_{\ssize_{k-1}}(u_{k-1})\|_\Hs^2, 
\end{split}
\end{equation*}
and,  thus,   \eqref{eq:B22} follows from the fact that  $\ssizek \geq  \ssizelb$  for all $k \in \N_0$.

Due to the convexity of $\costr$, we can write 
\begin{equation}
\label{eq:B23}
\min_{w \in \Hs} \mathcal{Q}_{\ssize_{k-1}}(w, u_{k-1}) \leq \min_{w \in \Hs} \{\cost(w)+  \frac{\ssize_{k-1}}{2} \|w-u_{k-1} \|^2_\Hs \}. 
\end{equation}
Further,  due to  the convexity of  $\cost$, we obtain for every $w= (1-\lambda)u_{k-1}+\lambda u^* $  with  $\lambda  \in  [0,1] $ and $u^* \in \crit$ that 
\begin{equation*}
\begin{split}
 \min_{w \in \Hs} & \{  \cost(w)+  \frac{\ssize_{k-1}}{2} \|w-u_{k-1} \|^2_\Hs \}  \leq \cost((1-\lambda)u_{k-1}+\lambda u^* ) + \frac{\ssize_{k-1}\lambda^2}{2}  \| u_{k-1} - u^* \|^2_\Hs\\ 
&\leq (1-\lambda) \cost(u_{k-1})+\lambda \cost^*+ \frac{\ssize_{k-1}\lambda^2}{2}  \| u_{k-1} - u^* \|^2_\Hs.
\end{split}
\end{equation*}
Combining with \eqref{eq:B22} and \eqref{eq:B23}, we obtain 
\begin{equation}
\label{eq:B24}
\begin{split}
\cost(u_{k}) & \leq (1-\lambda) \cost(u_{k-1})+\lambda \cost^*+ \frac{\ssize_{k-1}\lambda^2}{2}  \| u_{k-1} - u^* \|^2_\Hs + \frac{\LipcostrD}{2\ssizelb \ssize_{k-1}}\|\gmap_{\ssize_{k-1}}(u_{k-1})\|_\Hs^2 \\ 
& \leq (1-\lambda) \cost(u_{\ell(k-1)})+\lambda \cost^*+ \frac{\ssize_{k-1}\lambda^2}{2}  \| u_{k-1} - u^* \|^2_\Hs + \frac{\LipcostrD}{2 \ssizelb\ssize_{k-1}}\|\gmap_{\ssize_{k-1}}(u_{k-1})\|_\Hs^2.
\end{split}
\end{equation}
Now,  inserting $\nu(k)$ in the place of $k$ in \eqref{eq:B24} and using \eqref{eq:B30},  we obtain
\begin{equation}
\label{eq:B25}
\begin{split}
\cost(u_{\nu(k)})&\leq (1-\lambda) \cost(u_{\nu(k-1)})+\lambda \cost^*+ \frac{\ssize_{\nu(k)-1}\lambda^2}{2}  \| u_{\nu(k)-1} - u^* \|^2_\Hs \\
& + \frac{\LipcostrD}{2\ssizelb \ssize_{\nu(k)-1}}\|\gmap_{\ssize_{\nu(k)-1}}(u_{\nu(k)-1})\|_\Hs^2 .
\end{split}
\end{equation}
Subtracting  $\cost^*$ from both sides of \eqref{eq:B25}, setting $\tilde{C} :=\frac{\LipcostrD}{2\ssizelb}$ and using \rom{3} of Lemma \ref{lem:properties_of_iter} we obtain
\begin{equation}
\label{eq:B43}
\begin{split}
\cost(u_{\nu(k)})-\cost^* &\leq (1-\lambda)\left(\cost(u_{\nu(k-1)})- \cost^* \right) +\frac{\ssizebar \lambda^2}{2}  \|u_{\nu(k)-1}-u^*\|^2_\Hs \\& + \frac{\tilde{C}}{\ssize_{\nu(k)-1}} \|\gmap_{\ssize_{\nu(k)-1}}(u_{\nu(k)-1})\|_\Hs^2.
\end{split}
\end{equation}
Finally,  \eqref{eq:B44} follows from the fact that $u^* \in \crit$ was arbitrary and  
$\crit$ is  non-empty,  closed,  and convex.

Now we deal with the verification of \eqref{eq:B31}.  Let an arbitrary $u^* \in \crit$ be given,  setting $k = 1$ and $\lambda = 1$ in \eqref{eq:B43},  we obtain
\begin{equation}
\label{eq:B32}
\begin{split}
\cost(u_{\nu(1)})-\cost^*  \leq \frac{\ssizebar}{2} \| u_{\nu(1)-1} - u^* \|^2_\Hs + \frac{\tilde{C}}{ \ssize_{\nu(1)-1}} \|\gmap_{\ssize_{\nu(1)-1}}(u_{\nu(1)-1})\|_\Hs^2.
\end{split}
\end{equation}
Using, the fact that  $ 0\leq \nu(1)-1 \leq \mmax$ (see \ref{L1}),   the firm non-expansiveness of the proximal operator,   and the Lipschitz continuity of $\costrD$,  we can write that 
\begin{equation}
\label{eq:B76}
\begin{split}
& \| u_{\nu(1)-1} - u^* \|_\Hs  \leq   \left(1+ \frac{\LipcostrD}{\ssize_{\nu(1)-2}} \right)  \| u_{\nu(1)-2} - u^* \|_\Hs \\& \leq  \left(1+ \frac{\LipcostrD}{\ssizelb}  \right) \| u_{\nu(1)-2} - u^* \|_\Hs \leq \cdots \leq  \left(1+ \frac{\LipcostrD}{\ssizelb}  \right)^{\mmax} \| u_{0} - u^* \|_\Hs.
\end{split}
\end{equation}
Further,  using  \rom{4} of  Lemma \ref{lem:properties_of_iter} successively,  we  obtain that 
\begin{equation}
\label{eq:B77}
 \|\gmap_{\ssize_{\nu(1)-1}}(u_{\nu(1)-1})\|_\Hs \leq C^{\mmax}_G  \|\gmap_{\ssize_{0}}(u_{0})\|_\Hs .
\end{equation}
Combining \eqref{eq:B32}, \eqref{eq:B76},  and \eqref{eq:B77},  we can write 
\begin{equation}
\label{eq:B78}
\begin{split}
\cost(u_{\nu(1)})-\cost^* & \leq \frac{\ssizebar}{2} \left(1+ \frac{\LipcostrD}{\ssizelb}  \right)^{2\mmax} \| u_{0} - u^* \|^2_\Hs + \frac{\tilde{C}C^{2\mmax}_G}{\ssizelb}  \|\gmap_{\ssize_{0}}(u_{0})\|_\Hs^2 \\ & \leq \frac{\ssizebar}{2}  \left(1+ \frac{\LipcostrD}{\ssizelb}  \right)^{2\mmax} \| u_{0} - u^* \|^2_\Hs + \frac{\ssizebar^2 \tilde{C}C^{2\mmax}_G}{ \ssizelb} \| u_{1}-u_0 \|_\Hs^2.
\end{split}
\end{equation}
Further,   the firm non-expansiveness of the proximal operator and the Lipschitz continuity of $\costrD$ again imply that
\begin{equation}
\label{eq:B37}
\begin{split}
\| u_1 - u_0 \|_\Hs & \leq \| u_0 - u^* \|_\Hs + \| u_1 - u^*\|_\Hs\\
& \leq 2 \| u_0 - u^* \|_\Hs + \frac{\LipcostrD}{\ssize_0} \|u_0 - u^* \|_\Hs \leq \big(2 + \frac{\LipcostrD}{\ssizelb}\big) \|u_0 - u^* \|_\Hs.
\end{split}
\end{equation}
Thus,  combining  \eqref{eq:B78} and  \eqref{eq:B37} and setting 
\begin{equation*}
 C_0 \coloneqq  \frac{\ssizebar}{2}  \left(1+ \frac{\LipcostrD}{\ssizelb}  \right)^{2\mmax} + \frac{\ssizebar^2 \tilde{C}C^{2\mmax}_G}{ \ssizelb} \big(2 + \frac{\LipcostrD}{\ssizelb}\big)^2,
 \end{equation*}
 we arrive at 
\begin{equation*}
\cost(u_{\nu(1)})-\cost^* \leq C_0 \|u_0-u^*\|^2_{\Hs}.
\end{equation*}
Hence,  \eqref{eq:B31} follows from the fact that $u^* \in \crit$ is arbitrary. 

\bibliographystyle{spmpsci}
\bibliography{nonmon_FBS_Azm_Ber}

\begin{thebibliography}{10}
\providecommand{\url}[1]{{#1}}
\providecommand{\urlprefix}{URL }
\expandafter\ifx\csname urlstyle\endcsname\relax
  \providecommand{\doi}[1]{DOI~\discretionary{}{}{}#1}\else
  \providecommand{\doi}{DOI~\discretionary{}{}{}\begingroup
  \urlstyle{rm}\Url}\fi

\bibitem{MR4218411}
Ahookhosh, M., Themelis, A., Patrinos, P.: A {B}regman forward-backward
  linesearch algorithm for nonconvex composite optimization: superlinear
  convergence to nonisolated local minima.
\newblock SIAM J. Optim. \textbf{31}(1), 653--685 (2021).
\newblock \doi{10.1137/19M1264783}.
\newblock \urlprefix\url{https://doi.org/10.1137/19M1264783}

\bibitem{iusem_convergence_2003}
Alber, Y.I., Iusem, A.N., Solodov, M.V.: On the projected subgradient method
  for nonsmooth convex optimization in a {Hilbert} space.
\newblock Math. Program. \textbf{81}(1 (A)), 23--35 (1998).
\newblock \doi{10.1007/BF01584842}

\bibitem{MR917455}
Allgower, E.L., B\"ohmer, K.: Application of the mesh independence principle to
  mesh refinement strategies.
\newblock SIAM J. Numer. Anal. \textbf{24}(6), 1335--1351 (1987).
\newblock \doi{10.1137/0724086}.
\newblock \urlprefix\url{http://dx.doi.org/10.1137/0724086}

\bibitem{MR821912}
Allgower, E.L., B\"ohmer, K., Potra, F.A., Rheinboldt, W.C.: A
  mesh-independence principle for operator equations and their discretizations.
\newblock SIAM J. Numer. Anal. \textbf{23}(1), 160--169 (1986).
\newblock \doi{10.1137/0723011}.
\newblock \urlprefix\url{http://dx.doi.org/10.1137/0723011}

\bibitem{fenics}
Aln{\ae}s, M., Blechta, J., Hake, J., Johansson, A., Kehlet, B., Logg, A.,
  Richardson, C., Ring, J., Rognes, M.E., Wells, G.N.: The fenics project
  version 1.5.
\newblock Archive of Numerical Software \textbf{3}(100) (2015)

\bibitem{artacho2008characterization}
Artacho, F.A., Geoffroy, M.H.: Characterization of metric regularity of
  subdifferentials.
\newblock Journal of Convex Analysis \textbf{15}(2), 365 (2008)

\bibitem{zbMATH06285806}
Artacho, F.J.A., Geoffroy, M.H.: Metric subregularity of the convex
  subdifferential in {Banach} spaces.
\newblock J. Nonlinear Convex Anal. \textbf{15}(1), 35--47 (2014).
\newblock
  \urlprefix\url{www.yokohamapublishers.jp/online2/opjnca/vol15/p35.html}

\bibitem{zbMATH05382665}
Attouch, H., Bolte, J.: On the convergence of the proximal algorithm for
  nonsmooth functions involving analytic features.
\newblock Math. Program. \textbf{116}(1-2 (B)), 5--16 (2009).
\newblock \doi{10.1007/s10107-007-0133-5}

\bibitem{zbMATH06145973}
Attouch, H., Bolte, J., Svaiter, B.F.: Convergence of descent methods for
  semi-algebraic and tame problems: proximal algorithms, forward-backward
  splitting, and regularized {Gauss}-{Seidel} methods.
\newblock Math. Program. \textbf{137}(1-2 (A)), 91--129 (2013).
\newblock \doi{10.1007/s10107-011-0484-9}

\bibitem{azmi_analysis_2020}
Azmi, B., Kunisch, K.: Analysis of the {Barzilai}-{Borwein} {Step}-{Sizes} for
  {Problems} in {Hilbert} {Spaces}.
\newblock Journal of Optimization Theory and Applications \textbf{185}(3),
  819--844 (2020).
\newblock \doi{10.1007/s10957-020-01677-y}.
\newblock \urlprefix\url{https://doi.org/10.1007/s10957-020-01677-y}

\bibitem{Behzad2021}
Azmi, B., Kunisch, K.: {On the convergence and mesh-independent property of the
  Barzilai–Borwein method for PDE-constrained optimization}.
\newblock IMA Journal of Numerical Analysis  (2021).
\newblock \doi{10.1093/imanum/drab056}.
\newblock \urlprefix\url{https://doi.org/10.1093/imanum/drab056}.
\newblock Drab056

\bibitem{AzmKunRod}
Azmi, B., Kunisch, K., Rodrigues, S.S.: Saturated feedback stabilizability to
  trajectories for the schlögl parabolic equation.
\newblock IEEE Transactions on Automatic Control pp. 1--14 (2023).
\newblock \doi{10.1109/TAC.2023.3247511}

\bibitem{aze_nonlinear_2014}
Azé, D., Corvellec, J.N.: Nonlinear local error bounds via a change of metric.
\newblock Journal of Fixed Point Theory and Applications \textbf{16}(1-2),
  351--372 (2014).
\newblock \doi{10.1007/s11784-015-0220-9}.
\newblock \urlprefix\url{http://link.springer.com/10.1007/s11784-015-0220-9}

\bibitem{MR967848}
Barzilai, J., Borwein, J.M.: Two-point step size gradient methods.
\newblock IMA J. Numer. Anal. \textbf{8}(1), 141--148 (1988).
\newblock \doi{10.1093/imanum/8.1.141}.
\newblock \urlprefix\url{https://doi.org/10.1093/imanum/8.1.141}

\bibitem{MR3616647}
Bauschke, H.H., Combettes, P.L.: Convex analysis and monotone operator theory
  in {H}ilbert spaces, second edn.
\newblock CMS Books in Mathematics/Ouvrages de Math\'{e}matiques de la SMC.
  Springer, Cham (2017).
\newblock \doi{10.1007/978-3-319-48311-5}.
\newblock \urlprefix\url{https://doi.org/10.1007/978-3-319-48311-5}.
\newblock With a foreword by H\'{e}dy Attouch

\bibitem{B17}
Beck, A.: First-order methods in optimization, \emph{MOS/SIAM Ser. Optim.},
  vol.~25.
\newblock Philadelphia, PA: Society for Industrial {and} Applied Mathematics
  (SIAM); Philadelphia, PA: Mathematical Optimization Society (MOS) (2017).
\newblock \doi{10.1137/1.9781611974997}

\bibitem{zbMATH05618078}
Beck, A., Teboulle, M.: A fast iterative shrinkage-thresholding algorithm for
  linear inverse problems.
\newblock SIAM J. Imaging Sci. \textbf{2}(1), 183--202 (2009).
\newblock \doi{10.1137/080716542}.
\newblock
  \urlprefix\url{semanticscholar.org/paper/bcf48b5e76c7e22335c6820f0de0abe8c5f708b5}

\bibitem{MR4215308}
Bello-Cruz, Y., Li, G., Nghia, T.T.A.: On the linear convergence of
  forward-backward splitting method: {P}art {I}---{C}onvergence analysis.
\newblock J. Optim. Theory Appl. \textbf{188}(2), 378--401 (2021).
\newblock \doi{10.1007/s10957-020-01787-7}.
\newblock \urlprefix\url{https://doi.org/10.1007/s10957-020-01787-7}

\bibitem{MR4430995}
Bello-Cruz, Y., Li, G., Nghia, T.T.A.: Quadratic growth conditions and
  uniqueness of optimal solution to lasso.
\newblock J. Optim. Theory Appl. \textbf{194}(1), 167--190 (2022).
\newblock \doi{10.1007/s10957-022-02013-2}.
\newblock \urlprefix\url{https://doi.org/10.1007/s10957-022-02013-2}

\bibitem{MR3500980}
Bo\c{t}, R.I., Csetnek, E.R., L\'{a}szl\'{o}, S.C.: An inertial
  forward-backward algorithm for the minimization of the sum of two nonconvex
  functions.
\newblock EURO J. Comput. Optim. \textbf{4}(1), 3--25 (2016).
\newblock \doi{10.1007/s13675-015-0045-8}.
\newblock \urlprefix\url{https://doi.org/10.1007/s13675-015-0045-8}

\bibitem{MR3707370}
Bolte, J., Nguyen, T.P., Peypouquet, J., Suter, B.W.: From error bounds to the
  complexity of first-order descent methods for convex functions.
\newblock Math. Program. \textbf{165}(2, Ser. A), 471--507 (2017).
\newblock \doi{10.1007/s10107-016-1091-6}.
\newblock \urlprefix\url{https://doi.org/10.1007/s10107-016-1091-6}

\bibitem{MR3482398}
Bonettini, S., Loris, I., Porta, F., Prato, M.: Variable metric inexact
  line-search-based methods for nonsmooth optimization.
\newblock SIAM J. Optim. \textbf{26}(2), 891--921 (2016).
\newblock \doi{10.1137/15M1019325}.
\newblock \urlprefix\url{https://doi.org/10.1137/15M1019325}

\bibitem{zbMATH06431462}
Cartis, C., Sampaio, P.R., Toint, P.L.: Worst-case evaluation complexity of
  non-monotone gradient-related algorithms for unconstrained optimization.
\newblock Optimization \textbf{64}(5), 1349--1361 (2015).
\newblock \doi{10.1080/02331934.2013.869809}.
\newblock
  \urlprefix\url{citeseerx.ist.psu.edu/viewdoc/summary?doi=10.1.1.726.1663}

\bibitem{C12}
Casas, E.: Second order analysis for bang-bang control problems of pdes.
\newblock SIAM Journal on Control and Optimization \textbf{50}(4), 2355--2372
  (2012).
\newblock \doi{10.1137/120862892}

\bibitem{C17}
Casas, E.: A review on sparse solutions in optimal control of partial
  differential equations.
\newblock SeMA Journal \textbf{74}(3), 319--344 (2017)

\bibitem{MR4162938}
Casas, E., Mateos, M., R\"{o}sch, A.: Analysis of control problems of
  nonmontone semilinear elliptic equations.
\newblock ESAIM Control Optim. Calc. Var. \textbf{26}, Paper No. 80, 21 (2020).
\newblock \doi{10.1051/cocv/2020032}.
\newblock \urlprefix\url{https://doi.org/10.1051/cocv/2020032}

\bibitem{COMBETTES2001115}
Combettes, P.L.: Quasi-fejérian analysis of some optimization algorithms.
\newblock In: D.~Butnariu, Y.~Censor, S.~Reich (eds.) Inherently Parallel
  Algorithms in Feasibility and Optimization and their Applications,
  \emph{Studies in Computational Mathematics}, vol.~8, pp. 115--152. Elsevier
  (2001).
\newblock \doi{https://doi.org/10.1016/S1570-579X(01)80010-0}.
\newblock
  \urlprefix\url{https://www.sciencedirect.com/science/article/pii/S1570579X01800100}

\bibitem{MR2858838}
Combettes, P.L., Pesquet, J.C.: Proximal splitting methods in signal
  processing.
\newblock In: Fixed-point algorithms for inverse problems in science and
  engineering, \emph{Springer Optim. Appl.}, vol.~49, pp. 185--212. Springer,
  New York (2011).
\newblock \doi{10.1007/978-1-4419-9569-8\_10}.
\newblock \urlprefix\url{https://doi.org/10.1007/978-1-4419-9569-8_10}

\bibitem{zbMATH00422112}
Dontchev, A.L., Zolezzi, T.: Well-posed optimization problems, \emph{Lect.
  Notes Math.}, vol. 1543.
\newblock Berlin: Springer-Verlag (1993)

\bibitem{DruLew_Error}
Drusvyatskiy, D., Lewis, A.S.: Error bounds, quadratic growth, and linear
  convergence of proximal methods.
\newblock Math. Oper. Res. \textbf{43}(3), 919--948 (2018).
\newblock \doi{10.1287/moor.2017.0889}.
\newblock \urlprefix\url{https://doi.org/10.1287/moor.2017.0889}

\bibitem{zbMATH06409515}
Drusvyatskiy, D., Mordukhovich, B.S., Nghia, T.T.A.: Second-order growth, tilt
  stability, and metric regularity of the subdifferential.
\newblock J. Convex Anal. \textbf{21}(4), 1165--1192 (2014).
\newblock \urlprefix\url{www.heldermann.de/JCA/JCA21/JCA214/jca21061.htm}

\bibitem{MR3341671}
Frankel, P., Garrigos, G., Peypouquet, J.: Splitting methods with variable
  metric for {K}urdyka-\l ojasiewicz functions and general convergence rates.
\newblock J. Optim. Theory Appl. \textbf{165}(3), 874--900 (2015).
\newblock \doi{10.1007/s10957-014-0642-3}.
\newblock \urlprefix\url{https://doi.org/10.1007/s10957-014-0642-3}

\bibitem{garrigos_thresholding_2020}
Garrigos, G., Rosasco, L., Villa, S.: Thresholding gradient methods in
  {Hilbert} spaces: support identification and linear convergence.
\newblock ESAIM: Control, Optimisation and Calculus of Variations \textbf{26},
  28 (2020).
\newblock \doi{10.1051/cocv/2019011}.
\newblock \urlprefix\url{https://www.esaim-cocv.org/10.1051/cocv/2019011}

\bibitem{garrigos_convergence_2022}
Garrigos, G., Rosasco, L., Villa, S.: Convergence of the forward-backward
  algorithm: beyond the worst-case with the help of geometry.
\newblock Mathematical Programming  (2022).
\newblock \doi{10.1007/s10107-022-01809-4}.
\newblock \urlprefix\url{https://doi.org/10.1007/s10107-022-01809-4}

\bibitem{MR849278}
Grippo, L., Lampariello, F., Lucidi, S.: A nonmonotone line search technique
  for {N}ewton's method.
\newblock SIAM J. Numer. Anal. \textbf{23}(4), 707--716 (1986).
\newblock \doi{10.1137/0723046}.
\newblock \urlprefix\url{https://doi.org/10.1137/0723046}

\bibitem{MR2792408}
Hager, W.W., Phan, D.T., Zhang, H.: Gradient-based methods for sparse recovery.
\newblock SIAM J. Imaging Sci. \textbf{4}(1), 146--165 (2011).
\newblock \doi{10.1137/090775063}.
\newblock \urlprefix\url{https://doi.org/10.1137/090775063}

\bibitem{MR2034880}
He, Y.: Two-level method based on finite element and {C}rank-{N}icolson
  extrapolation for the time-dependent {N}avier-{S}tokes equations.
\newblock SIAM J. Numer. Anal. \textbf{41}(4), 1263--1285 (2003).
\newblock \doi{10.1137/S0036142901385659}.
\newblock \urlprefix\url{https://doi.org/10.1137/S0036142901385659}

\bibitem{MR1202003}
Heinkenschloss, M.: Mesh independence for nonlinear least squares problems with
  norm constraints.
\newblock SIAM J. Optim. \textbf{3}(1), 81--117 (1993).
\newblock \doi{10.1137/0803005}.
\newblock \urlprefix\url{http://dx.doi.org/10.1137/0803005}

\bibitem{MR2085262}
Hinterm\"uller, M., Ulbrich, M.: A mesh-independence result for semismooth
  {N}ewton methods.
\newblock Math. Program. \textbf{101}(1, Ser. B), 151--184 (2004).
\newblock \urlprefix\url{https://doi.org/10.1007/s10107-004-0540-9}

\bibitem{MR1871460}
Hinze, M., Kunisch, K.: Second order methods for optimal control of
  time-dependent fluid flow.
\newblock SIAM J. Control Optim. \textbf{40}(3), 925--946 (2001).
\newblock \doi{10.1137/S0363012999361810}.
\newblock \urlprefix\url{https://doi.org/10.1137/S0363012999361810}

\bibitem{H08}
Hinze, M., Pinnau, R., Ulbrich, M., Ulbrich, S.: Optimization with PDE
  Constraints.
\newblock Mathematical Modelling: Theory and Applications. Springer Netherlands
  (2008).
\newblock \urlprefix\url{https://books.google.de/books?id=PFbqxa2uDS8C}

\bibitem{MR1915930}
H\"{u}ther, B.: Global convergence of algorithms with nonmonotone line search
  strategy in unconstrained optimization.
\newblock Results Math. \textbf{41}(3-4), 320--333 (2002).
\newblock \doi{10.1007/BF03322774}.
\newblock \urlprefix\url{https://doi.org/10.1007/BF03322774}

\bibitem{MR4504980}
Kanzow, C., Mehlitz, P.: Convergence properties of monotone and nonmonotone
  proximal gradient methods revisited.
\newblock J. Optim. Theory Appl. \textbf{195}(2), 624--646 (2022).
\newblock \doi{10.1007/s10957-022-02101-3}.
\newblock \urlprefix\url{https://doi.org/10.1007/s10957-022-02101-3}

\bibitem{MR912453}
Kelley, C.T., Sachs, E.W.: Quasi-{N}ewton methods and unconstrained optimal
  control problems.
\newblock SIAM J. Control Optim. \textbf{25}(6), 1503--1516 (1987).
\newblock \doi{10.1137/0325083}.
\newblock \urlprefix\url{http://dx.doi.org/10.1137/0325083}

\bibitem{MR1049770}
Kelley, C.T., Sachs, E.W.: Approximate quasi-{N}ewton methods.
\newblock Math. Programming \textbf{48}(1, (Ser. B)), 41--70 (1990).
\newblock \urlprefix\url{https://doi.org/10.1007/BF01582251}

\bibitem{KunRod}
Kunisch, K., Rodrigues, S.S.: Global stabilizability to trajectories for the
  schlögl equation in a sobolev norm (2022).
\newblock \doi{10.48550/ARXIV.2212.01888}.
\newblock \urlprefix\url{https://arxiv.org/abs/2212.01888}

\bibitem{MR3429743}
Li, G., Pong, T.K.: Global convergence of splitting methods for nonconvex
  composite optimization.
\newblock SIAM J. Optim. \textbf{25}(4), 2434--2460 (2015).
\newblock \doi{10.1137/140998135}.
\newblock \urlprefix\url{https://doi.org/10.1137/140998135}

\bibitem{zbMATH02176481}
Mordukhovich, B.S.: Variational analysis and generalized differentiation. {I}:
  {Basic} theory. {II}: {Applications}, \emph{Grundlehren Math. Wiss.}, vol.
  330/331.
\newblock Berlin: Springer (2005).
\newblock \doi{10.1007/3-540-31247-1}

\bibitem{necoara_linear_2019}
Necoara, I., Nesterov, Y., Glineur, F.: Linear convergence of first order
  methods for non-strongly convex optimization.
\newblock Mathematical Programming \textbf{175}(1-2), 69--107 (2019).
\newblock \doi{10.1007/s10107-018-1232-1}.
\newblock \urlprefix\url{http://link.springer.com/10.1007/s10107-018-1232-1}

\bibitem{zbMATH06197840}
Nesterov, Y.: Gradient methods for minimizing composite functions.
\newblock Math. Program. \textbf{140}(1 (B)), 125--161 (2013).
\newblock \doi{10.1007/s10107-012-0629-5}

\bibitem{6422363}
O'Donoghue, B., Stathopoulos, G., Boyd, S.: A splitting method for optimal
  control.
\newblock IEEE Transactions on Control Systems Technology \textbf{21}(6),
  2432--2442 (2013).
\newblock \doi{10.1109/TCST.2012.2231960}

\bibitem{parikh2014proximal}
Parikh, N., Boyd, S., et~al.: Proximal algorithms.
\newblock Foundations and trends{\textregistered} in Optimization
  \textbf{1}(3), 127--239 (2014)

\bibitem{MR3310025}
Peypouquet, J.: Convex optimization in normed spaces.
\newblock SpringerBriefs in Optimization. Springer, Cham (2015).
\newblock \doi{10.1007/978-3-319-13710-0}.
\newblock \urlprefix\url{https://doi.org/10.1007/978-3-319-13710-0}.
\newblock Theory, methods and examples, With a foreword by Hedy Attouch

\bibitem{MR3707899}
Salzo, S.: The variable metric forward-backward splitting algorithm under mild
  differentiability assumptions.
\newblock SIAM J. Optim. \textbf{27}(4), 2153--2181 (2017).
\newblock \doi{10.1137/16M1073741}.
\newblock \urlprefix\url{https://doi.org/10.1137/16M1073741}

\bibitem{zbMATH05162727}
Schirotzek, W.: Nonsmooth analysis.
\newblock Universitext. Berlin: Springer (2007).
\newblock \doi{10.1007/978-3-540-71333-3}

\bibitem{MR3845278}
Themelis, A., Stella, L., Patrinos, P.: Forward-backward envelope for the sum
  of two nonconvex functions: further properties and nonmonotone linesearch
  algorithms.
\newblock SIAM J. Optim. \textbf{28}(3), 2274--2303 (2018).
\newblock \doi{10.1137/16M1080240}.
\newblock \urlprefix\url{https://doi.org/10.1137/16M1080240}

\bibitem{scipy}
Virtanen, P., Gommers, R., Oliphant, T.E., Haberland, M., Reddy, T.,
  Cournapeau, D., Burovski, E., Peterson, P., Weckesser, W., Bright, J.,
  et~al.: Scipy 1.0: fundamental algorithms for scientific computing in python.
\newblock Nature methods \textbf{17}(3), 261--272 (2020)

\bibitem{MR1756894}
Volkwein, S.: Mesh-independence for an augmented {L}agrangian-{SQP} method in
  {H}ilbert spaces.
\newblock SIAM J. Control Optim. \textbf{38}(3), 767--785 (2000).
\newblock \doi{10.1137/S0363012998334468}.
\newblock \urlprefix\url{http://dx.doi.org/10.1137/S0363012998334468}

\bibitem{MR2650165}
Wright, S.J., Nowak, R.D., Figueiredo, M.A.T.: Sparse reconstruction by
  separable approximation.
\newblock IEEE Trans. Signal Process. \textbf{57}(7), 2479--2493 (2009).
\newblock \doi{10.1109/TSP.2009.2016892}.
\newblock \urlprefix\url{https://doi.org/10.1109/TSP.2009.2016892}

\end{thebibliography}

\end{document}